\documentclass[11pt]{amsart}

\usepackage{amssymb,amsfonts}
\usepackage{graphicx}
\usepackage[all,arc]{xy}
\usepackage{enumerate}
\usepackage{hyperref}
\usepackage{mathrsfs}
\usepackage{tikz-cd}
\usepackage{mathtools}
\usepackage{nameref}
\usepackage{rotating}
\usepackage[left=1in,top=1in,right=1in,bottom=1in]{geometry}

\usepackage{todonotes}
\usepackage{comment}
\usepackage{float}

\usepackage[dvipsnames]{xcolor}

\usepackage{lmodern}
\usepackage[T1]{fontenc}
\usepackage[utf8]{inputenc}
\usepackage{microtype} %Better font spacing

\usepackage{amsmath,amssymb,amsthm,mathrsfs,latexsym,mathtools,mathdots,booktabs,enumerate,tikz,bm,url,comment}
\usepackage[enableskew,vcentermath]{youngtab}
\usepackage[centertableaux]{ytableau}

\usepackage{tkz-graph}
\usetikzlibrary{arrows,positioning}
\usetikzlibrary{arrows.meta,tikzmark,decorations.markings,fit,calc,shapes,matrix,math}
\usepackage[capitalize, noabbrev]{cleveref}

\usepackage{graphicx}

\tikzstyle{vertex}=[
	minimum size=0.1cm,
    inner sep=0pt,
    outer sep=1pt,
    circle,
    draw=black,
]
\tikzstyle{filled}=[
	vertex,
	fill=black
]
\tikzstyle{unfilled}=[
	vertex,
	fill=white
]
\tikzstyle{blank}=[
	vertex,
	fill=white,
	draw=white
]

\tikzstyle{hourglass}=[
	out=#1*20,
	in={#1*20+180},
	relative,
	looseness=1.5
]

\tikzstyle{dot}=[
  inner sep=0.5pt,
  outer sep=0.5pt,
]

\definecolor{orange}{rgb}{1.0,0.6,0}
\definecolor{purple}{rgb}{.6,0.5,1}
\definecolor{grey}{rgb}{.7,0.7,.7}
\definecolor{green}{rgb}{0,0.7,.2}
\definecolor{maroon}{rgb}{0.7,0,0.8}

\pgfdeclarelayer{bg}
\pgfdeclarelayer{fg}
\pgfsetlayers{bg,main,fg}

\newtheorem{thm}{Theorem}[section]
\newtheorem*{thm*}{Theorem}

\newtheorem{cor}[thm]{Corollary}
\newtheorem{prop}[thm]{Proposition}
\newtheorem{lemma}[thm]{Lemma}

\newtheorem{question}[thm]{Question}
\newtheorem{problem}[thm]{Problem}
\newtheorem*{mainresultA}{Theorem A} 
\newtheorem*{mainresultB}{Theorem B} 
\newtheorem*{prop*}{Proposition.} 
\newtheorem*{mainresultC}{Theorem C} 
\newtheorem*{mainresultD}{Theorem D} 
\theoremstyle{definition}
\newtheorem{defn}[thm]{Definition}

\newtheorem{exmp}[thm]{Example}

\newtheorem{rem}[thm]{Remark}

\theoremstyle{remark}

\newcommand{\ehr}{\mathrm{ehr}} 

\newcommand{\IntEhr}{\mathrm{IntEhr}} 
\newcommand{\calS}{\mathcal{S}}
\newcommand{\calD}{\mathcal{D}} 

\newcommand{\Gr}{Gr} 

\newcommand{\initial}{\mathrm{in}}

\newcommand{\calL}{\mathcal{L}} 

\newcommand{\intehr}{\mathrm{intehr}}

\newcommand{\calK}{\mathcal{K}} 

\newcommand{\F}{\mathcal{F}}
\newcommand{\calT}{\mathcal{T}} 
\renewcommand{\d}{\partial}
\newcommand{\revlex}{\mathrm{revlex}} 
\newcommand{\Ehr}{\mathrm{Ehr}}

\newcommand{\bbP}{\mathbb{P}} 
\newcommand{\bbZ}{\mathbb{Z}} 

\newcommand{\bbC}{\mathbb{C}} 
\newcommand{\bbR}{\mathbb{R}} 

\renewcommand{\O}{\mathcal{O}}

\newcommand{\calF}{\mathcal{F}}

\newcommand{\ch}{\mathrm{ch}}

\newcommand{\om}{\omega}

\newcommand{\Cone}{\mathrm{Cone}} 

\newcommand{\bdy}{\partial} 
\newcommand{\bbN}{\mathbb{N}}

\newcommand{\calO}{\mathcal{O}}

\newcommand{\id}{\text{id}}

\newcommand{\rsl}{\mathrm{rsl}}

\renewcommand{\emptyset}{\varnothing}

\newcommand{\w}{\mathrm{word}}
\newcommand{\col}{\mathrm{col}}
\newcommand{\row}{\mathrm{row}}
\renewcommand{\int}{\mathrm{int} \:}
\newcommand{\rotL}
{\rotatebox[origin=c]{180}{L}}
\newcommand{\rotgeqa}
{\rotatebox[origin=c]{315}{$\geq$}}
\newcommand{\rotgeqb}
{\rotatebox[origin=c]{45}{$\geq$}}
\newcommand{\dem}{\delta}
\makeatletter
\let\c@equation\c@thm
\makeatother
\numberwithin{equation}{section}

\crefname{equation}{}{}

\definecolor{maroon}{rgb}{0.7,0,0.8}
\definecolor{grey}{rgb}{0.7,0.7,0.7}

\title{Fence Complexes and Toric Degenerations of Positroid Varieties}
\author[Chang]{Cameron Chang}
\address[Chang]{Department of Mathematics, University of California, Berkeley, CA, United States of America}
\email{changca@berkeley.edu}
\author[Enugandla]{Pranav Enugandla}
\address[Enugandla]{Department of Mathematics, University of California, Berkeley, CA, United States of America}
\email{shreepranav\_varma@berkeley.edu} 
\author[Hlavinka]{Josephine Hlavinka}
\address[Hlavinka]{Department of Mathematics, University of California, Berkeley, CA, United States of America}
\email{josephhlavinka@berkeley.edu} 

\begin{document}
\begin{abstract}
    We associate to each positroid variety in the Grassmannian $\Gr(k,n)$ a polyhedral complex, which we call a fence complex. Fence complexes consist of unions of faces of the Gelfand–Tsetlin polytope $P_{k,n}$ associated to a fundamental weight $\omega_k$. We show that these fence complexes are homeomorphic to closed balls. Furthermore, they endow the Gelfand–Tsetlin polytope with the structure of a regular CW complex, giving a polyhedral complex presentation of the regular CW complex structure on $\Gr(k,n)_{\geq 0}$. We also show that the Ehrhart polynomial of a fence complex equals the Hilbert polynomial of the associated positroid variety. We prove that under the Sturmfels–Gonciulea–Lakshmibai degeneration of $\Gr(k,n)$ to the toric variety of the Gelfand–Tsetlin polytope, positroid varieties degenerate to the reduced union of toric varieties corresponding to their fence complexes. As an application, we classify when positroid varieties contained inside hook Schubert varieties are arithmetically Gorenstein. We also derive a recursive character formula for cyclic Demazure modules, which we show is equivalent to a formula of Almousa, Gao and Huang.
\end{abstract}
\maketitle
\tableofcontents 

\section{Introduction}\label[section]{Introduction}
The theory of projective toric varieties over $\bbC$ gives a beautiful connection between algebraic geometry and the combinatorics of polytopes. Namely, to any toric variety $X\subset \bbP^N_\bbC$, one can canonically associate (in a way depending only on $X$ and its embedding) a lattice polytope $P\subset \bbZ^n$, where $n=\dim X$. Then, the following remarkable phenomena hold: 
\begin{enumerate}
     \item \textit{(Correspondence of strata)} Since $X$ is toric, we have a torus $T\subset X$ with an action on itself that extends to all of $X$. There are only finitely many orbits of this action, and they form a stratification of $X$, in the sense that the closure of an orbit $O$ is the union of orbits. 
     There is an inclusion preserving bijection between the faces of $P$ and the torus orbit closures of $X$, with a face $F$ contained in another face $F'$ if and only if the corresponding orbit closure $\overline{O}_F$ is contained in $\overline{O}_{F'}$. 
    \item \textit{(Nonnegative geometry)} We say a point $p\in \bbP^N_\bbC$ is nonnegative if $p=[x_0:x_1:\cdots :x_N]$ where the $x_i$ are all nonnegative real numbers. The locus of nonnegative points in $\mathbb{P}^N_\mathbb{C}$ is denoted $(\bbP^N_\bbC)_{\geq 0}$.  Given a variety $X\subset \bbP^N_\bbC$, we define \emph{the nonnegative part of $X$} to be \[X_{\geq 0}:= \{p\in X:  p\in (\bbP_\bbC^N)_{\geq 0} \}. \] 
    
    When $X$ is toric, the moment map with respect to the torus action gives a homeomorphism between $X_{\geq 0}$ and $P$. This homeomorphism maps the nonnegative parts of torus orbit closures to their corresponding faces. In particular, the faces of $P$ endow $X_{\geq 0}$ with the structure of a regular CW complex structure.
    \item \textit{(Lattice point count interpretation)} We define the \emph{Ehrhart polynomial} of $P$ to be $\ehr_P(m) = |mP\cap \bbZ^n|$, where $mP= \{x\in \bbR^n: \frac{1}{m} x \in P\}$.  Then, the Hilbert polynomial $H_X(m)$ of $X$ equals $\ehr_P(m)$. 
    \item \textit{(Ehrhart–Macdonald Reciprocity)} We also have a combinatorial interpretation for the Hilbert polynomial of $X$ evaluated at negative numbers. Namely, $H_X(-m) = (-1)^{\dim P} \intehr_P(m)$, where $\intehr_P(m)$ equals the number of interior lattice points of $mP$. 
\end{enumerate}
 
For more details and definitions regarding these facts see \cite{CLS} or \cite{FultonToricVarieties}. Motivated by the above phenomena, one could ask whether there exist similar polyhedral models for other stratified varieties. A natural candidate is the Grassmannian $\Gr(k,n)$ of $k$-planes in $n$-space with the Schubert stratification. The group of upper triangular invertible matrices $B \subset GL(n)$ acts on $\bbC^n$, and hence $\Gr(k,n)$. There are finitely many orbits under this action, and the closure of a $B$-orbit is the union of $B$-orbits, so this is indeed a stratification. The closures of $B$-orbits are called Schubert varieties and have been widely studied for a multitude of reasons, not least of which the fact that they form a basis for the cohomology ring of $\Gr(k,n)$. 

For the purpose of generalizing (2), however, this stratification does not behave particularly well. Namely, if we take an open Schubert variety $X^\circ\subset \Gr(k,n)$, then, it is not true that $(X^\circ)_{\geq 0}$ is homeomorphic to $\bbR^m$ for some $m$. For instance, in $\bbP^1_\bbC = \Gr(1,2)$ with coordinates $x,y$, there are two Schubert varieties, $X_{(1)}$ and $X_\emptyset$. $X_\emptyset$ is the point $[1:0]$ and $X_{(1)}$ equals the $\bbP^1_\bbC$. 
In particular, the open Schubert variety $X_{(1)}^\circ$ corresponds to setting $x\neq 0$, so $(X_{(1)})_{\geq 0} = \{[1:y] : y\in \bbR_{\geq 0}\}.$ Hence, $(X_{(1)})_{\geq 0}$ is homeomorphic to $\bbR_{\geq 0}$, which is not homeomorphic to $\bbR$.

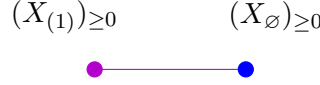
\begin{figure}[H]
    \centering
    \begin{tikzpicture}[scale=2]

    \node[filled, color=maroon, scale=2] at (0,0) {};
    \node[filled, color=blue, scale=2] at (1,0) {};
    
    \begin{pgfonlayer}{bg}
        \node[blank] at (1.2,.35){$(X_{\emptyset})_{\geq 0}$};
        \node[blank] at (-.2,.35){$(X_{(1)})_{\geq 0}$};
        \draw[color=maroon] (0,0) to (1,0) {};
    \end{pgfonlayer}
\end{tikzpicture}

    \caption{The TNN locus of $\mathbb{P}_\mathbb{C}^1$, denoted by $\mathbb{P}^1_{\geq 0}$, is a $1$-simplex. The TNN Schubert cell decomposition of $\mathbb{P}^1_{\geq 0}$ consists of $(X_\emptyset)_{\geq 0}$, which is a point, and $(X_{(1)})_{\geq 0}$, which is a ray and therefore not homeomorphic to any $\mathbb{R}^n$.}\label{fig:tnn-schubert strat}
\end{figure}

In \cite{postnikovpositroids}, Postnikov introduces the \emph{positroid stratification}, which is a refinement of the Schubert stratification. The locally closed strata $(\Pi_u^w)^\circ$, indexed by Bruhat intervals $[u,w]$ with $w$ a $k$-Grassmannian permutation, now have the property that $(\Pi_u^w)_{\geq 0}^\circ$ is homeomorphic to Euclidean space. In \cite{Speyer-Williams-Postnikov}, Postnikov, Speyer and Williams show that the nonnegative parts of these strata endow $\Gr(k,n)_{\geq 0}$ with a CW complex structure. Moreover, Postnikov conjectured in \cite{postnikovpositroids} that the positroid cells gave $\Gr(k,n)_{\geq 0}$ the structure of a regular CW complex, a difficult theorem which was proved by Galashin, Karp and Lam in \cite{Galashin-karp-Lam}.

Given this, one may wonder whether $\Gr(k,n)_{\geq 0}$ can in fact be identified with a polytope, with faces of the polytope corresponding to positroid cells. It is not hard to see this is impossible. For instance, in \cite{Rietsch-Williams}, Rietsch and Williams observe that $\Gr(2,4)_{\geq 0}$ is four dimensional, and has four $3$-dimensional positroid cells, but one can easily see that no four dimensional polytope with only four facets exists. Thus, it may seem unlikely that one can hope for the same nice interplay between polytopes, Ehrhart theory and nonnegative geometry as in the toric case. 

Nonetheless, in this paper, we observe that by replacing "faces of polytopes" with "unions of faces of a polytope", one can formulate analogues of toric phenomena (1)-(4) for positroids. In particular, for each $k\leq n$, we consider a polytope $P_{k,n}$, to be thought of as a polyhedral model of $\Gr(k,n)$. This is the order polytope of the product poset $[k]\times [n-k]$, as well as the Gelfand–Tsetlin polytope in type $A$ associated to the fundamental weight $\omega_k$. To each positroid $\Pi_u^w\subset \Gr(k,n)$, we associate a polyhedral complex $\calF_u^w$, called a \emph{fence complex}, consisting of a union of faces of $P_{k,n}$. We then prove the following analogues of the toric phenomena. 
\begin{mainresultA}[\Cref{regular CW complex structure}] 
Each $\calF_u^w$ is homeomorphic to a closed ball, and the fence complexes $\calF_u^w$ give $P_{k,n}$ the structure of a regular CW complex with cell poset the same as $\Gr(k,n)_{\geq 0}$. In particular, $\Gr(k,n)_{\geq 0}$ and $P_{k,n}$ with the fence complex structure are isomorphic as regular CW-complexes.   
\end{mainresultA} 

The result can be thought of as an analogue to (2). As the positroid stratification poset for $\Gr(k,n)_{\geq 0}$ is the same as the positroid stratification poset for $\Gr(k,n)$, this also gives a good analogue to (1).

We also prove analogues of the Ehrhart theory of toric varieties. The following is an analogue to (3):
\begin{mainresultB}[\Cref{ehrhart poly equals hilbert poly}] 
  The Ehrhart polynomial $\ehr_{\calF_u^w}$  of $\calF_u^w$ equals the Hilbert polynomial $H_{\Pi_u^w}$ of $\Pi_u^w$.
\end{mainresultB}

By understanding the topology of $\calF_u^w$, we prove an analogue of (4): 
\begin{mainresultC}[\Cref{mainresultC}] 
Let $m\in \bbN$. The Hilbert polynomial of $\Pi_u^w$ evaluated at $-m$ is, up to sign, the number of lattice points in the interior of $m\calF_u^w$. In symbols, \[(-1)^{\dim \calF_u^w} H_{\Pi_u^w}(-m) = \intehr_{\calF_u^w}(m).\] 
\end{mainresultC}

A partial explanation for the connection of $\Pi_u^w$ to $\calF_u^w$ is the existence of a toric degeneration of $\Pi_u^w$ . In \cite{Sturmfels-degen}, Sturmfels proves the existence of a $T$-equivariant flat family $\pi: \mathscr{X}\to \bbC$ such that:\\
\begin{enumerate}[(i)] 
    \item $\pi^{-1}(\bbC-\{0\})\to \bbC-\{0\}$ is isomorphic to $\Gr(k,n)\times (\bbC-\{0\})$ as schemes over $\bbC-\{0\}$.
    \item The fiber over $0$ is the toric variety associated to $P_{k,n}$.
\end{enumerate} 
Let $X(\calF_u^w)$ denote the reduced union of torus orbit closures corresponding to the faces of $\calF_u^w$. We then have the following theorem.

\begin{mainresultD}[\Cref{mainresultD}]
The scheme-theoretic closure of $\Pi_u^w\times (\bbC-\{0\})$ in $\mathscr{X}$ is a flat family $\mathscr{X}_u^w \to \bbC$  with generic fiber $\Pi_u^w$ and special fiber $X(\calF_u^w)$. Furthermore, this degeneration is $T$-equivariant with respect to the action of the standard torus on $\Pi_u^w$.  
\end{mainresultD}  

While Theorem B can be thought of as a combinatorial shadow of Theorem D, neither Theorem A nor Theorem C follow from Theorem D. The Ehrhart reciprocity of Theorem C requires crucially that the polyhedral complex $\calF_u^w$ is a ball, while verifying that the fence complexes fit together nicely into a CW complex structure on $P_{k,n}$ requires careful analysis of the combinatorics of their boundaries.

We present two applications of our construction. In one, we characterize completely when the homogeneous coordinate rings of positroid varieties contained in Schubert varieties are Gorenstein. 

\begin{prop*}[\Cref{hook gorenstein}] 
Assume that $u \leq w$, and that the Young diagram of $\pi_k(w)$ is a hook. Then $\Pi_u^w$ is arithmetically Gorenstein if and only if $$a(\pi_k(w)/\pi_k(u)) - |\{s_i\:|\:s_i \in u_J, i \geq k\}| = l(\pi_k(w)/\pi_k(u)) - |\{s_i\:|\:s_i \in u_J, i < k\}|.$$
\end{prop*} 

Here, $a(\pi_k(w)/\pi_k(u))$ equals the size of the arm of $\pi_k(w)/\pi_k(u)$ and $l(\pi_k(w)/\pi_k(u))$ equals the size of the leg of $\pi_k(w)/\pi_k(u)$. This result follows from our interpretation of the Hilbert polynomial of $\Pi_u^w$ as an Ehrhart function. In the case of positroids contained in hook Schuberts this Ehrhart function is that of an order polytope, and hence poset-theoretic tools can be applied to understand the Gorensteinness of $\Pi_u^w$. More generally, we characterize the arithmetic Gorensteinness of positroid varieties $\Pi_u^w$ with $\pi_k(w)/\pi_k(u)$ a border strip purely in terms of fence complexes in \Cref{border strip arith}.

Our second application recovers a recursive formula of Almousa–Gao–Huang in \cite{AGH} for the characters of cyclic Demazure modules. We use the $T$-equivariant nature of our degeneration to turn this formula into a statement about lattice points of $\calF_u^w$, which can then be visualized via the polyhedral geometry of $\calF_u^w$. In particular, we show the following character formula: 
\begin{prop*}[\Cref{character formula}] 
\[ \ch(V_u^w(d\om_k)) = t^{u[k]}\ch(V_u^w((d-1)\om_k)) + \sum_{x\in \calL_k'(u,w)} (-1)^{l(w)-l(x)+1} \ch(V_x^w(d\om_k)).\] 
\end{prop*} 

Here, the set $\calL_k'(u,w)$ consists of all $x$ such that there does not exist $u\lessdot \bar{u} \leq x$ with $\pi_k(\bar{u})= \pi_k(u)$. While this set is not obviously equal to the indexing set in the character formula given by Almousa, Gao and Huang, we show the two sets are equal in an appendix. 

\subsection{Relationship to Previous Work}

Toric degenerations of various subvarieties of the Grassmannian have been studied often in the literature. While this is by no means an exhaustive list, we mention a few results particularly relevant to our work. Sturmfels was the first to consider a toric degeneration of $\Gr(k,n)$ to the toric variety of $P_{k,n}$ in \cite{Sturmfels-degen}. In \cite{Gonciulea-Lakshmibai}, Gonciulea and Lakshmibai calculate the flat limit of all Schubert varieties in $\Gr(k,n)$ under this degeneration. Since Schubert varieties are special instances of positroid varieties, our fence complex construction agrees with the Gonciulea–Lakshmibai degeneration.\\[-5pt]

Gonciulea and Lakshmibai also show that the complete flag variety degenerates to the toric variety associated to a Gelfand–Tsetlin polytope. In \cite{Kogan-Miller}, Kogan and Miller show that in an affine neighborhood of the standard flag, Schubert varieties degenerate to the union of toric varieties associated to faces indexed by pipe dreams. In Kim's thesis (cf. \cite{Kim}), he removes the hypothesis of restricting to an affine neighborhood. Kim also shows that Richardson varieties in the complete flag variety degenerate to the union of toric varieties corresponding to certain faces of the Gelfand–Tsetlin polytope, indexed by pairs of pipe dreams. This can be viewed as a complete flag variety analogue of our Theorem D (and hence of our Theorem B). As positroid varieties can be viewed as projected Richardsons,  it would be very interesting to understand the relationship between Kim's degeneration and ours.\\[-5pt] 

It does not seem like the topology of the Gelfand–Tsetlin polytope faces associated to a Richardson variety in the complete flag variety has been considered before. It would be interesting to prove an analogue of Theorems A and C for Kim's faces in the complete flag variety case. Of course, it would also be interesting to extend these results to all partial flag varieties. Our proof that $\calF_u^w$ is homeomorphic to a ball heavily uses the fact that the polytope $P_{k,n}$ has a triangulation coming from its realization as an order polytope, as well as a result from \cite{KLS0} on the topology of this simplicial complex. It is not clear what the analogous triangulation should be for Gelfand–Tsetlin polytopes corresponding to degenerations of other partial flag varieties.\\[-5pt] 

Degenerations of positroid varieties to subspace arrangements under the Hodge degeneration have also been studied in the literature. In particular, \cite{KLS0} exhibited a Gröbner degeneration of the homogeneous coordinate ring $\bbC[\Pi_u^w]$ to the Stanley–Reisner ring of a certain simplicial complex. This degeneration is compatible with the Hodge degeneration. In \cite{AGH}, Almousa, Gao and Huang give an explicit Gröbner basis for this degeneration. Our work makes heavy use of this Gröbner basis in tandem with a technical result (see \Cref{so we are done})  in order to prove Theorem D.\\[-5pt]

In the same paper, Almousa, Gao and Huang also give a proof of a recursive formula for the character of cyclic Demazure modules, independent of their Gröbner basis construction. Their arguments connect the multigraded Hilbert series of positroid varieties (and hence cyclic Demazure characters) with the multigraded Hilbert series of matrix Schubert varieties. They then use an inductive recurrence on the multigraded Hilbert series of matrix Schubert varieties, coming from a geometric argument, to deduce the corresponding recursion for cyclic Demazure characters. We use our toric degeneration and the polyhedral geometry of $\calF_u^w$ to prove a seemingly different character formula, which makes no reference to matrix Schubert varieties. However, we show our formula is actually the same as theirs, giving a new proof of their formula.\\[-5pt]

The fact that the totally nonnegative Grassmannian with the positroid stratification is a regular CW complex was proved in \cite{Galashin-karp-Lam} by Galashin, Karp and Lam. The proof is rather hard and relies on, for instance, the generalized Poincaré conjecture. Our results do not give a new proof of this theorem. Instead, we simply observe that the $\calF_u^w$ give a regular CW complex structure on $P_{k,n}$, and since $\Gr(k,n)_{\geq 0}$ is also a regular CW complex with the same cell closure poset, the two must be isomorphic as regular CW complexes. It would be very interesting to give a new proof of Galashin, Karp and Lam's theorem using our results. 

\subsection{Acknowledgements} 
We would like to thank Ibrahim Ahmad, Chayapa Darayon, David Eisenbud, Christian Gaetz, Yuhan Jiang, Allen Knutson, Thomas Lam, John Lentfer, Sam Mayo, David Nadler, Lizzie Pratt, Linus Setiabrata, David Speyer, Alan Yan, and Alex Yong for helpful conversations. CC and JH were supported by an NSF GRFP Fellowship.\footnote{This material is based upon work supported by the National Science Foundation Graduate Research Fellowship Program under Grant No. DGE 2146752. Any opinions, findings, and conclusions or recommendations expressed in this material are those of the author(s) and do not necessarily reflect the views of the National Science Foundation.} 
CC and PE would like to thank ICERM for the pleasant working environment during their workshop "Webs in Algebra, Geometry, Topology and Combinatorics." 

%todo{conventions} 

\section{Preliminaries}\label[section]{Recall...}
%Let $G = GL(n)$, and let $B, B^{-}$ denote standard and opposite Borels in $G$, respectively. We write $P$ to denote a parabolic subgroup of $G$, and $P_k$ to denote the $k$-th maximal parabolic of $G$. Therefore, we write $Gr(k,n) := G/P_k$ to denote the Grassmannian variety of $k$-planes in $\mathbb{C}^n$. We write $B_i$ for the $i$-th cyclic shift of the standard Borel. Given $\lambda \leq ((n - k)^k)$ in dominance order we write $X^\circ_\lambda$ and $X_\lambda$ for the corresponding Schubert cell and variety, resp., in $Gr(k,n)$. We also write $X_w^\circ$ and $X_w$ for Schubert cells and varieties, resp., in $Fl(n)$, as well as ${R_u^w}^\circ$ and $R_u^w$ for Richardson cells and varieties.

\subsection{Positroid Varieties and k-Bruhat Combinatorics}
%\begin{defn} 
   % An \textit{open positroid cell} $\Pi^\circ$ is an intersection of cyclically shifted Grassmannian Schubert cells. That is, it is a variety of the form $$\Pi^\circ = \bigcap_i X^\circ_{\lambda, i} \subset Gr(k,n)$$ \noindent where $X^\circ_{\lambda, i}$ is the Schubert cell corresponding to $\lambda \subset ((n-k)^k)$ with respect to the $i$-th cyclic shift $(\langle e_i \rangle, \langle e_i, e_{i+1} \rangle, ..., \langle e_i, ..., e_n, e_1, ..., e_{i - 1}\rangle)$ of the standard flag. The closure $\Pi := \overline{\Pi^\circ}$ of a positroid cell in $Gr(k,n)$ is called a \textit{positroid variety}.
%\end{defn}

%We now justify the name positroid. 

%\begin{defn}
   % A rank $k$ matroid realizable over $\mathbb{C}$ is called a \textit{positroid} if it can be realized by a vector arrangement $\{v_1, ..., v_n\} \subset \mathbb{C}^k$ such that the matrix $[v_1, ..., v_k]$ has all non-negative maximal minors.
%\end{defn}

%\begin{thm}[\cite{KLS1} Theorem 3.1]
    %There is a bijection between positroids of rank $k$ which are realizable by a vector arrangement in $\mathbb{C}^n$ and positroid varieties in $Gr(k,n)$.
%\end{thm}
 %Let $P$ denote a Young subgroup $S_{n_1} \times S_{n_2 - n_1} \times ... \times S_{n_i - n_{i-1}} \times S_{n - n_i} \subset S_n$, and write $\pi_P$ for the parabolic quotient map $S_n \rightarrow S_n^P$. We write $u \lessdot_P w$ and say that \textit{$w$ $P$-covers $u$} if $w$ is a Bruhat cover of $u$ and $\pi_P(u) \neq \pi_P(w)$. The $P$-Bruhat order on $S_n$, denoted by $\leq_P$, is defined to be the transitive closure of the $P$-covering relations.

We briefly define (closed) positroid varieties here, but we recommend that readers unfamiliar with positroid varieties consult the survey \cite{speyer2024richardsonvarietiesprojectedrichardson}. Throughout the paper, all of our varieties are over $\bbC$. We fix $G=GL_n(\bbC)$, $B\subset GL_n(\bbC)$ the subgroup of upper triangular matrices, and $B^-\subset GL_n(\bbC)$ the subgroup of lower triangular matrices. Then, $G/B$ is the complete flag variety, parametrizing all flags \[0=E_0\subsetneq E_1 \subsetneq \cdots \subsetneq E_n = \bbC^n.\] For $u,w \in S_n$ such that $u\leq w$ in Bruhat order, we define the \emph{Richardson variety} $X_u^w\subset G/B$ to be the subvariety  $\overline{BuB/B}\cap \overline{B^-wB/B}$. 

Recall there is a projection map $\pi_k:G/B\to \Gr(k,n)$ defined by mapping a flag $0=E_0\subset E_1\subset \cdots \subset E_n = \bbC^n$ to the $k$-dimensional subspace $E_k$. 

\begin{defn}
    A \emph{positroid variety} $\Pi$ is the image $\pi_k(X)$ of a Richardson variety $X$ under $\pi_k$. 
\end{defn}

Throughout the paper, we write $\Pi_u^w = \pi_k(X_u^w)$. Since $\pi_k$ is proper and Richardson varieties are irreducible, positroid varieties are closed irreducible subvarieties of $\Gr(k,n)$. They form a stratification of $Gr(k,n)$ which refines the Schubert and Richardson stratifications. Positroid varieties are indexed by many different beautiful combinatorial objects, but we will focus on an indexing set that comes naturally from the projected Richardson description. 

 The naive hope would be that positroids are in bijection with the Richardson varieties projecting down to them, and hence in bijection with Bruhat intervals $[u,w]$. This turns out to be far from true: in many cases we have $\Pi_u^w = \Pi_{u'}^{w'}$ for two different Bruhat intervals $[u,w]$ and $[u',w']$. However, if we place conditions on the Bruhat interval, this does become a bijection. 

\begin{defn}
   A permutation $w\in S_n$ is \emph{$k$-Grassmannian} if its only possible descent is at $k$, i.e. we have $w(1)<w(2)<\cdots <w(k)$, and $w(k+1)<w(k+2)<\cdots <w(n)$.  
\end{defn}

\begin{rem}
If we identify the permutations $w$ in $S_n$ such that $w([k])=[k]$  with $S_k\times S_{n-k}$ in the natural way, then the $k$-Grassmannian permutations are precisely the representatives of cosets of $S_k\times S_{n-k}$ in $S_n$ that are minimal with respect to Bruhat order. As this is a useful description, we identify $S_k\times S_{n-k}$ with the subgroup of permutations $w\in S_n$ fixing $[k]$ for the remainder of the paper. 
\end{rem}

One advantage of considering $k$-Grassmannian permutations is the following theorem. 

\begin{thm}[\protect{\cite[Theorem 5.9]{KLS1}, \cite[Theorem 1.29]{speyer2024richardsonvarietiesprojectedrichardson}}]\label[theorem]{projected richardsons}
    For every positroid variety $\Pi$, there exists a unique Bruhat interval $[u,w]$ with $u,w\in S_n$ such that $w$ is $k$-Grassmannian and $\pi_k(X_u^w) = \Pi$. 
\end{thm}

Thus, positroid varieties in $\Gr(k,n)$ are indexed by Bruhat intervals $[u,w]$ where $w$ is $k$-Grassmannian. In general, it will be more convenient to index each positroid by an equivalence class of Bruhat intervals, which we now explain. In the below definition, we abuse notation and let $\pi_k$ denote the parabolic quotient map from $S_n\to {[n]\choose k}$ given by sending a permutation $w\in S_n$ to $\{w(1),\ldots, w(k)\}$. 

\begin{defn} Let $u,w\in S_n$. We say that $w$ $k$-Bruhat covers $u$, denoted $u\lessdot_k w$,  if $w$ covers $u$ in Bruhat order, and $\pi_k(u)\neq \pi_k(w)$. We define $k$-Bruhat order, denoted $\leq_k$, to be the partial order given by the transitive closure of $k$-Bruhat covers. \end{defn}

Observe that if $u\leq_k w$, then $u\leq w$, but the converse is not true.  

\begin{rem}\label[remark]{pik notation}
    We can also view $\pi_k$ as a map from $S_n$ to the set of $k$-Grassmannian permutations, taking a permutation $u$ to its minimal coset representative with respect to $S_k\times S_{n-k}$. This is done by identifying a $k$-Grassmannian permutation $w$ with the $k$-element set $w([k])$. For more details, cf. \cite[Section 2.4]{Bjorner-Brenti}. 

    Observe that $\pi_k$ induces an isomorphism of posets between $S_n^P$ with Bruhat order and $\binom{[n]}{k}$ with its Bruhat (also known as \textit{Gale}) order, defined by $I = \{i_1 < ... < i_k\} \leq J = \{j_1 < ... < j_k\}$ if and only if $i_m \leq j_m$ for all $m \leq k$. Similarly, there is also a bijection $\lambda \leftrightarrow I_\lambda = \{\lambda_{k+1-i}+i:i\in[k]\}$ between the interval $[\emptyset, (n-k)^k]$ of Young's lattice ordered by containment and ${[n]\choose k}$ with Bruhat order.
%%The latter bijection admits a convenient pictorial description. Namely, given a partition $\lambda\in [\emptyset, (n-k)^k]$, the lower right boundary of the Young diagram for $\lambda$ gives a lattice path (of length $n$) from the bottom left to the top right of the $k\times (n-k)$ rectangle. $I_\lambda\in {[n]\choose k}$ is then the set of upward steps.

In light of this, throughout this work we will further abuse notation and, given $w \in S_n$, write $\pi_k(w)$ to refer to the following three objects: the minimal coset representative of $w$ modulo $S_k \times S_{n-k}$, its associated Young diagram, and its associated set. It will be clear from context which we mean. 
\end{rem}

We have the following characterization of $k$-Bruhat order, due to Bergeron and Sottile in \cite{Bergeron-Sottile}. 
\begin{thm}\label[theorem]{thm: bergeron-sottile} 
   Let $u,w\in S_n$. Then, $u\leq_k w$ if and only if the following two conditions hold: 
   \begin{enumerate}
       \item $a\leq k<b$ implies $u(a)\leq w(a)$ and $u(b)\geq w(b)$. 
       \item If $a<b$, $u(a)<u(b)$, and $w(a)>w(b)$, then $a\leq k<b$. 
   \end{enumerate}
\end{thm} 

$k$-Bruhat intervals have the following natural geometric interpretation: they index the Richardsons $X_u^w$ where $\pi_k:X_u^w\to \Pi_u^w$ is birational. 

\begin{prop}[\protect{\cite[Proposition 1.27]{speyer2024richardsonvarietiesprojectedrichardson}}]
    Let $u,w \in S_n$ with $u\leq w$ in Bruhat order. Then, $\pi_k: X_u^w \rightarrow Gr(k,n)$ is birational onto $\Pi_u^w$ if and only if $u \leq_k w$. If $w$ is further assumed to be $k$-Grassmannian, then every other birational representative of $\Pi$ is of the form $\pi(R_{ux}^{wx})$, where $x \in S_k\times S_{n-k}$ and $\ell(ux) = \ell(u) + \ell(x)$.
\end{prop}

In fact, every positroid variety is the birational projection of some Richardson indexed by a $k$-Bruhat interval. 

\begin{prop}[\protect{\cite[Proposition 1.26]{speyer2024richardsonvarietiesprojectedrichardson}}]
    For every positroid variety $\Pi_u^w = \pi_k(X_u^w)$, there exists $u \leq u' \leq_k w' \leq w$ such that $\pi_k(X_u^w) = \pi_k(X_{u'}^{w'})$.
\end{prop}

Observe that if $u\leq_k w$, then since $\pi_k:X_u^w\to \Pi_u^w$ is birational, $\Pi_u^w$ is $\ell(w) - \ell(u)$ dimensional.

One cannot quite say that $\Pi^{w'}_{u'} \subset \Pi^w_u$ if and only if $[u', w'] \subset [u, w]$ (as either $k$-Bruhat or classical Bruhat intervals). 

\begin{exmp}
    Set $k=2$ and $n=3$. Let $u,w \in S_3$ where $u=213$ and $w=231$. Then, $u\leq_2 w$, and $\Pi_u^w$ contains the $T$-fixed point corresponding to the set $\{1,2\}$.  This point equals $\Pi_{123}^{123}$, but $[123,123]$ is not contained in $[213,231]$.  
\end{exmp}

Thus, one should be slightly more careful. We define an equivalence relation on $k$-Bruhat intervals by declaring $[u, w] \sim [x, y]$ if there exists a $z \in S_k \times S_{n-k}$ with length additive factorizations $$uz = x,\:wz = y.$$ Write $\langle u, w\rangle$ for the equivalence class of a $k$-Bruhat interval $[u, w]_k$, and write $\mathcal{Q}(k,n)$ for the poset of all such equivalence classes with the relation $q' \leq q$ if and only if there exist representatives for $q'$ and $q$ such that $[u', w']_k \subset [u, w]_k$. Then we have the following 

\begin{thm}[\protect{\cite[Theorem 5.9]{KLS1}}]\label[theorem]{calQ poset}
    The poset of positroid varieties ordered by inclusion is isomorphic to the poset $(\mathcal{Q}(k,n), \leq)$ defined above: $\Pi^{w'}_{u'} \subset \Pi^w_u$ if and only if $\langle u', w' \rangle \leq \langle u, w \rangle$.
\end{thm}

The poset structure can be described combinatorially as well. The following is due to Rietsch and can be found in \cite[Proposition 3.6]{KLS2}. 
\begin{prop}[Rietsch] 
   Fix $u\leq_k w$. We have $q \leq \langle u,w\rangle$ in $\mathcal{Q}(k,n)$ if and only if there exists some representative $[u',w']\in q$ such that $u\leq u'\leq_k w'\leq w$. 
\end{prop} 
In particular, it follows that if $u'\leq_kw'$ and $u\leq_k w$, with the inclusion of Bruhat intervals, $[u',w']\subset [u,w]$, then $\langle u',w'\rangle \leq \langle u,w\rangle$. This also follows from the interpretation of positroid varieties as projected Richardson varieties. 
 
The above discussion implies that we may index a positroid variety $\Pi$ by an equivalence class of $k$-Bruhat intervals in $\mathcal{Q}(k,n)$. Moreover, one can show that every equivalence class in $\mathcal{Q}(k,n)$ has a unique representative $[u,w]_k$ such that $w$ is $k$-Grassmannian. Moreover, when $w$ is $k$-Grassmannian, we have that $u\leq w$ in Bruhat order implies $u\leq_k w$ in $k$-Bruhat order. Hence, we may also index positroid varieties by permutations $u\leq w$ in Bruhat order, where $w$ is $k$-Grassmannian. 

We now introduce a natural simplicial complex associated to positroid varieties. 

\begin{defn}
    Given a finite poset $Q = (V(Q), \leq)$, we write $\Delta(Q)$ to denote the \textit{order complex} of $Q$, the simplicial complex with vertex set $Q$ and faces $\{v_1,\ldots, v_m\}$ where $v_1 <\cdots <v_m$ is a chain in $Q$. 
\end{defn}
For notational ease, we let $\Delta[u,w]$ denote the order complex $\Delta([u,w])$. 
    
\begin{defn}
    Let $[u, w] \subset S_n$ be an interval in Bruhat order. We write $\pi_k(\Delta[u,w])$ to denote the simplicial complex whose vertex set is $\pi_k([u,w])$ and whose faces are the images of $\Delta[u, w]$ under $\pi_k$. 
    \end{defn}
    
Observe that $\pi_k$ might not preserve the dimension of faces of $\Delta[u,w]$.  For example, if $w_1 < w_2$ but $\pi_k(w_1) = \pi_k(w_2)$, then the edge of $\Delta[u,w]$ corresponding to $w_1 < w_2$ projects to a vertex in $\pi_k(\Delta[u,w])$. In fact, $\pi_k(\Delta[u,w])$ might not even be an order complex! See \cite[Remark 8.2]{KLS0}. 

Let $u\leq_k w$. Then, $\pi_k(\Delta[u,w])$ has a number of nice properties, which we detail below. For definitions of terms we do not define, see \cite[Appendix A2]{Bjorner-Brenti}. 

\begin{prop}[\protect{\cite[Corollary 8.8, Lemma 10.1, Theorem 10.2]{KLS0}}]\label[proposition]{shellable}
    For any interval $[u, w] \subset S_n$, $\pi_k(\Delta[u,w])$ is shellable, subthin, and of pure dimension $\ell(w)-\ell(u)$. Moreover, $\pi_k(\Delta[u,w])$ is homeomorphic to a ball. 
\end{prop}

It actually suffices to consider the $k$-Bruhat interval $[u,w]_k$, and fix a representative in $\mathcal{Q}(k,n)$, as justified by the following proposition. 

\begin{prop}[\protect{\cite[Lemma 8.7, Corollary 8.8]{KLS0}}]\label[proposition]{maximal chains bij}
The complex $\pi_k(\Delta[u,w])$ is independent of choice of representative of $\langle u,w\rangle \in \mathcal{Q}(k,n)$. Furthermore, we have $\pi_k(\Delta[u,w]) = \pi_k(\Delta[u,w]_k)$. For $u \leq_k w$, the map from saturated $k$-Bruhat chains in $[u, w]$ to facets of $\pi_k(\Delta[u,w])$ is bijective.
\end{prop}

\subsection{Ehrhart Theory of Polyhedral Complexes}

We now recall facts from the Ehrhart theory of polyhedral complexes. We refer the reader unfamiliar with these concepts to Beck and Robins' text \cite{beck2007computing} for more details. 

\begin{defn}
    A \textit{polyhedral complex} $\calK\subset \bbR^n$ is a set of polyhedra in $\bbR^n$ such that if $\sigma, \tau \in \calK$, then either $\sigma\cap \tau$ is empty or $\sigma\cap \tau \in \calK$. The \textit{underlying space} of $\calK$ is denoted $|\calK|$ and defined as
    \[ \bigcup_{\sigma \in \calK} \sigma.\]
    \noindent where the union is as subsets of $\mathbb{R}^n$.
\end{defn}

When it is not confusing, we abuse notation and refer to $\calK$ both as the polyhedral complex and the underlying space.

\begin{exmp}
    Let $P$ be a polytope. Then, any set of faces of $P$ closed under intersection is a polyhedral complex. We call a polyhedral complex arising in this way a \textit{polyhedral subcomplex} of $P$.
\end{exmp}

A \textit{face} of a polyhedral complex $\calK$ is a face of any of the polyhedra $\sigma \in \calK$. A polyhedral complex is \textit{equidimensional} if all of its maximal faces have the same dimension. It is a \textit{lattice polyhedral complex} if all vertices of $\calK$ are in $\bbZ^n$. We henceforth assume that all polyhedral complexes we consider are equidimensional lattice polyhedral complexes. The \textit{dimension} of a polyhedral complex $\calK$ is defined to be the dimension of any maximal face of $\calK$.

Recall for a polytope $P$, we define the \emph{Ehrhart polynomial}, $\ehr_P(n)$ to be the unique polynomial satisfying 
\[\ehr_P(m) = \#(mP\cap \bbZ^n) \] 
for all $m\in \bbN$. Here, $mP$ denotes the polytope $P$ scaled up by $m$, i.e. the set $\{mx: x\in P\}$. We may also define the interior Ehrhart polynomial:
\[\intehr_P(m) = \#(\int(mP)\cap \bbZ^n), \] 
where $\int_P$ denotes the set of points in $P$ not lying on a proper face of $P$.  

One of the most miraculous features of the Ehrhart polynomial is the following reciprocity result: 
\begin{thm}[Ehrhart–Macdonald reciprocity]
    We have the following equality of polynomials: \[\ehr_P(n) = (-1)^{\dim P} \intehr_P(-n). \] 
\end{thm}
For a proof of this fact, as well as many nice applications in combinatorics, we recommend \cite{Beck-Sanyal}. For a beautiful realization of this theorem as Serre duality for toric varieties, see \cite{FultonToricVarieties}. 

We can define an Ehrhart polynomial for a polyhedral complex $\calK$. Namely, we define 
\[\ehr_{\calK} (m) = \#(m\calK \cap \bbZ^n).\] 
The definition of $\intehr$ is not as clear for an arbitrary polyhedral complex, but if $|\calK|$ is homeomorphic to a manifold with boundary, we declare 
\[\intehr_\calK(m) := \#(\int (m\calK) \cap \bbZ^n),\] 
where $\int (m\calK)$ denotes the interior of $m|\calK|$ viewing it as a manifold with boundary. In the case where the polyhedral complex has sufficiently nice topology, we have a nice generalization of Ehrhart–Macdonald reciprocity. The following is likely known, but we include the proof for lack of an adequate reference.  

\begin{prop}\label[proposition]{general reciprocity} Let $\calK\subset \bbR^N$ be a polyhedral complex homeomorphic to a ball. Then, $$(-1)^{\dim\calK} \ehr_\calK(-m) = \mathrm{intehr}_\calK(m).$$ 
\begin{proof}
Let $Q$ be the poset of faces of $\calK$ with the empty set and $\calK$ appended as a $0$ and $1$ respectively. For all $F\in Q$, we have 
\[\ehr_F(m) = \sum_{F'\leq F}\intehr_{F'}(m).\] Here we declare $\intehr_{\calK}(m) =\intehr_{\emptyset}(m) = 0$. Möbius inversion then implies that \[0 = \intehr_{\calK}(m) = \sum_{F<\calK} \ehr_F(m)\mu_Q(F,\calK)+\ehr_\calK(m). \] 

Plugging in $-m$ and applying Ehrhart–Macdonald reciprocity, we have 
\[\tag{$*$} -\ehr_\calK(-m) = \sum_{F<\calK} \intehr_F(m)\cdot (-1)^{\dim F} \mu(F,\calK).\] 

Now, observe that $\calK$ is a regular CW-complex, with cells given by faces $F\in Q$ and $F\neq \calK, \emptyset$. $Q$ is simply the poset of its cells with a $\hat{0}$ and $\hat{1}$ adjoined. By  \cite[Proposition 3.8.9]{EC1}, we have that for all $F<\calK$, 
$\mu(F,\calK) = (-1)^{\dim\calK-\dim F +1}$ if $\emptyset<F<\calK$ and $F$ is not on the boundary of $\calK$, and $\mu(F,\calK) = 0$ otherwise. With this topological input, we have that ($*$) becomes 
\[\ehr_\calK(-m) = \sum_{F<\calK} \intehr_F(m)\cdot (-1)^{\dim \calK},\] 
as desired. 
\end{proof}
\end{prop}

\begin{rem}
    By \cite[Equations 3.25 and 3.27]{EC1} and surrounding discussion, the above proof works when $\calK$ is homeomorphic to any manifold with boundary and reduced Euler characteristic $0$.
\end{rem}

It is sometimes useful to break a polyhedral complex into more understandable pieces. 

\begin{defn}
    Let $\calK$ be a $d$-dimensional polyhedral complex embedded in $\mathbb{R}^n$. Then, a \textit{triangulation} of $\calK$ is a finite collection $\mathcal{T}$ of $d$-simplices $\Delta \subset \mathbb{R}^n$ with the following properties:

    \begin{itemize}
        \item $\calK = \bigcup_{\Delta \in \mathcal{T}} \Delta$, and
        \item For every $\Delta_1, \Delta_2 \in \mathcal{T}$, $\Delta_1 \cap \Delta_2$ is a face of both $\Delta_1$ and $\Delta_2$.
    \end{itemize}
\end{defn}

\begin{defn}
    A triangulation $\mathcal{T}$ is called \textit{unimodular} if, for every $\Delta = \text{conv} \{v_0, ..., v_d\} \in \mathcal{T}$, the vectors $\{v_i - v_0\}_{i \in [d]}$ form a lattice basis of span$(\Delta) \cap \mathbb{Z}^n \subset \mathbb{R}^n$.
\end{defn}

A polyhedral complex $\calK$ has an associated \textit{Ehrhart series}, whose definition is extended from the one for polyhedra by the formula $$ \Ehr_\calK(t) = \sum_{m=0}^\infty \ehr_{\mathcal{K}}(m)t^m.$$ \noindent In the case where $\calK$ is a polytope $P$, we may define the \textit{interior Ehrhart series}, given by 
\[\IntEhr_P(t) = \sum_{m=1}^\infty \intehr_P(m)t^m.\] 

We then define \textit{$h^*$-polynomial} of $\mathcal{K}$ to be  the unique polynomial in $t$ such that $$ \Ehr_\calK(t)= \frac{h_\calK^*(t)}{(1 - t)^{d + 1}}.$$ 

\begin{exmp}\label[example]{simplex ehrhart}
    The Ehrhart series of a unimodular $d$-simplex $\Delta$ is $ \Ehr_\Delta(t) = \frac{1}{(1 - t)^{d+1}}$, and so $h^*_\Delta(t) = 1$.
\end{exmp}

A triangulation $\mathcal{T}$ of a polyhedral complex also still admits an $h$-vector, defined so that if $f_k$ denotes the number of $k$-dimensional faces of $\mathcal T$ (with $f_{-1} = 1$), then $$h_{\mathcal T }(t) := \sum_{k = -1}^d f_k t^{k+1}(1 - t)^{d - k}.$$

We then have the following useful proposition.

\begin{prop}\label[proposition]{triangle equals complex h-vector}
    If $\calK$ is a $d$-dimensional polyhedral complex that admits a unimodular triangulation $\calT$, then $$h^*_\calK(t) = h_\calT(t).$$
\end{prop}

\begin{proof}
    The proof follows that of \cite[Theorem 10.3]{beck2007computing}.
    
    Write $\calK = \bigcup_{\Delta \in \calT} \Delta^\circ$ (where the relative interiors of simplices in the disjoint union are taken inside of the polyhedral complex $\calT$). This union is disjoint, so we have $$\Ehr_\calK(t) = 1 + \sum_{\Delta \in \calT} \IntEhr_{\Delta}(t),$$ \noindent which is equal to $1 + \sum_{\Delta \in \calT} (\frac{t}{1 - t})^{dim \Delta + 1}$ by  the computation in \Cref{simplex ehrhart} and Ehrhart–Macdonald reciprocity. Therefore we see that, setting $f_{-1}=1$, $$\frac{h_\calK^*(t)}{(1 - t)^{d + 1}} = \Ehr_\calK(t) = \sum_{k = -1}^d f_k (\frac{t}{1 - t})^{k + 1} = \frac{\sum_{k = -1}^d f_kt^{k+1}(1 - t)^{d - k}}{(1 - t)^{d+1}} = \frac{h_\calT(t)}{(1 - t)^{d+1}}.$$
\end{proof}

%%\begin{defn}
    %%A \textit{semi-toric variety} $X$ is a variety whose irreducible components are toric varieties.
%%\end{defn} 

%%Semi-toric varieties are (roughly) to polyhedral complexes as projective toric varieties are to polytopes. More precisely, we will use the following construction throughout the paper. 
\subsection{Unions of Toric Varieties Associated to Polyhedral Subcomplexes}
 In this paper, we will consider certain unions of toric subvarieties contained in a fixed toric variety. Throughout, given a polytope $P$ we write $X(P)$ to denote the corresponding toric variety. Let $\mathcal{K}$ be a polyhedral subcomplex of a polytope $P$. Then, recall (see \cite[Sections 1.5 and 3.1]{FultonToricVarieties}) that by taking the normal fan and applying the orbit-cone correspondence, these faces correspond to a union of torus orbit closures in $X(P)$. 
 
\begin{defn}\label[definition]{semi-toric from complex}
    We define $X(\mathcal{K})$ to be the subscheme of $X(P)$ given taking the reduced union of the torus orbit closures corresponding to the faces of $\calK$. 
\end{defn}

A choice of polyhedral subcomplex of a polytope $\mathcal{K} \subset P$ induces a choice of ample line bundle $\calL$ on both $X(\mathcal{K})$ and $X(P)$: on the latter by the correspondence between polytopes and ample line bundles on toric varieties, and thus on $X(\mathcal{K})\subset X(P)$ via restriction. Hence, we may, and will, regard both $X(P)$ and $X(\calK)$ as canonically embedded in projective space. We will write $I(\mathcal{K})$ and $I(P)$ for the homogeneous ideals of $X(\mathcal{K})$ and $X(P)$ respectively under this embedding.

\begin{lemma}\label[lemma]{global sections of polyhedral complex} 
   Let $\bbC \langle m\calK\cap \bbZ^n \rangle$ be the free vector space on the lattice points of $m\calK$. Then, \[H^0(X(\mathcal{K}),\calL^{\otimes m}) = \bbC \langle m\calK\cap \bbZ^n \rangle. \]
\begin{proof}
    Recall (cf. \cite[Section 3.4]{FultonToricVarieties}) that $H^0(X(P),\calL^{\otimes m})=\bbC \langle mP\cap \bbZ^n\rangle$. Moreover, the lattice points of $mP$ form a $T$-weight basis for the global sections of $\calL^{\otimes m}$. Next, recall that torus orbit closures of $X(P)$ are compatibly split with respect to a Frobenius splitting on $X(P)$ (see \cite[Proposition 3.2]{Payne}). Moreover, unions and intersections of compatibly split schemes are compatibly split, hence $X(\calK)$ is compatibly split (see \cite[Proposition 1.2.1]{Brion-Kumar}). Since $\calL$ is ample, by \cite[Theorem 1.2.8]{Brion-Kumar} we have the restriction map \[H^0(X(P),\calL^{\otimes m})\to H^0(X(\calK),\calL^{\otimes m})\] is surjective. Moreover, this map is $T$-equivariant, so to calculate $H^0(X(\calK),\calL^{\otimes m})$ it suffices to check which lattice points of $mP$ restrict to nontrivial sections on $X(\calK)$. It follows from the description of sections of $\calL^{\otimes m}$ in \cite[Section 3.4]{FultonToricVarieties} that these sections are precisely the lattice points of $m\calK$. 
\end{proof}
\end{lemma}
The following is an immediate corollary of the above lemma. 
\begin{cor}\label[corollary]{ehr = cohom}
For all $m\in \bbN$ 
    \[\ehr_\calK (m) = \chi(X(\calK), \calL^{\otimes m}).\]
\end{cor}
\begin{proof}
  As in the proof of \Cref{global sections of polyhedral complex}, we have $X(\calK)$ is Frobenius split, so higher cohomology of $\calL$ vanishes by \cite[Theorem 1.2.8]{Brion-Kumar}. The result now follows from \Cref{global sections of polyhedral complex}. 
\end{proof}

\begin{rem}\label[remark]{normal polytope and coordinate ring}
Recall \[\bigoplus_{m=0}^\infty H^0(X(P),\calL^{\otimes m})=\bigoplus_{m=0}^\infty \bbC\langle mP\cap \bbZ^n\rangle \] forms a graded ring, with multiplication given by the addition of lattice points. Then, the above result and the fact that multiplying sections commutes with restriction shows that 
\[\bigoplus_{m=0}^\infty H^0(X(\calK),\calL^{\otimes m})=\bigoplus_{m=0}^\infty \bbC\langle m\calK\cap \bbZ^n\rangle\] 
where multiplication is once again given by addition of lattice points. 

Recall $P$ is a \emph{normal polytope} if for all $m\in \bbN$, every lattice point in $mP$ can be written as a sum of $m$ lattice points in $P$. In this case, $\bigoplus_{m=0}^\infty \bbC\langle mP\cap \bbZ^n\rangle$ is generated in degree $1$, and hence $X(P)$ is projectively normal (see \cite[Theorem 2.4.1]{CLS}). If $P$ is a normal polytope, then one can show that  the ring $\bigoplus_{m=0}^\infty \bbC\langle m\calK\cap \bbZ^n\rangle$ is also generated in degree one, and hence  $\bigoplus_{m=0}^\infty \bbC\langle m\calK\cap \bbZ^n\rangle$ is the homogeneous coordinate ring of $X(\calK)$. 
\end{rem}

\subsection{Order polytopes}
In this paper, we consider polyhedral subcomplexes of a particular polytope, which is an example of an order polytope. 

Let $Q$ be a finite poset, and let $\hat{Q}$ denote $Q\coprod \{\hat{0},\hat{1}\}$ where we declare $\hat{0}$ to be the minimum element and $\hat{1}$ to be the maximum element.  Writing $\mathbb{R}^Q$ for the real space with coordinates $\{x_v\}_{v \in Q}$, we can canonically associate a polytope to $Q$. 
\begin{defn}
The order polytope of $Q$, denoted $\calO(Q)$, is the polytope defined by the inequalities $x_p\leq x_q$ for each $p\leq q$ in $Q$ and $0\leq x_p\leq 1$ for all $p \in Q$. 
\end{defn}
Indicator functions of order ideals of $Q$ correspond to the vertices of $\calO(Q)$. Therefore, given an order ideal $I\subset Q$, we let $v_I$ denote the vertex of $\calO(Q)$ corresponding to $I$.

We now recall a classification of the faces of an order polytope due to Geissinger (see \cite[Theorem 1.2]{Stanley-twoposetpolytopes} for a statement and further references). Let $\mathcal{B}=\{B_1,\ldots, B_k\}$ be a partition of the poset $\hat{Q}$ into a disjoint union. We say that $\mathcal{B}$ is \emph{connected} if each block $B_i\in \mathcal{B}$ is connected as an induced subposet of $Q$. We say that $\mathcal{B}$ is \emph{compatible} if the transitive closure of the following relation $\leq$ on the $B_i$ is a partial order:  $B_i\leq B_j$ if $x\leq y$ for some $x\in B_i, y\in B_j$.  

Given a face $F$ of $\calO(Q)$, we may associate a partition of $\hat{Q}$ as follows: define an equivalence relation $\sim_F$ on $\hat{Q}$, where $p\sim_F q$ if $x_p = x_q$ on all of $F$ (here, we declare $x_{\hat{0}} =0, x_{\hat{1}}=1$. We then define $\mathcal{B}_F$ to be the partition given by the equivalence classes of $\sim_F$. 

\begin{prop}[Geissinger]\label[proposition]{Geissinger} 
    Faces of $\calO(Q)$ are in bijection with connected, compatible partitions of $\hat{Q}$, via $F\mapsto \mathcal{B}_F$ 
\end{prop}

Order polytopes are particularly nice to work with from the perspective of Ehrhart theory because they come naturally equipped with unimodular triangulations. 

\begin{defn}
    Given an order polytope $\O(Q)$ of a finite poset $Q$, the \textit{canonical triangulation} $\mathcal{T}$ of $\O(Q)$ is defined to consist of all simplices of the form \[\Delta_\sigma = \{(x_q)_{q\in Q} \mid 0\leq x_{\sigma(1)}\leq x_{\sigma(2)}\leq \cdots \leq x_{\sigma(|Q|)}\leq 1\}, \] where $\sigma: [|Q|] \to Q$ ranges over all inverses of total extensions of $Q$.
\end{defn}

The canonical triangulation of $\O(Q)$ restricts to the canonical triangulation of any of its faces when considered as an order polytope $\O(Q/\sim_F)$. This means also that the canonical triangulation of $\O(Q)$ induces a unimodular triangulation on any polyhedral subcomplex of $\O(Q)$.

\begin{defn}\label[definition]{canonical triangle}
    The \textit{canonical triangulation} of a polyhedral subcomplex $\mathcal{K}$ of an order polytope $\O(Q)$, also denoted $\mathcal{T}$, is the intersection of the canonical triangulation of $\O(Q)$ with $\mathcal{K}$. 
\end{defn}

By definition, the vertex set of the simplicial complex of a canonical triangulation of a polyhedral subcomplex $\calK\subset \calO(P)$ is contained in the vertex set of $\calK$. 

\begin{rem}\label[remark]{Order polytopes are normal} 
Since order polytopes have a unimodular triangulation, it follows from \cite[Proposition 2.60]{bruns2009polytopes} that each face is a normal polytope. 
\end{rem} 

\section{Fundamentals of Fence Combinatorics}\label[section]{Fundamentals of Fence Combinatorics}

%%Put succinctly, the goal of this work is to construct a polyhedral complex which admits a triangulation that is isomorphic as a simplicial complex to $\pi_k(\Delta([u, w]))$ but which can see more of the geometry of $\Pi^w_u$. We will then utilize this complex to both give a polyhedral model for the TNN Grassmannian that is compatible with its positroidal CW structure, and to construct toric degenerations of positroid varieties. By design, these degenerations will see things about the geometry of $\Pi^u_w$ that the classical Hodge degenerations can't.

This section will largely be devoted to introducing the basics of the relevant polyhedral complexes, dubbed \textit{fence complexes}. Let us fix conventions to do so; we will turn to geometry in later sections. %%It is worth mentioning that in particular, these polyhedral complexes will be Given the geometric context, however, we find it prudent to first recall how $\pi_k(\Delta([u, w]))$ relates to a more classical, non-toric Gr{\"o}bner degeneration of $\Pi_u^w$.

%%We follow \cite{KLS2} Section 7, though restrict to the case where $W/W_P$ is the portion of Young's lattice consisting of Young diagrams fitting inside of a $k \times (n - k)$ rectangle (so $G = GL_n, P = P_k$). We may therefore think of $G/P$ and all of its subvarieties as embedded via the Pl{\"ucker} embedding into $Proj(k[p_I\:|\: I \in \binom{[n]}{k}])$.

%%Choose any total order $\leq$ on $\binom{[n]}{k}$ refining Bruhat order. By abuse of notation, we also use $\leq$ to denote the corresponding revlex term order on monomoials of $k[p_I]$. Given any subvariety $X$ of $Proj(k[p_I\:|\: I \in \binom{[n]}{k})$ satisfying $k[X] = k[p_I]/J$ (where $k[X]$ denotes the \textit{homogeneous} coordinate ring of $X$), we write $$In_\leq k[X] := k[p_I]/(In_\leq J)$$ \noindent where 
%%$In_\leq(J)$ is the initial ideal of $J$ with respect to the given term order.

%%The following theorem then holds. We note that the special case  where $\Pi_u^w = Gr(k,n)$ is due to Hodge.

%%\begin{thm}[\cite{KLS2} Theorem 7.1.]
    %%For any $u \leq w$, $$In_\leq(k[\Pi_u^w]) = SR(\pi_k(\Delta[u,w])$$ \noindent where SR denotes the Stanley–Reisner ring of a simplicial complex.
%%\end{thm}

%%In particular one gets a flat degeneration from $\Pi_u^w$ to $Proj(SR(\pi_k(\Delta[u,w]))$.

\subsection{Combinatorial Conventions} 
We draw all Young diagrams in English notation. Further, all of our diagrams are assumed to fit inside a $k\times (n-k)$ rectangle. We will number our rows from top to bottom and columns from left to right (with starting index $1$), so that the top left cell of the rectangle has coordinates $(1,1)$. The cell with coordinates $(i, j)$ is then the cell in row $i$ and column $j$. For us, an antidiagonal is the set of all cells $(i,j)$ such that $i+j=c$ for some fixed $c$. For instance, $(i+1, j-1), (i,j), (i-1, j+1)$ all lie in the same antidiagonal. 

A skew-shape is called a \textit{border strip} if it contains no $2 \times 2$ sub-shapes. Similarly, a (row and column weakly-increasing) tableau $T$ is a \textit{border strip tableau} if it contains no $2 \times 2$ sub-shapes which are filled with the same number. 

Given a permutation $u\in S_n$, we define $u_J = \pi_k(u)^{-1}u$. This notation matches that used for a general parabolic quotient, see \cite[Section 2.4]{Bjorner-Brenti}. For permutations, we let $\leq$ denote Bruhat order, and $\leq_k$ denote $k$-Bruhat order. Finally, if two words $\w_1, \w_2$  are equal to each other as permutations, we write $\w_1\equiv \w_2$. In this case, we say that the two words are \emph{move equivalent} (since by Matsumoto's theorem, we can get from one to the other using braid moves and commutation moves). 

\subsection{Fence Diagrams}
Fix a skew shape $\lambda/\mu$ throughout. We will use the term \textit{segment-decorated skew-shape (SDSS) of shape $\lambda/\mu$} to mean a collection of line segments drawn between the centers of adjacent cells in the skew-shape $\lambda/\mu$. If there is a segment between two adjacent cells, we say they are \emph{connected}. The following definition is the critical one of the paper.

\begin{defn}\label[definition]{fence diagrams} A \emph{fence diagram $F$ of shape $\lambda/\mu$} is a segment-decorated skew-shape of shape $\lambda/\mu$ such that
    \begin{equation}\label{eq:fence}
    \tag{$F$}
        \text{the cell } (i,j+1) \text{ is connected to both } (i, j) \text{ and } (i+1, j+1) \text{ if and only if }(i+1, j) \text{ is.}
    \end{equation}

    We say that two cells of an SDSS $F$ of shape $\lambda/\mu$ are in the same \emph{component} or (if $F$ is a fence diagram) \textit{fence component} of $F$ if they are connected by a sequence of segments of $F$. If $F$ is a fence diagram we will refer to its segments as \textit{fences}. Further, we will denote the number of components of $F$ by $|F|$. For a fence diagram $F$ of shape $\lambda/\mu$ without fences (a.k.a a skew-shape) this recovers the notation $|F| = |\lambda/\mu|.$

    Finally, we say that a triple of cells with coordinates $(i, j), (i, j+1), (i+1, j )$ or $(i + 1, j + 1), (i, j+1), (i+1, j)$ is a \textit{bad configuration} of $F$ if one of the following two conditions hold:
    \begin{enumerate}[(a)] 
        \item There are segments between $(i+1,j)$ and $(i,j),(i+1,j+1)$, but there is either no segment between $(i,j+1)$ and $(i,j)$ or no segment between $(i,j+1)$ and $(i+1,j+1)$. 
        \item There are segments between $(i,j+1)$ and $(i,j), (i+1,j+1)$, but there is either no segment between $(i+1,j)$ and $(i,j)$ or no segment between $(i+1,j)$ and $(i+1,j+1)$. 
    \end{enumerate}
\end{defn}
 See \Cref{fig:bad_configurations} for a depiction of this rule. The left configuration illustrates condition (a), and the middle configuration illustrates condition (b). 

\begin{figure}[H]

    \centering
    \def\borderInnerSep{2pt}
\begin{tikzpicture}

\node[filled, color=maroon] at (0.5,0.5) {};
\node[filled, color=maroon] at (1.5,0.5) {};
\node[filled, color=maroon] at (0.5,1.5) {};

\draw (0,0)-- (0,2)--(2,2)--(2,1)--(2,0)--(0,0);
\draw (1,0)--(1,2);
\draw (2,0)--(2,2);
\draw (0,1)--(2,1);
\draw [thick, color=maroon](.5,.5)--(.5,1.5);
\draw [thick, color=maroon](.5,.5)--(1.5,.5);

\begin{scope}[shift={(5,0)}]

\node[filled, color=maroon] at (1.5,1.5) {};
\node[filled, color=maroon] at (1.5,0.5) {};
\node[filled, color=maroon] at (0.5,1.5) {};

\draw (0,0)-- (0,2)--(2,2)--(2,1)--(2,0)--(0,0);
\draw (1,0)--(1,2);
\draw (2,0)--(2,2);
\draw (0,1)--(2,1);
\draw [thick, color=maroon](1.5,1.5)--(1.5,.5);
\draw [thick, color=maroon](1.5,1.5)--(.5,1.5);
\end{scope}

\begin{scope}[shift={(10,0)}]

\node[filled, color=maroon] at (0.5,0.5) {};
\node[filled, color=maroon] at (1.5,0.5) {};
\node[filled, color=maroon] at (0.5,1.5) {};
\node[filled, color=maroon] at (1.5,1.5) {};

\def\nodescl{0.7}
\draw (0,0)-- (0,2)--(2,2)--(2,1)--(2,0)--(0,0);
\draw (0,1) -- (2,1);
\draw (1,0) -- (1, 2);
\draw[thick, color=maroon] (.5,.5)--(1.5,.5)--(1.5,1.5);
\draw[thick, color=maroon](.5,.5) -- (.5,1.5)--(1.5,1.5);
\end{scope}
\end{tikzpicture}
\caption{An SDSS $F$ is a fence diagram if and only if every time $F$ contains either of the first two arrangements of segments (possibly with more segments present), it actually contains the last one. We refer to these segment configurations as the/an  [L], [\rotL], and box, respectively.}
\label{fig:bad_configurations}
\end{figure}
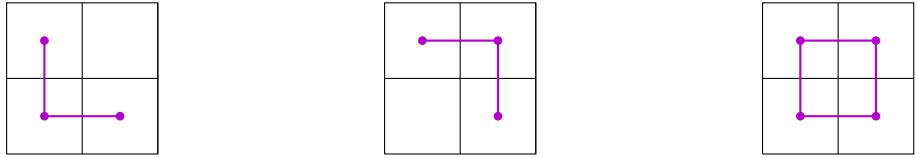
An SDSS is a fence diagram if and only if it contains no bad configurations. 

\begin{rem}\label[remark]{rem:tableau_to_fence}
Recall a \emph{reverse plane partition} of shape $\lambda/\mu$ is any filling of the skew shape with nonnegative integers such that it is weakly increasing from left to right in rows and top to bottom in columns. Given any reverse plane partition $T$ of shape $\lambda/\mu$, we can canonically associate a fence diagram $F$ for $\lambda/\mu$ by drawing fences between adjacent cells with equal entries. 
\end{rem}

 \begin{rem}
    We will later prove, in \Cref{prop:fence_faces}, that fence diagrams are in bijection with the faces of a certain order polytope. The segments will then correspond to setting certain coordinates equal to each other, and the skew shape will correspond to setting certain coordinates equal to $0$ or $1$. Then, fence diagrams will correspond to sets of equality conditions that are actually realized by a face of the polytope. 
\end{rem}

To any positroid variety $\Pi_u^w$, we now associate a set of fence diagrams. To do this we need an intermediate definition, of words associated to any fence diagram.
 
\begin{defn}
\label[definition]{row column words} 
    Let $F$ be a fence diagram of shape $\lambda/\mu$. For $1\le j\le n-k$, let
    \(\col_F^{(j)} = s_{i_1}\cdots s_{i_p},\)
    where $1\le i_1<\cdots<i_p\le k-1$ are the indices such that $(k-i_t+1, j)$ and $(k-i_t, j)$ have a fence between them in $F$. Then
    \[\col_F := \col_F^{(1)}\cdots \col_F^{(n-k)}.\]

    Similarly, for $1\le i\le k$, let
    \(\row_F^{(i)} = s_{j_1}\cdots s_{j_q},\)
    where $n-k-1\ge j_1>\cdots > j_q\ge 1$ are the indices such that $(i,j_t)$ and $(i, j_t+1)$ have a fence between them in $F$. Then
    \[\row_F := \row_F^{(1)}\cdots \row_F^{(k)}.\]

    Finally, we define the word of $F$ to be $$\text{word}_F := (\col_F)(\row_F\{k\}),$$ where the notation $w\{k\}$ means, if $w = s_{i_1}\cdots s_{i_j}$, the word $w\{k\} = s_{i_1 + k} \cdots s_{i_j + k}$. See \Cref{fig:words} for an illustration of this rule.
    
    We say that $F$ is \emph{reduced} if $\col_F$ and $\row_F$ are both reduced words for elements of $S_k$ and $S_{n-k}$ respectively (equivalently, if $\text{word}_F$ is a reduced word for an element of $S_n$).
\end{defn}

In other words, to get $\col_F$, read the columns of $F$ from left to right, and then bottom to top in each column. To get $\row_F$, read the rows of $F$ from top to bottom, and then right to left in each row. 

\begin{rem}
    In the body we will usually write $\text{word}_F = \col_F \row_F\{k\}$, sans parentheses. We stress the parentheses in the definition to eliminate any chance of confusion: only the row word is shifted by $k$ when defining $\text{word}_F$.
\end{rem}

\begin{figure}[H]
    \centering
    \begin{tikzpicture}[scale=1.5]

    \node[filled, color=maroon] at (0,0) {};
    \node[filled, color=maroon] at (1,0) {};
    \node[filled, color=maroon] at (2,0) {};
    \node[filled, color=maroon] at (0,-1) {};
    \node[filled, color=maroon] at (0,-2) {};

    \node[filled, color=maroon] at (1,-1) {};
    \node[filled, color=maroon] at (1,-2) {};
    \node[filled, color=maroon] at (2,-1) {};

    \node[blank] at (0.7, -0.15) {\textcolor{maroon}{$s_1$}};
    \node[blank] at (1.7, -0.15) {\textcolor{maroon}{$s_2$}};
    \node[blank] at (0.15, -0.65) {\textcolor{maroon}{$s_2$}};
    \node[blank] at (0.15, -1.65) {\textcolor{maroon}{$s_1$}};
    \node[blank] at (1.15, -1.65) {\textcolor{maroon}{$s_1$}};
    \node[blank] at (1.7, -1.15) {\textcolor{maroon}{$s_2$}};
    
    \draw[] (-0.5,0.5) to (2.5,0.5) {};
    \draw[] (-0.5,-0.5) to (2.5,-0.5) {};
    \draw[] (-0.5,-1.5) to (2.5,-1.5) {};
    \draw[] (-0.5,-2.5) to (2.5,-2.5) {};

    \draw[] (-.5,0.5) to (-0.5,-2.5) {};
    \draw[] (.5,0.5) to (.5,-2.5) {};
    \draw[] (1.5,0.5) to (1.5,-2.5) {};
    \draw[] (2.5,0.5) to (2.5,-2.5) {};

    \draw[ color=maroon] (0,0) to (2,0) {};
    \draw[ color=maroon] (0,0) to (0,-2) {};
    
    \draw[ color=maroon] (1,-1) to (2,-1) {};
    \draw[ color=maroon] (1,-1) to (1,-2) {};
\end{tikzpicture}

    \caption{A reduced fence diagram $F$ with $\col_F = (s_1s_2)(s_1)()$ and $\row_F = (s_2s_1)(s_2)()$.}
    \label{fig:words}
\end{figure}
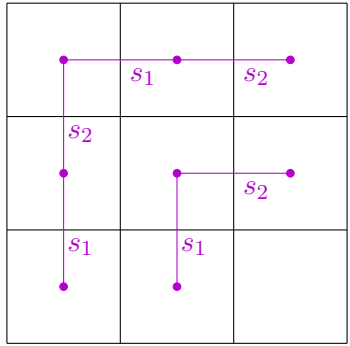

\begin{rem}\label[remark]{transpose word comparison}
    Given a fence diagram $F$, let $F^\top$ denote the fence diagram whose underlying skew-shape is the transpose of that of $F$ and whose fences go between the images of the cells of $F$ which are connected by fences. Then, one sees that \[\row_F = w_0 \col_{F^\top}w_0,\] and similarly, 
    \[ \col_F = w_0\row_{F^\top} w_0.\] 
\end{rem}

\begin{defn}\label[definition]{fence facet indexing set}
    Let $u \leq_k w$ be $k$-Bruhat comparable elements of $S_n$, such that $\pi_k(w) = \lambda$ and $\pi_k(u) = \mu$. Then, we define $\mathfrak{F}_u^w$ to be the set of all \emph{reduced} fence diagrams $F$ with underlying skew-shape $\lambda/ \mu$ such that $\text{word}_F = u_Jw_J^{-1}$.
\end{defn}

\begin{rem}\label[remark]{non-k-grass fence complex}
        One sees that $\mathfrak{F}_u^w$ only depends on the class $\langle u, w\rangle \in \mathcal{Q}(k,n)$. This is because if $\langle u, w\rangle = \langle x, y\rangle$ with $y$ $k$-Grassmannian, then by \cite[Lemma 2.4]{KLS2} we have $z = w_J^{-1}$ in the length-additive factorization $uz = x$ and $wz = y$. Therefore $\pi_k(u) = \pi_k(x)$, $\pi_k(w) = \pi_k(y)$ and $u_Jw_J^{-1} = x_J$ by definition. 
\end{rem}

\begin{defn}
    A \emph{standard fence tableau} $T$ for $\mathfrak F_u^w$ is a filling of some $F\in \mathfrak F_u^w$ with entries in $\{1,\dots, 
    |F|\}$, weakly increasing along rows and down columns, such that two cells contain the same entry if and only if they belong to the same fence component. $F$ is called the \emph{underlying fence diagram} of $T$.
\end{defn}
\begin{exmp}
If $w$ is $k$-Grassmannian, and $u$ is the identity permutation, then $u_J = w_J = \id$ and hence $\mathfrak{F}_u^w$ consists only of the empty fence diagram of shape $\lambda$. Then, standard fence tableaux are precisely the standard Young tableaux of shape $\lambda$. 
\end{exmp}

\begin{rem}
    It is not immediately clear that for $F\in \mathfrak F_u^w$, a standard fence tableau with underlying fence diagram $F$ exists. However, this will turn out to be the case (see \Cref{cor:standard_fence_tableaux}).
\end{rem}

Given a fence diagram $F$, a standard fence tableau $T$ with underlying fence diagram $F$ gives a total order on the set of fence components. We give an alternate way to understand $\w_F$ in terms of this order on the components of $F$.

Say the components of $F$ are $A_1,\dots, A_{|F|}$, where $A_i$ is the component containing all the $i$'s in $T$. Considering each $A_i$ as its own fence diagram with fences induced from $F$ gives $\w_i := \w_{A_i}$.

\begin{lemma}\label[lemma]{lem:word_fence_component}
    In the situation described above, 
    \[\w_F\equiv \w_1\w_2\cdots \w_{|F|}.\]
\end{lemma}
\begin{proof}
    By induction on $|F|$. Let $G$ be the fence diagram obtained by removing $A_1$ from $F$. By the induction hypothesis, we need only show that
    \[\w_F\equiv \w_1\w_G.\]
    It suffices to show this for the row word and column word of $F$. Let us work with $\row_F$ for simplicity.

    For $1\le i\le k$, define $\row_F^{(\le i)} := \row_F^{(1)}\cdots \row_F^{(i)}$, and similarly for $\row_G^{(\le i)}$ and $\row_{A_1}^{(\le i)}$. We shall induct on $i$ to show that
    \[\row_F^{(\le i)} \equiv \row_{A_1}^{(\le i)}\cdot \row_G^{(\le i)}\]
    When $i = 1$ this is clear. Suppose the claim is true for $i-1$, and we wish to prove this for $i$. If the $1$'s in row $i$ of $F$ are from columns $a$ to $b+1$ for $a\le b$, then
    \[\row_F^{(i)} = \row_G^{(i)}\row_{A_1}^{(i)}\equiv \row_{A_1}^{(i)}\row_G^{(i)}.\]
    Moreover, all the cells in $G$ in rows $\le i-1$ must be in columns $>b+1$. This shows that $\row_G^{(\le i-1)}$ only contains $s_c$ with $c>b+1$, and hence $\row_{A_1}^{(i)}$ commutes with $\row_G^{(\le i-1)}$. This shows that
    \begin{align*}
        \row_F^{(\le i)} &= \row_F^{(\le i-1)}\row_F^{(i)}\\
        &\equiv \row_{A_1}^{(\le i-1)}\row_G^{(\le i-1)}\row_{A_1}^{(i)}\row_G^{(i)}\\
        &\equiv \row_{A_1}^{(\le i-1)}\row_{A_1}^{(i)}\row_G^{(\le i-1)}\row_G^{(i)}\\
        &= \row_{A_1}^{(\le i)}\row_G^{(\le i)}.
    \end{align*}

   It follows that $\row_F \equiv \row_{A_1}\row_G$. Transposing and using \Cref{transpose word comparison} gives the result for column words, and so in general.
\end{proof}

\begin{defn}\label[definition]{border strip SDSS}
    A fence diagram is \textit{border strip} if its segments do not contain a box.
\end{defn}

\begin{lemma}\label[lemma]{lem:no_box}
    If $F$ is a fence diagram such that either $\row_F$ or $\col_F$ is reduced, then it is a border strip fence diagram.
\end{lemma}

\begin{proof}
    First suppose that $\col_F$ is reduced, and assume for contradiction that $F$ contains a box. Fix the smallest $i$ for which there exists some $j$ with $(i, j), (i, j+1), (i+1, j), (i+1, j+1)$ all in the same component of $F$. Then, if $(i, j)$ is connected to $(i-1, j)$, \Cref{eq:fence} would imply that $(i-1, j)$ is connected to $(i, j), (i-1, j+1)$ and $(i, j+1)$, contradicting the minimality of $i$. Hence $\col_F^{j}$ does not contain $s_{k - i + 1}$.

    Let $t$ be maximal such that $(i, j+1)$ is connected to $(i+t, j+1)$ by vertical segments. Repeated application of \Cref{eq:fence} implies that $(i, j)$ is connected to $(i+t, j)$ by vertical segments. Now $\col_F^{(j)}$ contains the string $s_{k - i - t + 1}\cdots s_{k - i}$, but not $s_{k - i + 1}$. Moreover, $\col_F^{(j+1)}$ contains the same string $s_{k - i - t + 1}\cdots s_{k - i}$.
    
    Since $\col_F^{j}$ contains no factor of $s_{k - i + 1}$ and $\col_F^{j+1}$ contains no factor of $s_{k - i - t}$, this means we may apply commutation moves in $\col_F$ to obtain a substring of the form $s_{k - i - t + 1}\cdots s_{k - i}s_{k - i - t + 1}\cdots s_{k - i}$. But an elementary check in $S_{i+t}$ shows that this element is not reduced, contradicting the fact that $\col_F$ is reduced. The proof for $\row_F$ reduced then follows from \Cref{transpose word comparison} and the fact that being border strip is unchanged under transposition.
\end{proof}

Therefore we see that the box is a forbidden configuration for reduced fence diagrams.

\begin{cor}\label[corollary]{reduced BST}
    Any standard fence tableau is a border strip tableau.
\end{cor}

\begin{rem}\label[remark]{rem:deleting fences from reduced fence diagrams}
    Let $F$ be a reduced fence diagram, and let $F'$ be an SDSS obtained by deleting some fences in $F$. \Cref{lem:no_box} implies that $F$ contains no boxes, which means that no such fence deletion can produce bad configurations. Thus $F'$ is also a (not necessarily reduced) fence diagram.
\end{rem}

\begin{comment}

Let us put the correspondence between fence diagrams and faces of the order polytope somewhat more scientifically, and in particular in terms of polyhedral geometry. The following is essentially a reinterpretation of the previous lemma in light of the preceding ones.

\begin{lemma}\label{notions agree}
    Let $F$ SDSS of shape $\lambda/\mu$, with components $C_1, ..., C_k$. Assume further that the coordinates on the order polytope $P_{\lambda/\mu}$ corresponding to the boxes in $C_i$ are $x_{i, 1}, ..., x_{i, j}$. Then, $F$ is a fence diagram if and only if there is a face of $\mathcal{O}(P_{\lambda/\mu})$ on whose interior we have $x_{i, 1} = ... = x_{i, j }$ for each components $C_i$ and otherwise $x_{a, b} < x_{a', b'}$ if and only if $(a, b) \leq (a', b')$ in $P_{\lambda, \mu}$, and any face of $\O(P_{\lambda/\mu})$ defines such a fence diagram. This correspondence is bijective.
\end{lemma}

\end{comment}

\subsection{Fence Tableaux, Saturated $k$-Bruhat Chains and $\pi_k(\Delta([u, w]))$}
To see that fence diagrams are related to $\pi_k(\Delta[u,w])$, recall from \Cref{shellable} and \Cref{maximal chains bij} that $\pi_k(\Delta([u, w]))$ is a pure simplicial complex with facets indexed by saturated $k$-Bruhat chains from $u$ to $w$. If $C =  \{u = v_1 < v_2 < \cdots < v_i < \cdots < v_m = w\}$ is a saturated chain in Bruhat order then $\pi_k(C)$ gives a chain in $[\pi_k(u), \pi_k(w)]$, which in turn produces a tableau $T_C$ of shape $\lambda/\mu$ where $\lambda = \pi_k(w), \mu = \pi_k(u)$ by setting $T_C$ to have underlying skew-shape $\lambda / \mu$ and an $i$ in every box contained in $\pi_k(v_i) / \pi_k(v_{i-1})$. The equalities in $T_C$ then give us a fence diagram $F_C$ as in \Cref{rem:tableau_to_fence}.

The assignments sending $C$ to $T_C$ or $F_C$ prove to be very well-behaved. We introduce several properties of fence diagrams and standard fence tableaux which will allow us to analyze it now. Our first lemma follows from the definitions, see \Cref{pik notation}.

%%To start, recall the order preserving bijection $\lambda \leftrightarrow I_\lambda = \{\lambda_{k+1-i}+i:i\in[k]\}$ between the interval $[\emptyset, (n-k)^k]$ of Young's lattice ordered by containment and ${[n]\choose k}$ with Gale order.

%%This bijection admits a convenient pictorial description. Namely, given a partition $\lambda\in [\emptyset, (n-k)^k]$, the lower right boundary of the Young diagram for $\lambda$ gives a lattice path (of length $n$) from the bottom left to the top right of the $k\times (n-k)$ rectangle. $I_\lambda\in {[n]\choose k}$ is then the set of upward steps. This description immediately leads to the following lemma.

\begin{lemma}\label[lemma]{lem:transpose}
    For $\lambda\in [\emptyset, (n-k)^k]$, the transpose partition $\lambda^\top\in [\emptyset, k^{n-k}]$ corresponds to the $n-k$ element set
    \[I_{\lambda^{\top}} = n+1 - I_\lambda^c = (n+1-I_\lambda)^c\in {[n]\choose n-k},\]
    where $A^c := [n]\setminus A$ and $n+1-A:= \{n+1-a\ |\  a\in A\}$ for any subset $A\subset [n]$.
\end{lemma}

\begin{lemma}\label[lemma]{lem:border_strip}
    Let $\mu\subset \lambda$ in the interval $[\emptyset, (n-k)^k]$ of Young's lattice. Then $|I_\lambda\setminus I_\mu|=1$ if and only if $\lambda/\mu$ is a border strip.

\end{lemma}

\begin{proof}
    Let $I_\lambda = \{a_1<\cdots< a_k\},$ and $I_\mu = \{c_1< \cdots< c_k\}$.

    Let $i$ be the smallest index such that $c_i\ne a_i$, and $j$ the largest index such that $c_j\ne a_j$. Since $I_\mu < I_\lambda$ in Bruhat order, this implies $c_i<a_i$. Then $|I_\lambda\setminus I_\mu| = 1$ is equivalent to 
    \begin{align*} c_\ell &= a_\ell, \hspace{6mm} &\ell\notin [i, j],\\
    c_\ell &< a_\ell, \hspace{6mm} &\ell=i,\\
    c_\ell &= a_{\ell-1}, \hspace{6mm} &\ell\in [i+1, j].\end{align*} 
    Recalling that $a_\ell = \lambda_{k-\ell+1}+\ell$ and $c_\ell = \mu_{k-\ell+1}+\ell$ for all $\ell\in [k]$, the above set of conditions is equivalent to 
    \begin{align*}\mu_{k-\ell+1} &= \lambda_{k-\ell+1}, \hspace{6mm} &\ell\notin [i, j]\\
    \mu_{k-\ell+1} &\le \lambda_{k-\ell+1}, \hspace{6mm} &\ell=i\\
    \mu_{k-\ell+1} &= \lambda_{k-\ell+2}-1, \hspace{6mm} &\ell\in [i+1, j].\end{align*}

    These are equivalent to $\lambda/\mu$ being a border strip going from row $k+1-j$ to row $k+1-i$, with $a_i-c_i = \lambda_{k-i+1}-\mu_{k-i+1}$ cells in row $k+1-i$.      
\end{proof}

\begin{lemma}\label[lemma]{lem:border_strip_endpoints}
  In the situation of \Cref{lem:border_strip}, let $I_\lambda = \{a_1<\cdots < a_k\}$, $I_\lambda^c = \{b_1<\cdots<b_{n-k}\}$,  and $I_{\mu} = I_\lambda\setminus\{a_j\}\cup b_{j'}$. If $i\le j$ and $i'\ge j'$ are the unique indices such that $a_{i-1}<b_{j'}<a_i$ and $b_{i'}<a_j<b_{i'+1}$, then $\lambda/\mu$ has endpoints $(k+1-i, j')$ and $(k+1-j, i')$.
\end{lemma}

\begin{proof}
    Write $I_\mu = \{c_1< \cdots< c_k\}$. Then, we have $c_i = b_{j'}$, so that $c_i = b_{j'}< a_i$. Since $c_t=a_t$ for all $t<i$, it follows that $i$ is the smallest index such that $c_i\neq a_i$. Observe that since $a_j>b_{j'}$, we have $j$ is the largest index such that $a_j \neq c_j$.  Thus, the final sentence of the proof of \Cref{lem:border_strip} shows that $\lambda/\mu$ goes from row $k+1-j$ to row $k+1-i$.

    Moreover, $I_\mu^c = I_\lambda^c\setminus \{b_{j'}\}\cup \{a_j\}$, so $I_{\mu^\top} = I_{\lambda^{\top}}\setminus \{n+1-b_{j'}\}\cup \{n+1-a_j\}$ by \Cref{lem:transpose}. If $i'\in [n-k]$ is such that $b_{i'}<a_j<b_{i'+1}$, then $n+1-b_{i'+1}< n+1-a_j < n+1-b_{i'}$. Hence, $n+1-b_{j'}$ is the $(n-k-j'+1)^{th}$ largest element of $I_{\lambda^\top}$, and $n+1-a_j$ is the $(n-k-i'+1)^{th}$ largest element in $I_{\mu^\top}$. Applying the arguments in the previous paragraph to $\lambda^\top/\mu^\top$ shows that $\lambda^\top/\mu^\top$ goes from row $(n-k)+1-(n-k-i'+1)=i'$ to row $(n-k)+1-(n-k-j'+1)=j'$. Transposing shows that $\lambda/\mu$ goes from column $i'$ to column $j'$.
\end{proof} 

Let $u=u_0\lessdot_k \cdots \lessdot_k u_m = w$ be a saturated chain in $k$-Bruhat order. By projecting this chain to ${[n]\choose k}$ and considering it as a sequence of partitions, \Cref{lem:border_strip} implies that we get a border strip tableau $T$. Conversely, given a border strip tableau $T$, we can construct a chain in $k$-Bruhat order, $w = u_m>\cdots>u_0$ such that $\pi_k(u_0) = I_\mu$, as follows. One may think of $T$ as a sequence $\mu = \mu_0\subset \cdots \subset \mu_m = \lambda$ of partitions, and hence as a chain $I_\mu = I_0<\cdots <I_m=I_\lambda$ in Bruhat order. \Cref{lem:border_strip} implies that $|I_i\setminus I_{i-1}| = 1$ for $i\in [m]$. Now, starting from $u_m = w$, we inductively define $u_i\le w$ with $\pi_k(u_i) = I_i$ as follows: if $I_{i} = I_{i+1}\setminus \{a\}\cup \{b\}$ with $b<a$, then 
\[u_i := (a\ b)\cdot u_{i+1}.\]
This constructs the descending Bruhat chain $w = u_m>\cdots >u_0$, with $\pi_k(u_0) = I_\mu$.

We now show that under these correspondences, saturated $k$-Bruhat chains correspond precisely to standard fence tableaux for $\mathfrak F_u^w$. More specifically, we show the following proposition. 

\begin{prop}\label[proposition]{prop: reduced fence faces know}
    Fix a representative interval $[u,w]$ for $\langle u, w\rangle\in \mathcal{Q}(k,n)$. The assignment sending a saturated $k$-Bruhat chain $C =\{w = u_m\gtrdot_k \cdots \gtrdot_k u_0 = u\}$ to the corresponding border strip tableau $T_C$ induces a bijection between saturated $k$-Bruhat chains from $u$ to $w$ and standard fence tableaux with underlying fence diagram $F_C \in \mathfrak F_u^w$. 
\end{prop}

In particular, once one knows that every $F\in \mathfrak F_u^w$ has a standard fence tableau filling, this proposition implies that one can think about the reduced fence diagrams of $\calF_u^w$ as corresponding to all possible equality patterns given by projecting saturated $k$-Bruhat chains from $u$ to $w$ as in \Cref{rem:tableau_to_fence}. We first introduce a couple of preparatory lemmas.

\begin{lemma}\label[lemma]{lem:k-reflection}
    Let $v\in S_n$, and $a>b\in [n]$ such that $v^{-1}(a)\le k$ and $v^{-1}(b)> k$. If $u = (a\ b)\cdot v = v\cdot \big(v^{-1}(a)\ \ v^{-1}(b)\big)$, then $v>u$ is a $k$-Bruhat cover if and only if $v(l)\notin [b,a]$ for all $v^{-1}(a)<l<v^{-1}(b)$.
\end{lemma}

\begin{proof}
    That $v > u$ is a Bruhat cover if and only if the condition holds follows from counting the number of inversions in $u$ and $v$. Then, if $v\gtrdot u$ is a Bruhat cover, it will automatically be a $k$-Bruhat cover since $\pi_k(v)\neq \pi_k(u)$ by construction. 
\end{proof}

\begin{lemma}\label[lemma]{lem:cycle}
    Let $\underline v$ be a reduced word for $v\in S_m$. For $1\le i\le j\le m$,     
    \[(s_is_{i+1}\cdots s_{j-1})\underline v \text{ is reduced }\iff v(\ell)\notin [i,j] \text{ for }\ell>v^{-1}(j),\]
    \[(s_{j-1}\cdots s_{i+1}s_i)\underline v \text{ is reduced }\iff v(\ell)\notin [i,j] \text{ for }\ell<v^{-1}(i).\]
\end{lemma}
\begin{proof}
    We sketch the proof of the first equivalence, and the second follows by similar arguments.

    Write $v' = (s_i\cdots s_{j-1})v$, so that
    \[v'(\ell) = \begin{cases}
        v(\ell)+1 &\text{if }v(\ell)\in[i, j-1],\\
        v(i) & \text{if }v(\ell) = j,\\
        v(\ell) &\text{if }v(\ell)\notin [i, j].
    \end{cases}\]
    
    Using this description, $v'$ has $j-i$ inversions more than $v$ if and only if $v(\ell)\notin [i,j]$ for $\ell>v^{-1}(j)$. Since $\ell(s_i\cdots s_{j-1}) = j-i$, this is equivalent to $(s_i\cdots s_{j-1})\underline v$ being reduced, as required.
\end{proof}

\begin{proof}[Proof of \Cref{prop: reduced fence faces know}]

%First, we show that we can reduce to the case where $w$ is $k$-Grassmannian. By Remark \ref{non-k-grass fence complex}, $\mathfrak F_u^w$ is independent of the representative chosen. Moreover, \ref{maximal chains bij} implies that saturated $k$-Bruhat chains correspond to facets of $\pi_k(\Delta[u,w])$, which by \cite[Lemma 8.7]{KLS0} is independent of the representative chosen. Hence, we may pick the unique representative $[u,w]$ of $\langle u, w\rangle$ such that $w$ is $k$-Grassmannian. 

We induct on $r = l(w)-l(u)$. The base case $r = 0$ is clear. Assume now that the assignment $C\mapsto T_C$ is a bijection for all $[u',w']$ with $l(w')-l(u')<r$.  

First, we show that every standard fence tableau in $\mathfrak F_u^w$ comes from a saturated $k$-Bruhat chain from $u$ to $w$. Let $T$ be a standard fence tableau of shape $F$ for some $F\in \mathfrak F_u^w$. By \Cref{reduced BST}, $T$ is a border strip tableau. We will show that $T$ corresponds to a saturated $k$-Bruhat chain from $u$ to $w$. The cells of $T$ containing $1$'s form a border strip, say with endpoints $(k-i+1, j')$ and $(k-j+1, i')$, where $i\le j$ and $i'\ge j'$. Deleting the cells containing $1$'s (and decrementing all entries) gives a smaller border strip tableau $T'$, with corresponding descending $k$-Bruhat chain $C'= \{w = u_m>_k\cdots >_ku_1 = v\}$ (by induction). Let $F'$ be the fence diagram associated to $T'$. We claim that $w=u_m>_k\cdots >_k u_1=v>u$ is our desired $k$-Bruhat chain.

Write $v[1, k]=\pi_k(v) = \{a_1<\dots<a_k\}$ and $v[k+1, n] = \pi_k(v)^c = \{b_1<\dots< b_{n-k}\}$. In one line notation, $v$ has the form 
\[v = a_{\sigma(1)}\cdots a_{\sigma(k)}b_{\psi(1)}\cdots b_{\psi(n-k)}\]
for some $\sigma\in S_k, \psi\in S_{n-k}$. This implies that $v_J = \sigma \cdot (\psi\{k\})$. Moreover, by \Cref{lem:border_strip_endpoints} and the definition of $i, j, i', j'$, it follows that 
\[\pi_k(u) = \pi_k(v)\setminus \{a_j\}\cup \{b_{j'}\},\ \ \pi_k(u)^c = \pi_k(v)^c \setminus \{b_{j'}\}\cup \{a_j\},\]
with $a_{i-1}<b_{j'}< a_i$ and $b_{i'}<a_j<b_{i'+1}$. Finally,

\[u = v\cdot \big(\sigma^{-1}(j)\ \ \psi^{-1}(j')+k\big) = (a_j\ b_{j'})\cdot v.\]
\vspace{-2mm} 

 Thus, as $C'$ is a saturated $k$-Bruhat chain, it suffices now to show that $v>u$ is a $k$-Bruhat cover. By \Cref{lem:k-reflection}, $v>u$ is a Bruhat cover if and only if $v(\ell)\notin [b_{j'},a_j]$ for $\sigma^{-1}(j) = v^{-1}(a_j)<\ell < v^{-1}(b_{j'}) = \psi^{-1}(j)+k$. The fact that $v(\ell) = \begin{cases}
    a_{\sigma(\ell)} & \text{if }\ell\le k\\
    b_{\psi(\ell-k)} & \text{if }\ell>k
\end{cases}$ along with the inequalities $a_{i-1}<b_{j'}< a_i,\ b_{i'}<a_j<b_{i'+1}$ shows that this is equivalent to the pair of conditions
\begin{equation}\label{eq:1}
    \sigma(\ell) \notin [i, j]\hspace{5mm}\text{for }\ell>\sigma^{-1}(j), \text{ and}   
\end{equation}
\begin{equation}\label{eq:2}
    \psi(\ell) \notin [j', i'] \hspace{5mm}\text{for }\ell<\psi^{-1}(j').
\end{equation}
Note that the contribution of the border strip of $T$ with $1$'s to $\col_F$ is $s_is_{i+1}\cdots s_{j-1}$, and to $\row_F$ is $s_{i'-1}\cdots s_{j'+1}s_{j'}$. Furthermore, $\col_F$ and $\row_F$ are reduced, by assumption. Therefore, by \Cref{lem:word_fence_component}, we have 
\[\col_F  \equiv (s_i\cdots s_{j-1})\col_{F'},\ \row_F\equiv (s_{i'-1}\cdots s_{j'})\row_{F'}.\]
Now, \Cref{lem:cycle}  implies that  $C$ is a saturated $k$-Bruhat chain and $T_C = T$.

Next, we show that if $C$ is a saturated $k$-Bruhat chain, then the associated border strip tableau $T_C$ is a standard fence tableau. Let $F$ denote the fence diagram associated to $T_C$ viewing it as a reverse plane partition, as in \Cref{rem:tableau_to_fence}. Let $F'$ be $F$, but with the fence component corresponding to all of the $1$s removed.
As before, since the cells of $T_C$ containing $1$'s form a border strip, denote its endpoints $(k-i+1, j')$ and $(k-j+1, i')$, where $i\le j$ and $i'\ge j'$.

We now show that $F$ is reduced. Let $\lambda = \pi_k(u_1)$ and $\mu=\pi_k(u_0)$. Note that once again the contribution of the border strip of $T$ with $1$'s to $\col_F$ is $s_is_{i+1}\cdots s_{j-1}$, and to $\row_F$ is $s_{i'-1}\cdots s_{j'+1}s_{j'}$. Using commutation braid moves, it follows that 
\[\col_F \equiv (s_i\cdots s_{j-1})\col_{F'},\ \row_F\equiv (s_{i'-1}\cdots s_{j'})\row_{F'}.\]
 Applying \Cref{lem:cycle} along with these equivalences, $\col_F$ is reduced if and only if $\col_{F'}$ is reduced and \Cref{eq:1} holds. Similarly, $\row_F$ is reduced if and only if $\row_{F'}$ is reduced and \Cref{eq:2} holds. Therefore, $\w_F$ is reduced if and only if $\w_{F'}$ is reduced and both equations \Cref{eq:1} and \Cref{eq:2} hold. By the induction hypothesis, $\w_F$ is reduced if and only if $C$ is a saturated $k$-Bruhat chain.

We are left to prove that when this $\w_F$ is reduced, it is in fact a word for $u_J$. Recall that $u = (a_j\ b_{j'})\cdot v$. Since $a_{i-1}<b_{j'}< a_i$, replacing $a_j$ by $b_{j'}$ has the effect of left multiplying $v_J$ by the cycle $(i\ \ i+1\
 \cdots \ j) = s_i\cdots s_{j-1}$. Similarly since $b_{i'}< a_j < b_{i'+1}$, replacing $b_{j'}$ by $a_j$ has the effect of left multiplying $v_J = \sigma$ by the cycle $(i'+k\ \ \ i'+k-1\ \cdots \ j'+k) = s_{i'+k-1}\cdots s_{j'+k}$. Therefore,
 \begin{align*}
     u_J &= (s_i\cdots s_{j-1})(s_{i'+k-1}\cdots s_{j'+k})v_J \\
     &= (s_i\cdots s_{j-1})\sigma\cdot  (s_{i'+k-1}\cdots s_{j'+k})\psi\{k\}\\
     &= (s_i\cdots s_{j-1})\col_{F'}\cdot (s_{i'+k-1}\cdots s_{j'+k})\row_{F'}\{k\}\\
     &= \col_F\row_F\{k\} = \w_F,
 \end{align*}
where the second equality follows from $v_J = \sigma\cdot \psi\{k\}$ with $\sigma\in S_k$ commuting with factors in $S_{n-k}$, and the second to last equality follows from the commutation move equivalences $\col_F \equiv (s_i\cdots s_{j-1})\col_{F'},\ \row_F\equiv (s_{i'-1}\cdots s_{j'})\row_{F'}.$

\end{proof}

\subsection{Faces of the Gelfand–Tsetlin Polytope}
Let $P_{\lambda/\mu}$ be the poset of cells in $\lambda/\mu$, where we declare $a\le b$ if and only if $\row(a)\le \row(b)$ and $\col(a)\le \col(b)$. In other words, cells increase as we go down and to the right. Let $\widehat{P}_{\lambda/\mu} := P_{\lambda/\mu}\coprod
\{\hat 0, \hat 1\}$ be the poset $P_{\lambda/\mu}$ with a smallest and largest element added.

When $\lambda$ is a $k \times (n - k)$ rectangle and $\mu = \emptyset$, we define $P_{k,n} := P_{\lambda/\mu}$. In this case, $P_{k,n}$ is simply the poset product of $[k]\times [n-k]$, where $[k]$ and $[n-k]$ have their natural total orders. Vertices of $P_{k,n}$ correspond to Young diagrams that fit inside a $k\times n-k$ rectangle. In this case, the order polytope $\O(P_{k,n})$ can be identified with the \textit{Gelfand–Tsetlin polytope} corresponding to the fundamental weight $\omega_k$, i.e. the polytope whose vertices are Gelfand–Tsetlin patterns with $k$ $1$'s at the top and $(n - k)$ $0$'s (see \Cref{fig:gt polytopes}). In this subsection we establish correspondences between its faces and fence diagrams. 

Throughout this paper, we will freely identify $\mathcal{O}(P_{\lambda/\mu})$ with the corresponding face of $\O(P_{k,n})$.

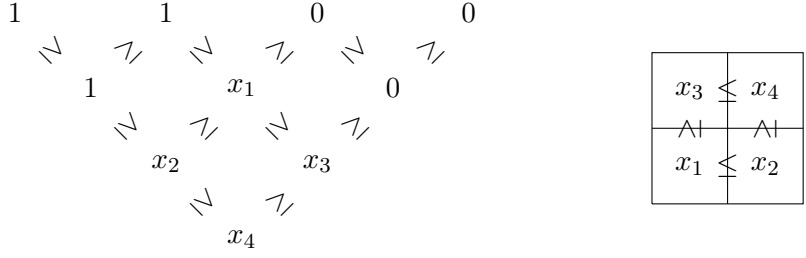
\begin{figure}[H]

    \centering
    \def\borderInnerSep{2pt}

$\begin{array}{ccccc}
\vcenter{\hbox{
\begin{tikzpicture}

    \node[blank] at (0, 0) {1};
    \node[blank] at (2,0) {1};
    \node[blank] at (4,0) {0};
    \node[blank] at (6,0) {0};
    \node[blank] at (1,-1) {1};
    \node[blank] at (3,- 1) {$x_1$};
    \node[blank] at (5,-1) {0};
    \node[blank] at (2,-2) {$x_2$};
    \node[blank] at (4,-2) {$x_3$};
    \node[blank] at (3,-3) {$x_4$};
    
    \node[blank] at (.5, -.5) {$\rotgeqa$};
    \node[blank] at (2.5, -.5) {$\rotgeqa$};
    \node[blank] at (4.5, -.5) {$\rotgeqa$};
    \node[blank] at (1.5, -1.5) {$\rotgeqa$};
    \node[blank] at (3.5, -1.5) {$\rotgeqa$};
    \node[blank] at (2.5, -2.5) {$\rotgeqa$};

    \node[blank] at (1.5, -.5) {$\rotgeqb$};
    \node[blank] at (3.5, -.5) {$\rotgeqb$};
    \node[blank] at (5.5, -.5) {$\rotgeqb$};
    \node[blank] at (2.5, -1.5) {$\rotgeqb$};
    \node[blank] at (4.5, -1.5) {$\rotgeqb$};
    \node[blank] at (3.5, -2.5) {$\rotgeqb$};

\end{tikzpicture}
}}
\hspace{2cm}
\vcenter{\hbox{
\begin{tikzpicture}
    \draw (0,0) -- (2,0) -- (2,2) -- (0,2) -- (0,0);
    \draw (1,0) -- (1,2);
    \draw (0,1) -- (2,1);

    \node at (0.5,0.5) {$x_1$};
    \node at (1.5,0.5) {$x_2$};
    \node at (0.5,1.5) {$x_3$};
    \node at (1.5,1.5) {$x_4$};
    \node at (1, .5) {$\leq$};
    \node at (1, 1.5) {$\leq$};
    \node at (.5, 1) {\rotatebox[origin=c]{90}{$\geq$}};
    \node at (1.5, 1)
    {\rotatebox[origin=c]{90}{$\geq$}};
    
\end{tikzpicture}
}}
\end{array}$
\caption{The polytope $\mathcal{O}(P_{2,4})$ is the subset of $\mathbb{R}^4_{x_1, x_2, x_3, x_4}$ consisting of all fillings of either of the above diagrams with $x_i$'s that satisfy the depicted inequalities. The diagram on the left gives the GT polytope, and fillings of it consisting of integers are called \textit{Gelfand–Tsetlin patterns}. The diagram on the right gives $\mathcal{O}(P_{2,4})$ as we defined it. The identification between the GT polytope and $\mathcal{O}(P_{2,4})$ proceeds by rotating the picture on the right clockwise.}
\label{fig:gt polytopes}
\end{figure}

\begin{defn}\label[definition]{fence face shape}
    A face $F$ of $\mathcal{O}(P_{k,n})$ is said to be of shape $\lambda/\mu$ if it is contained in $\mathcal{O}(P_{\lambda/\mu})$ but is not contained in the order polytope of any smaller skew-shape. Equivalently, $F$ is of shape $\lambda/\mu$ if, when thought of as a set of real-valued fillings of $P_{k,n}$, the entries of any such filling outside of a $\lambda/\mu$ skew-shape are required to be $0$ or $1$ and all others may vary.
\end{defn}

The goal of this subsection is to show that fence diagrams of shape $\lambda / \mu$ correspond to faces of the order polytope $\O(P_{\lambda/\mu})$.

\begin{prop}\label[proposition]{prop:fence_faces}
    There is a bijective correspondence
\vspace{0mm}
    \begin{equation}\label{eq:gt_faces}
        \left\{\text{Faces of }\O(P_{\lambda/\mu})\right\}\longleftrightarrow\left\{\text{Fence diagrams with shape contained in }\lambda/\mu\right\}
    \end{equation}

    More explicitly, given a face of $\O(P_{\lambda/\mu})$, pick an interior point. This gives some filling of $\lambda/\mu$ with entries in $[0,1]$. Deleting all cells containing $0$'s and $1$'s, and drawing fences between adjacent cells with equal entries gives a fence diagram contained in $\lambda/\mu$.\\[-5pt]

    In the other direction, let $F$ be a fence diagram of shape $\lambda'/\mu'$ contained in $\lambda/\mu$. Then the set of $[0,1]$ fillings of $\lambda/\mu$ where $\mu'/\mu$ is filled with $0$'s, $\lambda/\lambda'$ is filled with $1$'s, and fence components are filled with the same value, gives a face of $\O(P_{\lambda/\mu})$.
\end{prop}

\begin{proof}
 Recall from \Cref{Geissinger} that faces of $\O(P_{\lambda/\mu})$ correspond to partitions of $\widehat{P}_{\lambda/\mu}$ into \emph{connected}, \emph{compatible} blocks. Inspecting this correspondence, it suffices to show that fence diagrams with shape contained in $\lambda/\mu$ correspond to such partitions of $\widehat{P}_{\lambda/\mu}$.

Let $F$ be a fence diagram with shape $\lambda'/\mu'$ contained in $\lambda/\mu$. Define a partition $\Pi_F$ of $\widehat{P}_{\lambda/\mu}$ whose elements are the fence components of $F$, along with the blocks $\hat0\cup P_{\mu'/\mu}$ and $\hat 1\cup P_{\lambda/\lambda'}$. It is clear that the elements of $\Pi_F$ are connected blocks.

    Recall that for $\Pi_F$ to be  compatible, we must show that for $A, B\in \Pi_F$ with $a_1,a_2\in A$ and $b_1, b_2\in B$ such that $a_1\le b_1$ and $a_2\ge b_2$, it follows that $A = B$. When $A$ or $B$ is the block containing $\hat 0$ or $\hat 1$ this is clear from how skew shapes work. In the remaining case where $A$ and $B$ are fence components of $F$, this is proven below in \Cref{lem:compatible}.

    Conversely, suppose $\Pi$ is a partition of $\widehat{P}_{\lambda/\mu}$ into connected, compatible subsets. By compatibility, the block containing $\hat0$ must be $\hat0\cup P_{\mu'/\mu}$ and the block containing $\hat1$ must be $\hat1\cup P_{\lambda/\lambda'}$ for some $\mu\subseteq \mu'\subseteq \lambda'\subseteq \lambda$. Therefore we may restrict our attention to the skew-shape $\lambda'/\mu'$. Drawing in segments between adjacent cells belonging to the same block of $\Pi$ gives an SDSS $F_\Pi$. Invoking compatibility again shows that $F_{\Pi}$ satisfies the fence axiom \Cref{eq:fence}, and is hence a fence diagram of shape $\lambda'/\mu'$ with fence components equal to the blocks of $\Pi$ not containing $\hat 0$ or $\hat 1$. %(by connectedness of the subsets in $\Pi_F$).
\end{proof}

In particular we see that the map sending a fence diagram to the corresponding face of $\mathcal{O}(P_{\lambda/\mu})$ (i.e., which sends $F$ to the order polytope $\mathcal{O}(P_{\lambda/\mu}/\sim_F)$ of the corresponding quotient poset) is a bijection. We have the following corollary. 
\begin{cor}\label[corollary]{cor:standard_fence_tableaux}
    Every $F\in \mathfrak{F}_u^w$ has a standard fence tableau filling. In particular, the reduced fence diagrams of $\mathfrak{F}_u^w$ are precisely the equality patterns given by border strip tableaux coming from projected saturated $k$-Bruhat chains from $u$ to $w$. 
\begin{proof}
    Let $\lambda = \pi_k(w)$ and $\mu = \pi_k(u)$. Faces of order polytopes are order polytopes, and all posets have total extensions. In particular, if $F\in \mathfrak{F}_u^w$, the fences of $F$ correspond to equality conditions of a face of $P_{\lambda/\mu}$. Hence, we can take any total extension of the poset corresponding to $F$ and this gives a standard fence tableau filling of $F$. 

    The second statement follows from the discussion after \Cref{prop: reduced fence faces know}
\end{proof}
\end{cor}

Via \Cref{prop:fence_faces}, we will freely identify fence diagrams with their corresponding face on $\calO(P_{k,n})$. We end this subsection by proving \Cref{lem:compatible}: that fence components of a fence diagram do in fact form a compatible partition of the poset $P_{\lambda/\mu}$. To this end, we will need some preliminary lemmas.

\begin{lemma}\label[lemma]{lem:rect_diag}
    Let $A$ be a fence component in $F$ containing cells $a\le b$ in the poset $P_{\lambda/\mu}$. Then $A$ contains a path $p$ from $a$ to $b$ such that each cell $(i,j)\in p$ satisfies
    \[\col(a)-\row(b) \le j-i \le \col(b)-\row(a).\]
\end{lemma}

\begin{proof}
    Let $p$ be a path of minimal length in $A$ from $a$ to $b$. Suppose $f(i,j) = j-i$ has a maximum at $(i,j)\in p$, and assume $f(i,j)> \col(b)-\row(a)$. By this assumption, $(i,j)$ is neither of the endpoints $a$ or $b$. Then $p$ must be the corner of an $[\rotL]$. By \Cref{eq:fence}, we may replace this segment with the [L].

    This transformation either reduces the number of times $f$ attains a maximum on $p$ or lowers the value of its maxima. Repeating this procedure ensures $f(i,j)\le \col(b)-\row(a)$ for all $(i,j)\in p$.

    By a similar argument, we make local replacements of [L] by [\rotL] at minima of $f$ until $f(i,j)\ge \col(a)-\row(b)$ for all $(i,j)\in p$.
\end{proof}

\begin{lemma}\label[lemma]{lem:rect}
    Let $A$ be a fence component in $F$ containing cells $a\le b$. Then $A$ contains the entire interval $[a, b]\subset P_{\lambda/\mu}$.
\end{lemma}

\begin{proof}
    We induct on the taxicab distance $d(a,b) = \row(b)-\row(a) + \col(b)-\col(a)$. When $d(a,b) = 1$, there is nothing to show. 
    
    Suppose $d(a,b)>1$. We first deal with the case when $[a,b]$ is the unit square, i.e. the diamond poset. By \Cref{lem:rect_diag}, there exists a path $p$ from $a$ to $b$ such that each $(i,j)\in p$ satisfies $\col(a)-\row(b) \le j-i \le \col(b)-\row(a)$. It is elementary to see that $p$ must then contain either an [L] or [\rotL] with endpoints $a$ and $b$. It then follows from \Cref{eq:fence} that $A$ contains all of $[a,b]$.
    
    Now suppose $[a,b]$ is not a unit square, and pick $p$ from $a$ to $b$ as in \Cref{lem:rect_diag}.
    Define the set
    \[S = \{c\in (a,b)\ |\ \row(a)<\row(c)<\row(b)\text{ or }\col(a)<\col(c)<\col(b)\}.\]
    Since  $[a,b]$ is not a unit square, it follows that $S\ne \emptyset$, and by our assumptions on $p$, the path $p$ must intersect $S$ in some cell, say $c\in S$. 
    
    It suffices to consider the case $\row(a) < \row(c)<\row(b)$ by transposing if necessary. Define the cells
    \[a' = \big(\row(c), \col(a)\big),\ b' = \big(\row(c), \col(b)\big).\]

    Then $a< a'\le c\le  b'< b$. By the induction hypothesis, $A$ contains $a'$ and $b'$, and hence also $[a,b']$ and $[a', b]$. Using $[a,b]=[a,b']\cup [a', b]$, we are done.
\end{proof}
\begin{figure}[H]

    \centering
    \def\borderInnerSep{2pt}
\begin{tikzpicture}[scale=.75]
\draw[maroon] (.35, 2.85) -- (6.65, 2.85) -- (6.65, 4.65) -- (.35, 4.65) -- (.35, 2.85);
\draw[blue] (.35,.35) -- (6.65, .35) -- (6.65, 2.85) -- (.35, 2.85) -- (.35, .35);
%\draw[blue] (.5,.5) -- (6.5,.5) -- (6.5,4.5) -- (.5,4.5) -- (.5,.5);
\def\nodescl{0.7}
\node (A) at (0.1,4.6) {$a$};
\node (A) at (6.85,.4) {$b$};
\node (A) at (6.9, 2.7) {$b'$};
\node (A) at (0.1, 2.7) {$a'$};
\node (A) at (2.4, 2.55) {$c$};

\node[filled] at (2.4, 2.85){};

\fill[maroon, fill opacity=0.3] (0.35, 4.65) rectangle (6.65, 2.85);
\fill[blue, fill opacity=0.3] (0.35, 2.85) rectangle (6.65, .35);

\end{tikzpicture}\label{fig:lem 2.3 pic}
\caption{A graphical depiction of the last step in the proof of \Cref{lem:rect}. Induction shows that the fence component $A$ contains all fences in both the purple and blue rectangles, so $A$ contains every cell in the big rectangle.}
\end{figure}
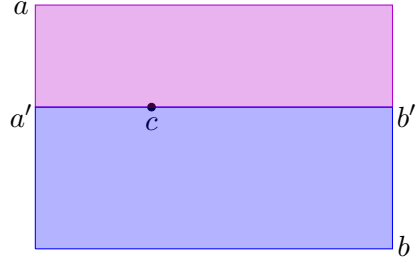

By a \emph{taxicab path} between cells $a$ and $b$, we mean a collection of fences forming a minimum length path from $a$ to $b$. Such a path must necessarily have length equal to the taxicab distance $$d(a,b) =|\row(a)-\row(b)|+|\col(a)-\col(b)|.$$

\begin{lemma}\label[lemma]{lem:taxicab}
    Let $A$ be a fence component in $F$ containing cells $a, b$. Then $A$ contains a taxicab path in $P_{\lambda/\mu}$ from $a$ to $b$.
\end{lemma}

\begin{proof} 
    We induct on $d(a,b)$. If $d(a,b) = 1$, then the path $p$ from \Cref{lem:rect_diag} must be a fence from $a$ to $b$, and we are done.
    
    Next, if $a$ and $b$ are comparable in 
    $P_{\lambda/\mu}$ with $d(a,b)>1$, we are done by \Cref{lem:rect}. So, suppose that $a$ and $b$ are incomparable in $P_{\lambda/\mu}$. Without loss of generality we may assume $\row(a)>\row(b)$, and hence $\col(a)<\col(b)$. Also by induction, assume the lemma is true for cells $a'$ such that $d(a', b)<d(a, b)$.
    
    Since $a, b$ are in the same component, there is some path $p$ in $A$ from $a$ to $b$. Consider the rays of cells
    $$R(a) = \{c\mid \row(c) = \row(a), \col(c) > \col(a)\}$$ and    
    $$C(a) =  \{c\mid  \row(c)<\row(a), \col(c) = \col(a)\}.$$
    Removing $R(a)\sqcup C(a)$ from $P_{\lambda/\mu}$ disconnects $a$ from $b$, so the path $p$ must intersect this union non-trivially. Assume first that there exists $c\in R(a)\cap p$. By \Cref{lem:rect} applied to $a\le c$, we deduce that $a' = \big(\row(a), \col(a) + 1\big)\in A$. By the induction hypothesis, $A$ contains some taxicab path from $a'$ to $b$, and adding the edge from $a$ to $a'$ gives the desired taxicab path from $a$ to $b$.

    The case $c\in C(a)\cap p$ follows similarly by considering $a' = \big(\row(a) - 1, \col(a)\big)$.
\end{proof}

\begin{figure}[H]
    \centering
    \begin{tikzpicture}[baseline={([yshift=-.7ex]current bounding box.center)},scale=1.25]

    \node[filled,color=maroon] at (3,0){};

  \node at (3.2,0) {$b$};
  \node at (-0.2,-2.25) {$a$};
  \node at (0.5,-2.2) {$a'$};
  \node at (4.2,-2.2) {$c$};
  \node at (1.8,-0.5) {\textcolor{maroon}{$p$}};
  \node[unfilled] at (0,-2){};
  \node[filled] at (0.5,-2){};
  \begin{pgfonlayer}{bg}
    \draw[color=grey] (0,-2) to (5,-2);
    \draw[color=grey] (0,-2) to (0,1);
    \node at (-0.4,-.5) {$C(a)$};
    \node at (2,-2.25) {$R(a)$};

    \draw[color=maroon] (0,-2) to (0,-2.5);
  \draw[color=maroon] (0,-2.5) to (4,-2.5);
  \draw[color=maroon] (4,-1) to (4,-2.5);
  \draw[color=maroon] (4,-1) to (2,-1);
  \draw[color=maroon] (2,0) to (2,-1);
  \draw[color=maroon] (2,0) to (3,0);

  \draw[color=blue] (0.5,-2)--(1.5,-2)--(1.5,-1.5)--(2.5,-1.5)--(2.5,-0.5)--(3,-0.5)--(3,0);
  \draw[color=blue] (0,-2)--(0.5,-2);
  \end{pgfonlayer}
\end{tikzpicture}
    \caption{A depiction of the situation described in the proof of \Cref{lem:taxicab}, when $a$ and $b$ are incomparable. Induction ensures the existence of the blue path from $a'$ to $b$ contained in $A$.} 
    \label{fig:taxicab}
\end{figure}
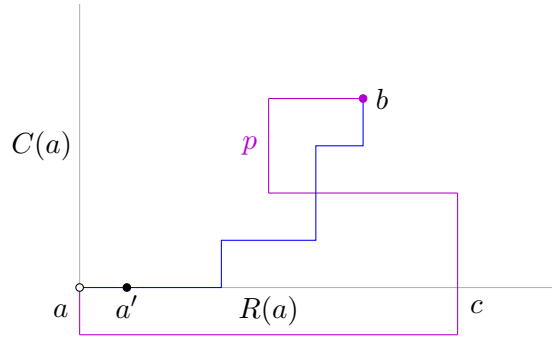

\begin{lemma}\label[lemma]{lem:compatible}
    Suppose $a_1,a_2\in A$ and $b_1, b_2\in B$ where $a_1, a_2, b_1, b_2$ are cells and $A, B$ are fence components of a fence diagram $F$. If $a_1\le b_1$ and $a_2\ge b_2$ in $P_{\lambda/\mu}$, then $A = B$.
\end{lemma}

\begin{proof}
    We leverage two kinds of symmetries that the setup has:
    \begin{itemize}
        \item Transposing the fence diagram, and
        \item Interchanging the roles of $a_i$ and $b_i$.
    \end{itemize}

    Imagining $a_1$ and $b_1$ to be fixed, we analyze the various cases for the positions of $a_2$ and $b_2$ relative to $a_1$ and $b_1$.

    \begin{enumerate}
        \item Suppose $a_2$ is comparable to $b_1$, or $b_2$ is comparable to $a_1$. 

        By potentially swapping the $a_i$ and $b_i$, we may assume $a_2$ is comparable to $b_1$.
        
        If $a_2\le b_1$, then by the assumption in the statement of the lemma, $b_2\le a_2\le b_1$. By \Cref{lem:rect}, $a_2\in B$, showing that $A = B$.

        Instead, if $a_2\ge b_1$, we get $a_1\le b_2\le a_2$, so \Cref{lem:rect} shows that $b_1\in A$, and thus $A=B$.

        \item Suppose $a_2$ is not comparable to $b_1$, and $b_2$ is not comparable to $a_1$. By potentially transposing, we may assume that
        \[\row(a_2)<\row(b_1),\ \col(a_2)>\col(b_1).\]
        Since $b_2\le a_2$, we may divide into two subcases based on whether $b_2$ is above or below $a_1$:

        \begin{enumerate}
            \item  $\row(b_2)<\row(a_1)$:
            
            Since $b_2$ is not comparable to $a_1$, this forces $\col(b_2)>\col(a_1)$ (see \Cref{fig:compatibility}). By \Cref{lem:taxicab}, $A$ contains a taxicab path from $a_1$ to $a_2$, and $B$ contains a taxicab path from $b_1$ to $b_2$. But any two such taxicab paths must intersect for this relative configuration of the cells, from which we deduce that $A=B$.

            \item  $\row(b_2)>\row(a_1)$:

            Since $b_2$ is not comparable to $a_1$, this forces $\col(b_2)<\col(a_1)$.

            Satisfying all three conditions
            \begin{equation*}
                \begin{cases}
                & b_2\le a_2,\\
                & \row(a_2)<\row(b_1),\ \col(a_2)>\col(b_1)\text{, and}\\
                & \row(b_2)> \row(a_1),\ \col(b_2)<\col(a_1)      
            \end{cases}
            \end{equation*}
            
            forces 
            \[\row(a_1)<\row(b_2)\le\row(a_2)<\row(b_1),\]
            \[\col(b_2)<\col(a_1),\ \col(a_2)>\col(b_1).\]
            (see \Cref{fig:compatibility}).

            In particular, $a_1\le a_2$ and $b_2\le b_1$ in $P_{\lambda/\mu}$. By \Cref{lem:rect}, $A$ contains the interval $[a_1, a_2]\subset P_{\lambda/\mu}$, and $B$ contains the interval $[b_2, b_1]\subset P_{\lambda/\mu}$.

            Pick integers $i$ and $j$ such that            
            \[\row(b_2)\le i\le \row (a_2),\]
            \[\col(a_1)\le j\le\col(b_1).\] 
            Then the cell $(i,j)$ is in both the intervals $[a_1,a_2]$ and $[b_2, b_1]$. Therefore, $(i,j)$ is in both $A$ and $B$, demonstrating that $A=B$.
        \end{enumerate}        
    \end{enumerate}

\end{proof}

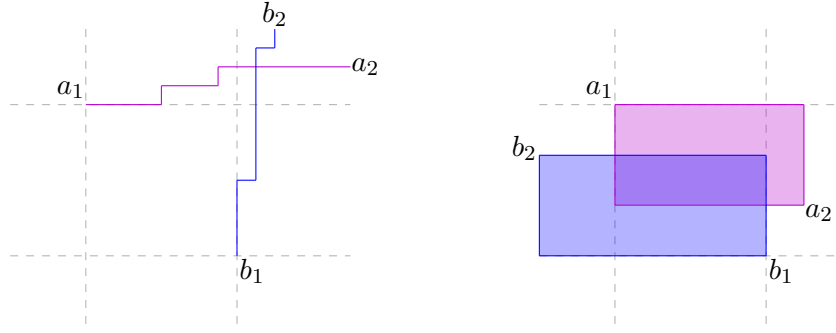
\begin{figure}[H]
    \centering
    \begin{tikzpicture}[baseline={([yshift=-.7ex]current bounding box.center)}]

  \node at (-0.2,0.2) {$a_1$};
  \node at (2.2,-2.2) {$b_1$};
  \node at (2.5,1.2) {$b_2$};
  \node at (3.7,0.5) {$a_2$};
  \begin{pgfonlayer}{bg}
    \draw[color=grey, dashed] (0,1) to (0,-3);
    \draw[color=grey, dashed] (2,1) to (2,-3);
    \draw[color=grey, dashed] (-1,0) to (3.5,0);
    \draw[color=grey, dashed] (-1,-2) to (3.5,-2);
  \end{pgfonlayer}

  \draw[color=maroon] (0,0) to (1,0);
  \draw[color=maroon] (1,0) to (1,0.25);
  \draw[color=maroon] (1.75,0.25) to (1,0.25);
  \draw[color=maroon] (1.75,0.25) to (1.75,0.5);
  \draw[color=maroon] (1.75,0.5) to (3.5,0.5);

  \draw[color=blue] (2.5,1) to (2.5,0.75);
  \draw[color=blue] (2.25,0.75) to (2.5,0.75);
  \draw[color=blue] (2.25,0.75) to (2.25,-1);
  \draw[color=blue] (2,-1) to (2.25,-1);
  \draw[color=blue] (2,-1) to (2,-2);

  \node at (6.8,0.2) {$a_1$};
  \node at (9.2,-2.2) {$b_1$};
  \node at (5.8,-0.57) {$b_2$};
  \node at (9.7,-1.43) {$a_2$};
  \begin{pgfonlayer}{bg}
    \draw[color=grey, dashed] (7,1) to (7,-3);
    \draw[color=grey, dashed] (9,1) to (9,-3);
    \draw[color=grey, dashed] (6,0) to (10,0);
    \draw[color=grey, dashed] (6,-2) to (10,-2);

    \fill[maroon, fill opacity=0.3] (7, 0) rectangle (9.5, -1.33);
    \fill[blue, fill opacity=0.3] (6, -0.67) rectangle (9, -2);
  \end{pgfonlayer}

  \draw[color=maroon] (7,0) to (9.5,0);
  \draw[color=maroon] (7,-1.33) to (9.5,-1.33);
  \draw[color=maroon] (7,0) to (7,-1.33);
  \draw[color=maroon] (9.5,0) to (9.5,-1.33);

  \draw[color=blue] (6,-0.67) to (9,-0.67);
  \draw[color=blue] (6,-2) to (9,-2);
  \draw[color=blue] (6,-0.67) to (6,-2);
  \draw[color=blue] (9,-2) to (9,-0.67);  
\end{tikzpicture}
    \caption{The situations in cases (2a) and (2b), respectively, in the proof of \Cref{lem:compatible}}
    \label{fig:compatibility}
\end{figure}

\begin{rem}\label[remark]{rem:quotient_poset}
    Given a fence diagram $F$ of shape $\lambda/\mu$, the poset $P_{\lambda/\mu}$ induces a poset structure on $(P_{\lambda/\mu})/\sim_F$, where $\sim_F$ is the equivalence relation of belonging to the same fence component. The compatibility from \Cref{lem:compatible} ensures that this is well defined. Then, one readily sees that standard fence tableaux correspond to total extensions of  $(P_{\lambda/\mu})/\sim_F$. 
\end{rem}

\section{Ehrhart Theory of Fence Complexes and Positroid Geometry}\label[section]{Ehrhart Theory of Fence Complexes and Positroid Geometry} 
In this section, we describe a polytopal model for the positroid stratification of $Gr(k, n)$. Throughout this section, fix $u,w \in S_n$ with $u\leq_k w$. 

\subsection{Fence Complexes and $\pi_k(\Delta([u, w]))$}

\begin{defn}\label[definition]{fence complexes}
    The \emph{fence complex} $\F_u^w$ is the polyhedral subcomplex of $\O(P_{k, n})$ with total space the union of the faces corresponding to the fence diagrams in $\mathfrak{F}_u^w$.
\end{defn}
\begin{rem}\label[remark]{fence complex is pure dimensional}
One can verify from the definitions that each $F\in \mathfrak{F}_u^w$ corresponds to a $\ell(w)-\ell(u)$-dimensional facet of $\calF_u^w$, and hence $\calF_u^w$ is a pure dimensional complex whose facets are precisely given by the reduced fence diagrams.
\end{rem}

This definition makes sense by \Cref{prop:fence_faces}. Since $\mathcal{F}_u^w$ is a polyhedral subcomplex of $\O(P_{k, n})$, we may consider its canonical triangulation $\mathcal{T}_u^w$ as in \Cref{canonical triangle}.

\begin{lemma}\label[lemma]{canonical equals projected}
    The vertices of  $\mathcal{F}_u^w$, and hence of the simplices in $\mathcal{T}_u^w$ are precisely the elements of $\pi_k([u,w])$. 
\end{lemma}

\begin{proof}
    Since canonical triangulations of polyhedral subcomplexes of order polytopes don't introduce new vertices, it suffices to check that the vertices of $\mathcal{F}_u^w$ are the elements of $\pi_k([u, w])$.

    Given a fence diagram $F\in \mathfrak{F}_u^w$, the vertices correspond to order ideals of the poset $P_{\pi_k(w)/\pi_k(u)}/\sim_F$. Given a poset $Q$, any order ideal of the poset may be written as $f^{-1}([m])$ where $f:Q\to [|Q|]$ is a total extension of $Q$. Total extensions of all $P_{\pi_k(w)/\pi_k(u)}/\sim_F$ correspond to standard fence tableaux of $F$. By \Cref{prop: reduced fence faces know}, each standard fence tableau $T$ corresponds to a saturated $k$-Bruhat chain $w=u_m\gtrdot_k\cdots\gtrdot_ku_0=u$ in $[u,w]_k$, with $\pi_k(u_i)$ corresponding to the region of $T$ filled with numbers from $1$ to $i$. Since every partition in $\pi_k([u,w])$ lifts to a saturated $k$-Bruhat chain by \Cref{maximal chains bij}, the result follows. 
\end{proof}

We may now show the following proposition, which relates the fence complex to the projected order complex studied in \cite{KLS0} and \cite{KLS2}. 

\begin{prop}\label[proposition]{canonical triangulation is the same as pi_k(delta[u,w]))} 
    As a simplicial complex, $\mathcal{T}_u^w$ is isomorphic to $\pi_k(\Delta[u,w])$. 
\end{prop}
\begin{proof}
    Since the two simplicial complexes have the same vertex set, it suffices to show they have the same facets. But by \Cref{maximal chains bij}, the facets of $\pi_k(\Delta[u,w])$ are precisely given by projections of saturated $k$-Bruhat chains from $u$ to $w$, and the facets of $\mathcal{T}_u^w$ are in bijection with standard fence tableaux of the fence diagrams corresponding to $[u,w]$, and hence by \Cref{prop: reduced fence faces know} with saturated $k$-Bruhat chains. Therefore the facets of the two simplicial complexes agree.
\end{proof}

Recall from \Cref{shellable} that $\pi_k(\Delta[u,w])$ is homeomorphic to a closed ball. In particular, we have the following corollary: 

\begin{cor}\label[corollary]{fence complexes are homeomorphic to balls}
    $\calF_u^w$ is homeomorphic to a closed ball. 
\end{cor}

\begin{exmp}
    Let $u = 2143$, $w = 3412$, and $k = 2$. In this case $u_J = u = (s_1, s_1) \in S_2 \times S_2$ and $w_J = e$, so the set $\mathfrak{F}_u^w$ consists of two fence diagrams $F_1$ and $F_2$, the first two depicted in \Cref{fig:example}. The faces of $\calO(P_{k,n})$ corresponding to $F_1$ and $F_2$ are both isomorphic to the order polytope of the total order on $[2]$, which is a $2$-simplex. Further, these faces intersect along the face of $\calO(P_{2,4})$ corresponding to $F_3$, the final diagram in \Cref{fig:example}, which is a $1$-simplex. Therefore we see that $\mathcal{F}_{2143}^{3412}$ is two $2$-simplices glued along an edge, which is certainly homeomorphic to a closed ball.
\end{exmp}

\begin{figure}[H]

    \centering
    \def\borderInnerSep{2pt}
\begin{tikzpicture}

\node[filled, color=maroon] at (1.5,1.5) {};
\node[filled, color=maroon] at (0.5,0.5) {};
\node[filled, color=maroon] at (0.5,1.5) {};
\node[] at (1, -.5) {$F_1$};

\draw (0,0)-- (0,2)--(2,2)--(2,1)--(2,0)--(0,0);
\draw (1,0)--(1,2);
\draw (2,0)--(2,2);
\draw (0,1)--(2,1);
\draw [thick, color=maroon](1.5,1.5)--(.5,1.5);
\draw [thick, color=maroon](.5,1.5)--(.5,.5);

\begin{scope}[shift={(5,0)}]

\node[filled, color=maroon] at (1.5,1.5) {};
\node[filled, color=maroon] at (1.5,0.5) {};
\node[filled, color=maroon] at (0.5,0.5) {};
\node[] at (1, -.5) {$F_2$};

\draw (0,0)-- (0,2)--(2,2)--(2,1)--(2,0)--(0,0);
\draw (1,0)--(1,2);
\draw (2,0)--(2,2);
\draw (0,1)--(2,1);
\draw [thick, color=maroon](1.5,1.5)--(1.5,.5);
\draw [thick, color=maroon](.5,.5)--(1.5,.5);
\end{scope}

\begin{scope}[shift={(10,0)}]

\node[filled, color=maroon] at (0.5,0.5) {};
\node[filled, color=maroon] at (1.5,0.5) {};
\node[filled, color=maroon] at (0.5,1.5) {};
\node[filled, color=maroon] at (1.5,1.5) {};
\node[] at (1, -.5) {$F_3$};

\def\nodescl{0.7}
\draw (0,0)-- (0,2)--(2,2)--(2,1)--(2,0)--(0,0);
\draw (0,1) -- (2,1);
\draw (1,0) -- (1, 2);
\draw[thick, color=maroon] (.5,.5)--(1.5,.5)--(1.5,1.5);
\draw[thick, color=maroon](.5,.5) -- (.5,1.5)--(1.5,1.5);
\end{scope}
\end{tikzpicture}
\caption{The fence diagrams corresponding to the faces of $\calO(P_{2,4})$ contained in $\mathcal{F}_{2143}^{3412}$. The fences correspond to the equality conditions cutting out the faces.}\label{fig:example}
\end{figure}
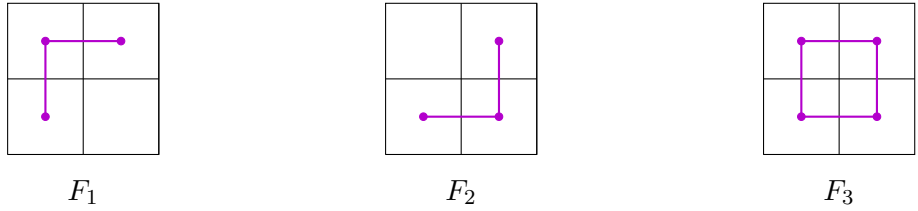

\subsection{The Hilbert Polynomial of $\Pi_u^w$}

 We can now connect the Ehrhart function of the fence complex to the Hilbert function of the positroid variety. First, recall that if $\Delta$ is a simplicial complex with vertex set $V(\Delta)$, then \emph{the Stanley–Reisner ring} $SR(\Delta)$ is the ring $\bbC[x_i\ |\  i\in V(\Delta)]/I(\Delta)$ where $I(\Delta)$ is generated by $x_{i_1}\cdots x_{i_k}$ for all $i_1,\ldots, i_k \in V(\Delta)$ such that $\{i_1,\ldots, i_k\}$ is not contained in a face of $\Delta$. 

Stanley–Reisner rings enter our story for the following nice geometric reason: 
\begin{thm}\cite[Theorem 9.1.]{KLS0}
There exists a Gröbner degeneration from $\bbC[\Pi_u^w]$ to $SR(\pi_k(\Delta[u,w]))$.    
\end{thm}

\begin{thm}\label[theorem]{ehrhart poly equals hilbert poly} 
  The Ehrhart polynomial of $\calF_u^w$ equals the Hilbert polynomial of $\Pi_u^w$. 
\begin{proof}
 Since Gröbner degenerations are flat, and flat degenerations preserve Hilbert functions, we have \[H_{\Pi_u^w}(m) = H_{SR(\pi_k(\Delta[u,w]))}(m).\] Recall, say from \cite[Theorem 6.15]{Francisco2014ASO}, the formula for the Hilbert series of a Stanley–Reisner ring gives us 
    \[\sum_{m=0}^\infty H_{\Pi_u^w}(m)t^m = \sum_{m=0}^\infty H_{SR(\pi_k(\Delta[u,w]))}(m)t^m =   \frac{h(t)}{(1-t)^{\ell(w)-\ell(u)+1}},\] 
    where $h(t)$ is the $h$-polynomial of $\pi_k(\Delta [u,w])$. By \Cref{canonical equals projected}, $\pi_k([u,w])$ gives a unimodular triangulation of $\calF_u^w$. Then, by \Cref{triangle equals complex h-vector} we have that 
    \[\sum_{m=0}^\infty \ehr_{\calF_u^w}(m)t^m =\frac{h(t)}{(1-t)^{\ell(w)-\ell(u)+1}}, \] 
    and so 
    \[\sum_{m=0}^\infty \ehr_{\calF_u^w}(m)t^m = \sum_{m=0}^\infty H_{\Pi_u^w}(m)t^m,\] 
    which gives the result. 
\end{proof}
\end{thm}

By \Cref{fence complexes are homeomorphic to balls}, we may apply \Cref{general reciprocity} to the fence complex. This combined with \Cref{ehrhart poly equals hilbert poly} gives the following analogue of Ehrhart–Macdonald reciprocity.  
\begin{thm}\label[theorem]{mainresultC} 
Let $m\in \bbN$. The Hilbert polynomial of $\Pi_u^w$ evaluated at $-m$ is, up to sign, the number of lattice points in the interior of $m\calF_u^w$. In symbols, \[(-1)^{\dim \calF_u^w} H_{\Pi_u^w}(-m) = \mathrm{intehr}_{\calF_u^w}(m).\] 
\end{thm}

\begin{rem}
    In \cite[Conjecture 10]{Alexandersson-Alhajjar}, they consider principal specializations of key polynomials $\kappa_{\lambda, w}$ corresponding to an integral weight $\lambda$ and $w\in S_n$. They then conjecture that $P_{\lambda, w}(d) = \kappa_{d\lambda,w}(1,1,\ldots, 1)$ is a polynomial with positive coefficients. 
    
    One can show that when $\lambda =\omega_k$ is a fundamental weight, $P_{\omega_k, w}(d)= \dim \bbC[\Pi_e^w]_d$, where $e$ is the identity permutation. In particular, by \Cref{ehrhart poly equals hilbert poly}, $P_{\omega_k, w}(d) = \ehr_{\calF_e^w}(d)$. \Cref{mainresultC} then implies that the $\deg(P_{\omega_k,w})-1$ degree coefficient of $P_{\omega_k, w}(d)$ is always positive. This is always true for Ehrhart polynomials of polytopes, but not necessarily for polyhedral complexes. We omit the details. 
\end{rem}

\subsection{Arithmetic Gorensteinness of $\Pi_u^w$} 

As another corollary of our fence complex construction, one can deduce the arithmetic Gorensteinness of certain positroid varieties. 

\begin{defn}\label[definition]{arith gor}
    A projective variety $X$ is called \textit{arithmetically Gorenstein} if its homogeneous coordinate ring is a Gorenstein algebra.
\end{defn}

The homogeneous coordinate ring of a projective variety is naturally graded, and Stanley gave a characterization for when a graded Cohen–Macaulay integral domain is Gorenstein. Recall that a polynomial $f(t)$ is \emph{palindromic} if $f(t)$ equals $f(1/t)t^N$ for some $N\in \bbN$. Then, Stanley showed that such a ring is Gorenstein if and only if the numerator of its Hilbert series is palindromic, see \cite{STANLEY197857}. Since the homogeneous coordinate rings of positroid varieties are always Cohen–Macaulay (see \cite[Corollary 4.9]{KLS0}), we can understand the arithmetic Gorensteinness of positroid varieties by understanding the Hilbert series of their homogeneous coordinate rings. Furthermore, by \Cref{ehrhart poly equals hilbert poly}, we only need to understand the Ehrhart series of the fence complex.

Recall that for a polyhedral complex $\calF$, the \emph{$h^*$-polynomial} of $\calF$ is defined to be the unique polynomial $h^*(\calF,t)$ satisfying
\[\sum_{m=0}^\infty \ehr_\calF (m)t^m = \frac{ h^*(\calF,t)}{(1-t)^{\dim \calF +1}}.\] 
The above discussion implies:
\begin{cor}
    A positroid variety $\Pi_u^w$ is arithmetically Gorenstein (with respect to the usual embedding) if and only if $h^*(\calF_u^w,t)$ is palindromic.
\end{cor}

We give some examples now where $\calF_u^w$ is a single polytope and the above condition is readily verifiable. 

\begin{lemma}
    Fix $u \leq w$, and assume that either $u$ and $w$ are both $k$-Grassmannian, or that $\pi_k(w)/\pi_k(u)$ is a border-strip. Then $\calF_u^w$ is an order polytope.
\end{lemma}

\begin{proof}
    In the former case, the elements of $\mathfrak{F}_u^w$ correspond to fence diagrams of shape $\pi_k(w)/\pi_k(u)$ with no fences, of which there is only one.

    In the latter case, assume given $F \in \mathfrak{F}_u^w$. The assumptions on $\pi_k(w)/\pi_k(u)$ ensure that, given a letter $s_i \in \col_F$, there is a unique fence one could draw in $\pi_k(w)/\pi_k(u)$ and contribute $s_i$ to said word. The same is true for $\row_F$. Therefore in this case $\col_F$ and $\row_F$ uniquely determine $F$, and so $\mathfrak{F}_u^w$ consists of a single element.

    Finally, since every $F$ we consider is a face of $\mathcal{O}(P_{k,n})$, and faces of an order polytope are themselves order polytopes, the result follows.
\end{proof}

The following proposition connects the symmetry of the $h^*$-polynomial of an order polytope to the combinatorics of the poset. 
\begin{prop}\label[proposition]{prop: palindromicity}
    Let $\calO(P)$ be an order polytope corresponding to a poset $P$. Then, $h^*(\calO(P),t)$ is palindromic if and only if $P$ is graded. 

\end{prop}
For lack of an explicit reference in our language, we give a proof, but we note that this essentially follows from Stanley's work. This result also follows from the algebra with straightening law structure placed on the coordinate ring of the order polytope in \cite{Hibi}. All undefined terminology can be found in \cite[Chapter 3]{EC1}. 

First, recall that $\ehr_{\calO(P)}(m) = \Omega_P(m)$, where $\Omega_P(m)$ equals the number of order preserving maps from $P\to [m]$. In particular, by \cite[Theorem 3.15.8]{EC1}, 
\[h^*(\calO(P),t) = \sum_{w\in \calL_P(w)} t^{1+\textrm{des}(w)},\]
where $\calL_P(w)$ is the set of all total extensions of $P$, and $\textrm{des}(w)$ is the number of descents of the total extension with respect to a fixed natural labeling of $P$. We then have the following lemma: 
\begin{lemma}
    $h^*(\calO(P),t)$ has degree equal to $|P|-l$, where $l$ is the maximum length of a chain in $P$. 
\begin{proof}
    Fix a natural labeling of $P$, i.e. assume that $P$ equals $[|P|]$ as a set, and the partial order on $P$ respects the total order on $[|P|]$.
    
    If $p_1<\cdots <p_{l+1}$ is a chain of maximum length, note that for any total extension $\sigma$, we cannot have all descents between $\sigma(p_i)$ and $\sigma(p_{i+1})$, as $\sigma(p_i)<\sigma(p_{i+1})$. Hence, any total extension has at most $|P|-l-1$ descents. We can construct a total extension with exactly $|P|-l-1$ descents as follows: order all the minimal elements in $P$ in decreasing order with respect to the natural labeling, and then remove them from $P$, producing a new poset, $Q$. Repeat the process with the minimal elements of $Q$, and continue. By construction, the permutation one gets is a total extension of $P$ with $|P|-l-1$ descents.  
\end{proof}
\end{lemma}

\begin{proof}[Proof of \Cref{prop: palindromicity}]  \cite[Corollary 3.15.18]{EC1} implies that $P$ is graded if and only if \[\Omega_P(-l-m) = (-1)^{|P|}\Omega_P(m)\] 
for all $m\in \bbN$. This second statement is equivalent to 
\[\sum_{m=1}^\infty \Omega_P(m) t^m = \sum_{m=1}^\infty (-1)^{|P|}\Omega_P(-l-m)t^m.\] 

On the other hand, \cite[3.15.11]{EC1} implies that for any poset $P$, we have 
\[\sum_{m=1}^\infty \Omega_P(-m)t^m =- \frac{h^*(\calO(P),1/t)}{(1-1/t)^{|P|+1}} = (-1)^{|P|} \frac{t^{|P|+1} \cdot h^*(\calO(P),1/t)}{(1-t)^{|P|+1}}.\] 
Since for $m\in \bbN$, $|\Omega_P(-m)|$ equals the number of strict order preserving maps to $[m]$ (by \cite[Corollary 3.15.12]{EC1}), it follows that 
\[\sum_{m=1}^\infty \Omega_P(-m)t^m  = \sum_{m={l+1}}^\infty \Omega_P(-m)t^m, \] 
and so 
\[\sum_{m=0}^\infty (-1)^{|P|}\Omega_P(-l-m)t^m  =  \frac{1}{t^l} \cdot \left(\sum_{m=l}^\infty (-1)^{|P|}\Omega_P(-m)t^m\right) = \frac{t^{|P|-l+1} \cdot h^*(\calO(P), 1/t)}{(1-t)^|P|+1}.\]

It follows that $\Omega_P(-l-m) = (-1)^{|P|}\Omega_P(m)$ 
for all $m\in \bbN$ if and only if $t^{|P|-l+1} \cdot h^*(\calO(P), 1/t) = h^*(\calO(P),t)$. Since $h^*(\calO(P),t)$ has degree $|P|-l$ and its minimum nonzero degree term has degree $1$ (by the descent characterization), this holds if and only if $h^*(\calO(P),t)$ is palindromic.  Thus, $P$ is graded if and only if $h^*(\calO(P),t)$ is palindromic. 
\end{proof}
When we discuss corners of a skew-shape $\lambda/\mu$, we say left and right corners to refer to the \textit{outer} corners along the left and right side of $\lambda/\mu$, respectively. Equivalently, the left and right corners of $\lambda/\mu$ are the cells corresponding to the minimal and maximal elements of $P_{\lambda/\mu}$ (resp.).

We may now give some applications. Arithmetic Gorensteinness of Schubert varieties in $\Gr(k,n)$ was originally studied by Stanley in \cite{STANLEY197857}, and arithmetic Gorensteinness of Richardson varieties in $\Gr(k,n)$ has been studied in the upcoming preprint of Darayon (\cite{Darayon}). The latter provides a combinatorial criterion for detecting arithmetic Gorenstein of Grassmannian Richardson varieties. Using the formalism of fence complexes along with the above discussion, one can recover her result, but we omit the proof.

\begin{prop}\label[proposition]{Darayon}
    Assume that $u \leq w$ and both $u$ and $w$ are $k$-Grassmannian and that $\pi_k(w)/\pi_k(u)$ is connected. Then, $\Pi_u^w$ is arithmetically Gorenstein if and only if both all left corners of $\pi_k(w)/\pi_k(u)$ lie along the same antidiagonal and all right corners of $\pi_k(w)/\pi_k(u)$ lie along the same antidiagonal.
\end{prop}

%\begin{proof}
%By \Cref{prop: palindromicity} it suffices to determine when the quotient poset $P_{k,n}/\sim_F$ corresponding to the unique $F \in \mathfrak{F}_u^w$ is graded. In the case at hand this poset is $P_{\pi_k(w)/\pi_k(u)}$. Then, if the assumptions on the corners are satisfied one sees that the taxicab distance between any pair of corners in $\pi_k(w)/\pi_k(u)$ is the same, which means that every maximal chain in $P_{\pi_k(w)/\pi_k(u)}$ has the same length.
    
    %Conversely, say without loss of generality that not all of the right corners of $\pi_k(w)/\pi_k(u)$ are along the same antidiagonal. Ordering the right corners of $\pi_k(w)/\pi_k(u)$ by $(i_1, j_1) < (i_2, j_2)$ if and only if $i_1 < i_2$, one sees that the meet of two adjacent corners in this order must exist in $P_{\pi_k(w)/\pi_k(u)}$ by connectedness of $\pi_k(w)/\pi_k(u)$. Further, there must be a pair of right corners $x$ and $y$ of $\pi_k(w)/\pi_k(u)$ which are adjacent in this order but such that the antidiagonal $x$ is not equal to the antidiagonal of $y$.
    
    %Thus there exist right corners $x$ and $y$ in $\pi_k(w)/\pi_k(u)$ lying on different antidiagonals, but greater in $P_{\pi_k(w)/\pi_k(u)}$ than a common minimal element $l$. This means that the lengths of maximal chains from $l$ to $x$ and $l$ to $y$ are not the same, but that such chains exist in $P_{\pi_k(w) / \pi_k(u)}$. Therefore $P_{\pi_k(w) / \pi_k(u)}$ is not graded.
%\end{proof}

Now consider the case where $\pi_k(w)$ is a  \textit{hook}, by which we mean a partition $\lambda$  such that $\lambda_i\leq 1$ for all $i>1$. We define $a(\lambda):= \lambda_1$ and $l(\lambda)$ to equal the number of parts of $\lambda$. If  $\mu\leq \lambda$, then $\mu$ is also a hook. we define $a(\lambda/\mu) = a(\lambda)-a(\mu)$, and $l(\lambda/\mu) = l(\lambda)-l(\mu)$. We call the first row of $\lambda/\mu$ \emph{the arm} and the first column of $\lambda/\mu$ \emph{the leg}.

\begin{prop}\label[proposition]{hook gorenstein}
Assume that $u \leq w$ such that $\pi_k(w)$ is a hook. Then $\Pi_u^w$ is arithmetically Gorenstein if and only if $$a(\pi_k(w)/\pi_k(u)) - |\{s_i\:|\:s_i \in u_J, i \geq k\}| = l(\pi_k(w)/\pi_k(u)) - |\{s_i\:|\:s_i \in u_J, i < k\}|.$$

%%the number of fence components contained in the unique row of $F \in \mathfrak{F}_u^w$ containing more than one cell is equal to the number of fence components contained in the unique column of $F$ containing more than one cell.

\end{prop}

\begin{proof}
    By \Cref{prop: palindromicity} it suffices to determine when the quotient poset $P_{k,n}/\sim_F$ corresponding to the unique $F \in \mathfrak{F}_u^w$ is graded. In the case at hand this poset is $P_{\pi_k(w)/\pi_k(u)}$. and the assumptions ensure that there are only two maximal chains in $(P_{\pi_k(w)/\pi_k(u)})/\sim_F:$ the one from the top-left cell of $\pi_k(w)/\pi_k(u)$ to the end of the arm, and the one from the same cell to the end of the leg.
    
    The length of these chains is the number of boxes in $\pi_k(w)/\pi_k(u)$ contained in the arm/leg (resp.) minus the number which have been identified by $\sim_F$. The latter count is the number of fences drawn between cells in the arm/leg, which agrees with the length of the row/column word of $F \in \mathfrak{F}_u^w$. That said length equals $|\{s_i\:|\:s_i \in u_J, i \geq k\}|$ or $|\{s_i\:|\:s_i \in u_J, i < k\}|$ is definitional. Therefore we see that $(P_{\pi_k(w)/\pi_k(u)})/\sim_F$ is graded if and only if the claimed equality holds.
\end{proof}

In fact the previous proof essentially shows the following proposition, which is quite indexically cumbersome to express only through Bruhat combinatorics but natural to state in terms of fences.

\begin{cor}\label[corollary]{border strip arith}
    Assume that $u \leq w$ and $\pi_k(w) / \pi_k(u)$ is a connected border strip. Then $\Pi_u^w$ is arithmetically Gorenstein if and only if the quantity $$|X| - |\{\text{fences of $F$ contained in X}\}|,\:F \in \mathfrak{F}_u^w$$ \noindent is the same for every row or column $X \subset \pi_k(w)/\pi_k(u)$ containing more than one cell.
\end{cor}

We leave the reader to deduce analogues of these examples when $\pi_k(w)/\pi_k(u)$ is disconnected. The statements are essentially the same, save with the additional hypothesis that lengths of maximal chains are constant across connected components of $P_{\pi_k(w)/\pi_k(u)}/\sim_F$.

\section{Further Topology of Fence Complexes}\label[section]{Further Topology of the Fence Complex}
Recall that the Demazure product $\star$ is the unique associative product $S_n\times S_n\to S_n$ such that 
\[w\star s_i = \begin{cases}
    w \text{ if } \ell(ws_i)<\ell(w) \\ ws_i \text{ if } \ell(ws_i)>\ell(w). 
\end{cases}\] 

This definition extends to words in $S_n$.

\begin{defn}\label[definition]{Demazure algorithm}
    Given a word $\alpha= s_{i_1}s_{i_2}\cdots s_{i_m}$, its \textit{Demazure product} $\delta(\alpha)$ is defined by
    \[\delta(\alpha) := s_{i_1}\star s_{i_2} \star \cdots \star s_{i_m}.\]
\end{defn}

\begin{rem}\label[remark]{rem:greedy_demazure}
    Given a word $\alpha = s_{i_1}s_{i_2}\cdots s_{i_m}$, the associativity of $\star$ implies that we may evaluate $\delta(\alpha)$ by picking a left justified parenthesizing $(\cdots((s_{i_1}\star s_{i_2})\star s_{i_3})\star\cdots)\star s_{i_m}$. By the inductive definition of $\star$, this picks out the ``leftmost'' reduced subword $\beta$ of $\alpha$ whose product is $\delta(\alpha)$. We will refer to $\beta$ as the \textit{greedy Demazure subword}. 
\end{rem}

 Using \Cref{prop:fence_faces}, we will denote by $F$ both a fence diagram and the corresponding face of $\O(P_{k,n})$. Given a face $F$ of a polyhedral subcomplex of a polytope, we let $\int\F$ denote the interior of $F$, viewing it as a closed ball. This is the same as $F$ with all the proper faces of $F$ removed. Similarly, we denote by $\int \calF_u^w$ the interior of $\calF_u^w$, which is well defined as it is homeomorphic to a closed ball. 

In this section, we will prove the following two theorems:

\begin{thm}\label[theorem]{dem prod characterization} 
    Assume a face $F$ is of shape $\pi_k(w)/\pi_k(u)$. Then $F\subset \calF_u^w$ if and only if $\delta(\w_F) \geq u_Jw_J^{-1}$. 
\end{thm}
 
\begin{thm}\label[theorem]{interior of the fence complex} 
    Assume $F\subset \calF_u^w$ is of shape $\lambda/\mu$. Then $\int F\subset \int \calF_u^w$ if and only if $\lambda = \pi_k(w), \mu = \pi_k(u),$ and $\delta(\w_F) = u_Jw_J^{-1}$.
\end{thm}

We will then characterize the boundary of $\calF_u^w$, and finally conclude that the fence complexes on $\calO(P_{k,n})$ give a presentation of the totally nonnegative Grassmannian as a regular CW complex.

\subsection{Technical Lemmas}

In this subsection, we record two lemmas to be used in the proofs of \Cref{dem prod characterization} and \Cref{interior of the fence complex}.

\begin{lemma}\label[lemma]{AAL} 
Let $F$ be a face of $\mathcal{O}(P_{k,n})$ of shape $\pi_k(w)/\pi_k(u)$ with $\delta(\w_F) = v$. Then, there exists a facet $G$ of $\calF_{\pi_k(u)v}^{\pi_k(w)}$ containing $F$.
\end{lemma}
\begin{proof}
    Per \Cref{rem:greedy_demazure}, let $\beta$ denote the greedy Demazure subword of $word_F$. Only drawing the fences of the fence diagram of $F$ which correspond to the letters of $\beta$ gives an SDSS which we will call $G$. We claim that $G$ has no bad configurations, and hence is a fence diagram. By definition it is reduced and hence corresponds to an element of $\mathfrak{F}_{\pi_k(u)v}^{\pi_k(w)}$, so this will imply that $G$ corresponds to a facet of $\calF_{\pi_k(u)v}^{\pi_k(w)}$ by \Cref{fence complex is pure dimensional}. 

    Suppose for the sake of contradiction that $G$ contains a bad configuration. Let $\alpha$ denote an order-minimal such configuration, i.e. one such that there is no bad configuration weakly top left of $\alpha$. By transposing if necessary, we may assume that $\alpha$ is an [L], joining $(i, j)$ to $(i+1, j)$ to $(i+1, j+1)$.

    Define $\row^{(t)}_{G}$ for $1\le t\le k$ in the same way as one does for a fence diagram, and let $r\ge 0$ be the largest integer such that $\row^{(i+1)}_{G}$ contains $s_{j+r}s_{j+r-1}\cdots s_{j}$. Then $(i+1, j), (i+1, j+1),\dots, (i+1, j+r+1)$ are in the same component of $G$ as $(i,j)$, and hence the same is true for components of $F$. Since $F$ is a fence diagram, this forces the component of $(i,j)$ to contain the whole rectangle $[i, i+1]\times [j, j+r+1]$. In particular, $\row_F^{(i)}$ and $\row_F^{(i+1)}$ contain $s_{j+r}\cdots s_{j}$.
    
    Write $\row_G^{(i)} = x s_{m_1}\cdots s_{m_\ell} y$, a decreasing word with $j+r\ge m_1\ge\cdots \ge m_l\geq  j-1$ and $x,y$ decreasing subwords containing no $s_i$ for $i\in [j-1,j+r]$. Similarly, write $\row_G^{(i+1)} = x's_{j+r}\cdots s_jy'$ for some decreasing subwords $x', y'$. Note that $x'$ does not contain $s_{j+r+1}$ by the choice of $r$. We show that $m_1\le j+r-1$ and $m_\ell\ge j$. To see the latter, observe that if $\row_G^{(i)}$ contains $s_{j-1}$, then $G$ has a component connecting $(i,j-1)$ to $(i,j)$ to $(i,j+1)$. This contradicts the choice of $\alpha$.
    
    That $m_1\le j+r-1$ is slightly more involved. If $\row_G^{(i)}$ contains $s_{j+r}$, find the smallest $t\le r$ such that $\row_G^{(i)}$ contains $s_{j+r}s_{j+r-1}\cdots s_{j+t}$. Then this subword can be commuted to the right so that $\row_G^{(i)}$ is commutation braid move equivalent to $xys_{m_{\ell'}}\cdots s_{m_\ell}s_{j+r}\cdots s_{j+t}$ for some $\ell'\le \ell$. Similarly, $\row_G^{(i+1)}$ is commutation braid move equivalent to $s_{j+r}\cdots s_jx'y'$. This forces $\row_G$ to contain the non-reduced subword $(s_{j+r}s_{j+r-1}\cdots s_{j+t})(s_{j+r}s_{j+r-1}\cdots s_{j+t})$. This contradicts the fact that $\row_G$ is the reduced subword of $\row_F$ coming from the greedy Demazure product algorithm.

    Now we define a new diagram $G'$ by the following transformation (pictured in \Cref{fig:AAL_transformation_row}) from $G$: in the region between columns $j$ and $j+r$, shift all the fences in row $i$ to the right by $1$ unit, and then swap rows $i$ and $i+1$.
    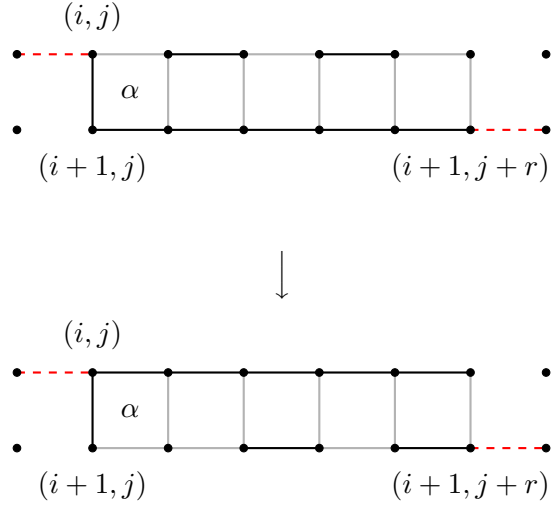
\begin{figure}
        \centering
        $\begin{array}{ccccc}

% Row 1
\vcenter{\hbox{
\begin{tikzpicture}[baseline={([yshift=-.7ex]current bounding box.center)}]
  \node [filled] at (0,0) {};
  \node [filled] at (1,0) {};
  \node [filled] at (2,0) {};
  \node [filled] at (3,0) {};
  \node [filled] at (4,0) {};
  \node [filled] at (5,0) {};
  \node [filled] at (6,0) {};
  \node [filled] at (7,0) {};
  
  \node [filled] at (0,1) {};
  \node [filled] at (1,1) {};
  \node [filled] at (2,1) {};
  \node [filled] at (3,1) {};
  \node [filled] at (4,1) {};
  \node [filled] at (5,1) {};
  \node [filled] at (6,1) {};
  \node [filled] at (7,1) {};

  \node [blank] at (1.5,0.5) {$\alpha$};

\begin{pgfonlayer}{bg}
  \node [blank] at (1,1.5) {$(i,j)$};
  \node [blank] at (1,-0.5) {$(i+1,j)$};
  \node [blank] at (6,-0.5) {$(i+1,j+r)$};
\end{pgfonlayer}

  \draw [thick] (1,0) to (1,1);
  \draw [thick] (1,0) to (2,0);
  \draw [thick] (2,0) to (3,0);
  \draw [thick] (3,0) to (4,0);
  \draw [thick] (4,0) to (5,0);
  \draw [thick] (5,0) to (6,0);

  \draw [thick] (2,1) to (3,1);
  \draw [thick] (4,1) to (5,1);

\begin{pgfonlayer}{bg}
  \draw [thick, color=red, dashed] (0,1) to (1,1);
  \draw [thick, color=red, dashed] (6,0) to (7,0);

  \draw [thick, color=grey] (1,1) to (2,1);
  \draw [thick, color=grey] (3,1) to (4,1);
  \draw [thick, color=grey] (5,1) to (6,1);
  \draw [thick, color=grey] (2,1) to (2,0);
  \draw [thick, color=grey] (3,1) to (3,0);
  \draw [thick, color=grey] (4,1) to (4,0);
  \draw [thick, color=grey] (5,1) to (5,0);
  \draw [thick, color=grey] (6,1) to (6,0);
\end{pgfonlayer}
  
\end{tikzpicture}
}}\\

\vcenter{\hbox{
$\Big\downarrow$
}}
\\
\vcenter{\hbox{
\begin{tikzpicture}[baseline={([yshift=-.7ex]current bounding box.center)}]
  \node [filled] at (0,0) {};
  \node [filled] at (1,0) {};
  \node [filled] at (2,0) {};
  \node [filled] at (3,0) {};
  \node [filled] at (4,0) {};
  \node [filled] at (5,0) {};
  \node [filled] at (6,0) {};
  \node [filled] at (7,0) {};
  
  \node [filled] at (0,1) {};
  \node [filled] at (1,1) {};
  \node [filled] at (2,1) {};
  \node [filled] at (3,1) {};
  \node [filled] at (4,1) {};
  \node [filled] at (5,1) {};
  \node [filled] at (6,1) {};
  \node [filled] at (7,1) {};

  \node [blank] at (1.5,0.5) {$\alpha$};

\begin{pgfonlayer}{bg}
  \node [blank] at (1,1.5) {$(i,j)$};
  \node [blank] at (1,-0.5) {$(i+1,j)$};
  \node [blank] at (6,-0.5) {$(i+1,j+r)$};
\end{pgfonlayer}

\begin{pgfonlayer}{bg}
  \draw [thick] (1,0) to (1,1);
  \draw [thick, color=grey] (1,0) to (2,0);
  \draw [thick, color=grey] (2,0) to (3,0);
  \draw [thick] (3,0) to (4,0);
  \draw [thick, color=grey] (4,0) to (5,0);
  \draw [thick] (5,0) to (6,0);

  \draw [thick] (2,1) to (3,1);
  \draw [thick] (4,1) to (5,1);

  \draw [thick, color=red, dashed] (0,1) to (1,1);
  \draw [thick, color=red, dashed] (6,0) to (7,0);

  \draw [thick] (1,1) to (2,1);
  \draw [thick] (3,1) to (4,1);
  \draw [thick] (5,1) to (6,1);
  \draw [thick, color=grey] (2,1) to (2,0);
  \draw [thick, color=grey] (3,1) to (3,0);
  \draw [thick, color=grey] (4,1) to (4,0);
  \draw [thick, color=grey] (5,1) to (5,0);
  \draw [thick, color=grey] (6,1) to (6,0);
\end{pgfonlayer}
  
\end{tikzpicture}
}}
\end{array}$
        \caption{Local transformation to go from $G$ (top) to $G'$ (bottom). Fences of $G$ and $G'$ are drawn in black, whereas fences of $F$ are drawn in gray. Red dashed edges indicate absence of fences in $G$ and $G'$. If there is no edge, it does not matter to the argument whether a fence is there or not.}
        \label{fig:AAL_transformation_row}
    \end{figure}

    Writing $\row_G^{(i)} = xs_{m_1}\cdots s_{m_\ell}y$ and $\row_G^{(i+1)} = x' s_{j+r}\cdots s_jy'$ as before, we set
    \[\row_{G'}^{(i)} := xs_{j+r}\cdots s_{j},\]
    \[\row_{G'}^{(i+1)} := x' s_{m_1+1}\cdots s_{m_\ell+1}y'.\]

    Then we have the following chain of Coxeter move equivalences:
    \begin{align*}
        \row_G^{(i)}\row_G^{(i+1)} &= (xs_{m_1}\cdots s_{m_\ell}y)(x' s_{j+r}\cdots s_jy') \\
        &\equiv xx'(s_{m_1}\cdots s_{m_\ell})(s_{j+r}\cdots s_j)yy'\\
        &\equiv xx'(s_{j+r}\cdots s_j)(s_{m_1+1}\cdots s_{m_\ell+1})yy'\\
        &\equiv (xs_{j+r}\cdots s_jy)(x's_{m_1+1}\cdots s_{m_\ell+1}y')\\
        &= \row_{G'}^{(i)}\row_{G'}^{(i+1)}.
    \end{align*}

    Since $G'$ differs from $G$ only in rows $i$ and $i+1$, this shows that $\row_G \equiv \row_{G'}$ and trivially, $\col_G \equiv \col_{G'}$ (as we don't change the column word at all). Therefore $\w_{G'}$ is another reduced word for $v$, and since it only uses fences that belong to $F$ it is a subword of $\w_F$. However, $\w_{G'}$ contains the letters $s_{j+r}$ in $\row_F^{(i)}$, whereas $\w_G$ does not. Moreover, these words coincide up until that point, contradicting the choice of $\w_G$ as the greedy reduced subword for $v$. This contradicts the existence of $\alpha$, and so $G$ corresponds to a facet of $\F_{\pi_k(u)v}^{\pi_k(w)}$.
    
\end{proof}

    \begin{lemma}\label[lemma]{lem:mario}
        Let $F\subset \F_u^w$ be a facet, and let $G\subset F$ be a codimension one face of $F$ with the same skew shape as $F$. If $\delta(\w_{G}) = \delta(\w_F)$, then $G \subset \int\F_u^w$.
    \end{lemma}
    \begin{proof}
        We know that $\F_u^w$ triangulates a ball, so the codimension $1$ face $G\subset F$ belongs to the interior of $\F_u^w$ if and only if it belongs to two facets of $\F_u^w$. Thus it suffices to find a facet $F'\in \calF_u^w$ different from $F$ that also contains $G$.
        
        $G$ has codimension $1$ in $F$ while sharing the same skew shape, so $G$ is obtained from $F$ by adding a single fence and performing the fence closure operation to obtain a genuine fence diagram.\\[-5pt]
\begin{itemize}
        \item \textbf{Case 1}: Suppose we are in the easier case where $G$ is also border strip, in the sense of \Cref{border strip SDSS}.

        Let $\w_F = xy$ and $\w_{G} = xs_i y$ for some words $x, y$ such that $xy$ is reduced. We analyze subcases according to whether $\dem(xs_i) = \dem(x)$ or $\dem(xs_i)>\dem (x)$.

        \begin{enumerate}
            \item If $\dem(xs_i) = \dem(x)$, then $s_i$ is a right descent of $x$. By deletion, there is a letter $s_j$ in $x$ such that $x's_i$ is a reduced word equivalent to $x$ (where $x'$ is obtained by deleting $s_j$ from $x$).  Observe deleting the fence corresponding to this $s_j$ from $G$ gives an SDSS $F'$ with $\w_{F'}$ reduced. By \Cref{lem:no_box}, it follows that $F'$ is a reduced fence diagram and hence corresponds to a facet of $\F_u^w$ different from $F$ containing $G$.

            \item If $\dem(xs_i) > \dem(x)$, then $xs_i$ is already reduced. Since $\dem (xs_iy) = \dem(xy)$, the greedy algorithm for computing the Demazure product of $\dem(xs_i y)$ ensures that there must be some letter $s_j$ of $y$ that we omit when computing $\delta(xs_iy)$. Deleting the fence corresponding to this $s_j$ from $G$ gives the required facet.\\[.5pt] 
        \end{enumerate}

       \item  \textbf{Case 2}: Now, suppose $G$ is not border strip, so the fence components of $G$ contain some boxes. In particular, we get $G$ by setting two adjacent components (i.e., ones which has segments passing through adjacent cells) of $F$, $A$ and $B$, equal to each other. Since $F$ is reduced, $A$ and $B$ are in fact border strip components.
        
        By the \emph{track} $T$ between $A$ and $B$, we mean the union of the fences in $G$ bounding boxes. Then $T$ is a union of boxes forming a path from the bottom left to top right with rightward or upward steps (since both $A$ and $B$ are border strips). Considering $A$ and $B$ as fence diagrams, we can define $\w_A$ and $\w_B$. For convenience of notation, let $\w_T:=\w_A\w_B$, and similarly for $\col_T$ and $\row_T$. \Cref{lem:word_fence_component} then shows that $\w_F\equiv m\ \w_T\ m'$ for some reduced words $m,m'$. In particular $\w_T$ is a reduced word (equivalently, both $\col_T$ and $\row_T$ are reduced).

        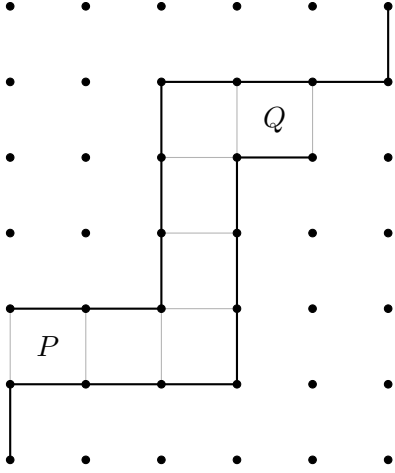
\begin{figure}
            \centering
            \begin{tikzpicture}[baseline={([yshift=-.7ex]current bounding box.center)}]
  \node [filled] at (0,0) {};
  \node [filled] at (1,0) {};
  \node [filled] at (2,0) {};
  \node [filled] at (3,0) {};
  \node [filled] at (4,0) {};
  \node [filled] at (5,0) {};
  \node [filled] at (0,1) {};
  \node [filled] at (1,1) {};
  \node [filled] at (2,1) {};
  \node [filled] at (3,1) {};
  \node [filled] at (4,1) {};
  \node [filled] at (5,1) {};
  \node [filled] at (0,2) {};
  \node [filled] at (1,2) {};
  \node [filled] at (2,2) {};
  \node [filled] at (3,2) {};
  \node [filled] at (4,2) {};
  \node [filled] at (5,2) {};
  \node [filled] at (0,3) {};
  \node [filled] at (1,3) {};
  \node [filled] at (2,3) {};
  \node [filled] at (3,3) {};
  \node [filled] at (4,3) {};
  \node [filled] at (5,3) {};
  \node [filled] at (0,4) {};
  \node [filled] at (1,4) {};
  \node [filled] at (2,4) {};
  \node [filled] at (3,4) {};
  \node [filled] at (4,4) {};
  \node [filled] at (5,4) {};
  \node [filled] at (0,5) {};
  \node [filled] at (1,5) {};
  \node [filled] at (2,5) {};
  \node [filled] at (3,5) {};
  \node [filled] at (4,5) {};
  \node [filled] at (5,5) {};
  \node [filled] at (0,6) {};
  \node [filled] at (1,6) {};
  \node [filled] at (2,6) {};
  \node [filled] at (3,6) {};
  \node [filled] at (4,6) {};
  \node [filled] at (5,6) {};
  \node [blank] at (0.5, 1.5) {{$P$}};
  \node [blank] at (3.5, 4.5) {{$Q$}};
  \draw [thick] (0,0) to (0,1);
  \draw [thick] (0,1) to (1,1);
  \draw [thick] (1,1) to (2,1);
  \draw [thick] (2,1) to (3,1);
  \draw [thick] (3,1) to (3,2);
  \draw [thick] (3,2) to (3,3);
  \draw [thick] (3,3) to (3,4);
  \draw [thick] (3,4) to (4,4);

  \draw [thick] (0,2) to (1,2);
  \draw [thick] (1,2) to (2,2);
  \draw [thick] (2,2) to (2,3);
  \draw [thick] (2,3) to (2,4);
  \draw [thick] (2,4) to (2,5);
  \draw [thick] (2,5) to (3,5);
  \draw [thick] (3,5) to (4,5);
  \draw [thick] (4,5) to (5,5);
  \draw [thick] (5,5) to (5,6);

  \begin{pgfonlayer}{bg}
  
  \draw [color=grey] (0,1) to (0,2);
  \draw [color=grey] (1,1) to (1,2);
  \draw [color=grey] (2,1) to (2,2);
  \draw [color=grey] (2,2) to (3,2);
  \draw [color=grey] (2,3) to (3,3);
  \draw [color=grey] (2,4) to (3,4);
  \draw [color=grey] (3,4) to (3,5);
  \draw [color=grey] (4,4) to (4,5);
  
  \end{pgfonlayer}

\end{tikzpicture}
            \caption{Example of a track, with horizontal start and horizontal end. Fences in black belong to both $F$ and $G$, whereas fences in gray belong only to $G$.}
            \label{fig:track}
        \end{figure}        
        Denote by $P$ and $Q$ the bottom leftmost and top rightmost boxes enclosed by $T$, respectively, as in \Cref{fig:track}. Being codimension $1$ forces additional conditions on $P$ and $Q$. Every box in $T$ has vertices in $A\cup B$, or else the codimension of $G\subset F$ would be higher than $1$. Moreover, not all four vertices of a box can belong to $A$ (or $B$), since $F$ is border strip.
        
        Say that the start (resp. end) of $T$ is \emph{vertical} or \emph{horizontal} based on whether the configuration of fences in $F$ bounding $P$ (resp. $Q$) looks like \Cref{fig:vertical_track} or \Cref{fig:horizontal_track}.

        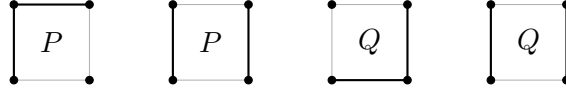
\begin{figure}
            \centering
            $\begin{array}{ccccc}

% Row 1
\vcenter{\hbox{
\begin{tikzpicture}[baseline={([yshift=-.7ex]current bounding box.center)}]
  \node [filled] (a) at (0,0) {};
  \node [filled] (b) at (1,0) {};
  \node [filled] (c) at (0,1) {};
  \node [filled] (d) at (1,1) {};
  \node [blank] at (0.5, 0.5) {$P$};
\begin{pgfonlayer}{bg}
  \draw [thick] (0,0) to (0,1);
  \draw [thick] (0,1) to (1,1);
  \draw [color=grey] (0,0) to (1,0);
  \draw [color=grey] (1,0) to (1,1);    
\end{pgfonlayer}

\end{tikzpicture}
}}

&\quad

\vcenter{\hbox{
\begin{tikzpicture}[baseline={([yshift=-.7ex]current bounding box.center)}]
  \node [filled] (a) at (0,0) {};
  \node [filled] (b) at (1,0) {};
  \node [filled] (c) at (0,1) {};
  \node [filled] (d) at (1,1) {};
  \node [blank] at (0.5, 0.5) {$P$};
\begin{pgfonlayer}{bg}
  \draw [thick] (0,0) to (0,1);
  \draw [thick] (1,0) to (1,1);
  \draw [color=grey] (0,0) to (1,0);
  \draw [color=grey] (0,1) to (1,1);  
\end{pgfonlayer}

\end{tikzpicture}
}}
&\quad
\vcenter{\hbox{
\begin{tikzpicture}[baseline={([yshift=-.7ex]current bounding box.center)}]
  \node [filled] (a) at (0,0) {};
  \node [filled] (b) at (1,0) {};
  \node [filled] (c) at (0,1) {};
  \node [filled] (d) at (1,1) {};
  \node [blank] at (0.5, 0.5) {$Q$};
\begin{pgfonlayer}{bg}
  \draw [color=grey] (0,0) to (0,1);
  \draw [color=grey] (0,1) to (1,1);
  \draw [thick] (0,0) to (1,0);
  \draw [thick] (1,0) to (1,1);    
\end{pgfonlayer}

\end{tikzpicture}
}}

&\quad

\vcenter{\hbox{
\begin{tikzpicture}[baseline={([yshift=-.7ex]current bounding box.center)}]
  \node [filled] (a) at (0,0) {};
  \node [filled] (b) at (1,0) {};
  \node [filled] (c) at (0,1) {};
  \node [filled] (d) at (1,1) {};
  \node [blank] at (0.5, 0.5) {$Q$};
\begin{pgfonlayer}{bg}
  \draw [thick] (0,0) to (0,1);
  \draw [thick] (1,0) to (1,1);
  \draw [color=grey] (0,0) to (1,0);
  \draw [color=grey] (0,1) to (1,1);    
\end{pgfonlayer}

\end{tikzpicture}
}}\\[-5pt]

\end{array}$
            \caption{Vertical tracks. The first two are vertical starts, and the last two are vertical ends.}
            \label{fig:vertical_track}
        \end{figure}

        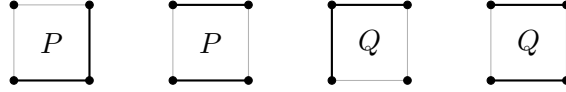
\begin{figure}
            \centering
            $\begin{array}{ccccc}

% Row 1
\vcenter{\hbox{
\begin{tikzpicture}[baseline={([yshift=-.7ex]current bounding box.center)}]
  \node [filled] (a) at (0,0) {};
  \node [filled] (b) at (1,0) {};
  \node [filled] (c) at (0,1) {};
  \node [filled] (d) at (1,1) {};
  \node [blank] at (0.5, 0.5) {$P$};
\begin{pgfonlayer}{bg}
  \draw [color=grey] (0,0) to (0,1);
  \draw [color=grey] (0,1) to (1,1);
  \draw [thick] (0,0) to (1,0);
  \draw [thick] (1,0) to (1,1);    
\end{pgfonlayer}

\end{tikzpicture}
}}

&\quad

\vcenter{\hbox{
\begin{tikzpicture}[baseline={([yshift=-.7ex]current bounding box.center)}]
  \node [filled] (a) at (0,0) {};
  \node [filled] (b) at (1,0) {};
  \node [filled] (c) at (0,1) {};
  \node [filled] (d) at (1,1) {};
  \node [blank] at (0.5, 0.5) {$P$};
\begin{pgfonlayer}{bg}
  \draw [color=grey] (0,0) to (0,1);
  \draw [color=grey] (1,0) to (1,1);
  \draw [thick] (0,0) to (1,0);
  \draw [thick] (0,1) to (1,1);  
\end{pgfonlayer}

\end{tikzpicture}
}}
&\quad
\vcenter{\hbox{
\begin{tikzpicture}[baseline={([yshift=-.7ex]current bounding box.center)}]
  \node [filled] (a) at (0,0) {};
  \node [filled] (b) at (1,0) {};
  \node [filled] (c) at (0,1) {};
  \node [filled] (d) at (1,1) {};
  \node [blank] at (0.5, 0.5) {$Q$};
\begin{pgfonlayer}{bg}
  \draw [thick] (0,0) to (0,1);
  \draw [thick] (0,1) to (1,1);
  \draw [color=grey] (0,0) to (1,0);
  \draw [color=grey] (1,0) to (1,1);    
\end{pgfonlayer}

\end{tikzpicture}
}}

&\quad

\vcenter{\hbox{
\begin{tikzpicture}[baseline={([yshift=-.7ex]current bounding box.center)}]
  \node [filled] (a) at (0,0) {};
  \node [filled] (b) at (1,0) {};
  \node [filled] (c) at (0,1) {};
  \node [filled] (d) at (1,1) {};
  \node [blank] at (0.5, 0.5) {$Q$};
\begin{pgfonlayer}{bg}
  \draw [color=grey] (0,0) to (0,1);
  \draw [color=grey] (1,0) to (1,1);
  \draw [thick] (0,0) to (1,0);
  \draw [thick] (0,1) to (1,1);    
\end{pgfonlayer}

\end{tikzpicture}
}}\\

\end{array}$
            \caption{Horizontal tracks, First two are horizontal starts, last two are horizontal ends.}
            \label{fig:horizontal_track}
        \end{figure}
        
        If both the start and end of $T$ are vertical, then $\col_T$ is move equivalent to the non-reduced word $s_is_{i+1}\cdots s_js_is_{i+1}\cdots s_j$ for some $i\le j$, which is a contradiction. Transposing and using \Cref{transpose word comparison} shows that both the start and end of $T$ cannot be horizontal either.
        
        This leaves us with the case where the start of $T$ is vertical and the end of $T$ is horizontal, or vice versa. By transposing if necessary, it suffices to consider the former case. We make the local transformations at $P$ and $Q$ shown in \Cref{fig:transformation} (reading left to right), and claim that this gives the required facet $F'$ containing $G$.

        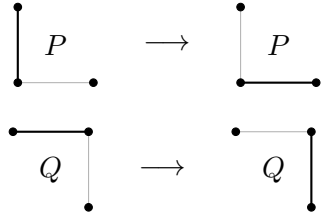
\begin{figure}[H]
            \centering
            $\begin{array}{ccccc}

% Row 1
\vcenter{\hbox{
\begin{tikzpicture}[baseline={([yshift=-.7ex]current bounding box.center)}]
  \node [filled] (a) at (0,0) {};
  \node [filled] (b) at (1,0) {};
  \node [filled] (c) at (0,1) {};
  \node [blank] at (0.5, 0.5) {$P$};
\begin{pgfonlayer}{bg}
  \draw [thick] (0,0) to (0,1);
  \draw [color=grey] (0,0) to (1,0);
\end{pgfonlayer}

\end{tikzpicture}
}}
&
\vcenter{\hbox{
$\longrightarrow$
}}
&
\vcenter{\hbox{
\begin{tikzpicture}[baseline={([yshift=-.7ex]current bounding box.center)}]
  \node [filled] (a) at (0,0) {};
  \node [filled] (b) at (1,0) {};
  \node [filled] (c) at (0,1) {};
  \node [blank] at (0.5, 0.5) {$P$};
\begin{pgfonlayer}{bg}
  \draw [color=grey] (0,0) to (0,1);
  \draw [thick] (0,0) to (1,0);    
\end{pgfonlayer}

\end{tikzpicture}
}}
\end{array}$

\vspace{5mm}
$\begin{array}{ccccc}

% Row 1
\vcenter{\hbox{
\begin{tikzpicture}[baseline={([yshift=-.7ex]current bounding box.center)}]
  \node [filled] (b) at (1,0) {};
  \node [filled] (c) at (0,1) {};
  \node [filled] (d) at (1,1) {};
  \node [blank] at (0.5, 0.5) {$Q$};
\begin{pgfonlayer}{bg}
  \draw [thick] (1,1) to (0,1);
  \draw [color=grey] (1,1) to (1,0);
\end{pgfonlayer}

\end{tikzpicture}
}}
&
\vcenter{\hbox{
$\longrightarrow$
}}
&
\vcenter{\hbox{
\begin{tikzpicture}[baseline={([yshift=-.7ex]current bounding box.center)}]
  \node [filled] (d) at (1,1) {};
  \node [filled] (b) at (1,0) {};
  \node [filled] (c) at (0,1) {};
  \node [blank] at (0.5, 0.5) {$Q$};
\begin{pgfonlayer}{bg}
  \draw [color=grey] (1,1) to (0,1);
  \draw [thick] (1,1) to (1,0);    
\end{pgfonlayer}

\end{tikzpicture}
}}
\end{array}$
            \caption{Local transformations converting $F$ to $F'$. The left column shows a track with vertical start and horizontal end, and the right column shows a track with horizontal start and vertical end.}
            \label{fig:transformation}
        \end{figure}
        
        Clearly $F'$ is a face of $\F_u^w$ containing $G$, so it suffices to show that $\w_{F'}\equiv \w_F$. Note applying the above transformation only changes the subwords $\col_T$ and $\row_T$ of $\w_F$, say to $\col_{T}'$ and $\row_{T}'$. $\col_T$ is commutation move equivalent to $s_{i_1}s_{i_1+1}\cdots s_{j_1}s_{i_1}{s_{i_1+1}}\cdots s_{j_1-1}$ for some $i_1\le j_1$. Then $\col_T'$ is commutation move equivalent to $s_{i_1+1}\cdots s_{j_1}s_{i_1}{s_{i_1+1}}\cdots s_{j_1}$, which by an elementary computation, is move equivalent to $\col_T$. Similarly, $\row_T$ is move equivalent to $s_{i_2}s_{i_2+1}\cdots s_{j_2}s_{i_2}{s_{i_2+1}}\cdots s_{j_2-1}$ for some $i_2\le j_2$, and $\row'_T$ is commutation move equivalent to $s_{i_2+1}\cdots s_{j_2}s_{i_2}{s_{i_2+1}}\cdots s_{j_2}$, which is move equivalent to $\row_T$. Thus $\w_{F'}$ is equivalent to $\w_F$, showing that $F'\in \calF_u^w$ is a facet different from $F$ containing $G$.
    \end{itemize} 
    \end{proof}

\subsection{Proofs of \Cref{dem prod characterization} and \Cref{interior of the fence complex}}
\begin{proof}[Proof of \Cref{dem prod characterization}]
    Let $F$ be of shape $\pi_k(w)/\pi_k(u)$, and assume $F\subset \calF_u^w$. Pick a facet $G$ of $\calF_u^w$ containing $F$. By \Cref{eq:gt_faces}, every fence in $G$ is also a fence in $F$. Hence $\w_G$ is a reduced subword of $\w_F$. Since $\w_G$ equals $u_Jw_J^{-1}$ as a permutation by definition, it follows that $\delta(\w_F)\ge u_Jw_J^{-1}$ as required.

    Conversely, suppose $\delta(\w_F)=v\ge u_Jw_J^{-1}$. By \Cref{AAL}, we obtain a facet $G\subset \calF_{\pi_k(u)v}^w$ containing $F$. Since $u_Jw_J^{-1}\le v=\w_G$, there exists some reduced subword of $\w_G$ that evaluates to $u_Jw_J^{-1}$. \Cref{rem:deleting fences from reduced fence diagrams} shows that drawing in only the fences of $G$ corresponding to this subword gives a fence diagram $G'$. Then $G'$ is a facet of $\calF_u^w$ and $F\subset G\subset G'$, so $F\subset \calF_u^w$.
\end{proof}

We prove a couple more easy lemmas before moving on to the proof of \Cref{interior of the fence complex}.

\begin{lemma}\label[lemma]{lem:bruhatcoversubwords} 
Let $x,y\in S_n$ such that $x\lessdot y$ in Bruhat order. If $m$ is a reduced word for $y$, then $x$ appears exactly once as a subword of $m$. 
\end{lemma}
\begin{proof}
    The fact that $x$ appears at least once is standard, see for instance \cite[Theorem 2.2.2]{Bjorner-Brenti}. Suppose $x$ appears twice. Then $m= as_ibs_jc$  where $a,b,c$ are words, and \[abs_jc = as_ibc = x.\]
    But then, $s_ibs_j = b$, and hence $y = abc$, contradicting the fact that $m$ is reduced. 
\end{proof}

\begin{lemma}\label[lemma]{lem:codim_1}
    Suppose $F\subset \calF_u^w$ is a codimension one face such that either \begin{enumerate}
        \item $F$ has shape $\pi_k(w)/\pi_k(u)$ and $\delta(\w_F)>u_J$, or
        \item $F$ has a smaller shape than $\pi_k(w)/\pi_k(u)$. 
    \end{enumerate} 
 Then, $F\subset\d\calF_u^w$. 
 \end{lemma} 
 \begin{proof}
    Since $\calF_u^w$ is homeomorphic to a closed ball, a codimension one face is contained in the boundary if and only if it is contained in exactly $1$ facet.
    
    In case (1), by \Cref{AAL}, $F$ is contained in a reduced face $F'$ with $\delta(\w_{F'})= \delta(\w_F)>u_J$. By \Cref{dem prod characterization}, we have $F'\subset \calF_u^w$. Moreover, $F'$ is not a facet of $\calF_u^w$, so since $F$ is codimension one, we must have $F=F'$, and hence $F$ is reduced. Let $G$ be a facet of $\calF_u^w$ containing $F$. Since $F$ is reduced, it is obtained from $G$ by adding a single fence. Thus, either $\col_G = \col_F$ or $\row_G=\row_F$.
    
    Without loss of generality, suppose $\row_G=\row_F$, i.e. that $\col_G$ is a proper subword of $\col_F$. This implies that for \textit{any} facet $G'$ of $\calF_u^w$ containing $F$, we have $\col_G = \col_{G'}$, and that both $\col_G$ and $\col_{G'}$ are given to us as reduced subwords of $\col_F$. But by \Cref{lem:bruhatcoversubwords}, this means that $\col_{G'}$ and $\col_G$ must be equal as subwords of $\col_F$. Therefore $G=G'$, and hence $F$ is on the boundary, as desired.
    
    In case (2), since $F$ is codimension $1$, we have that $F$ is contained in a facet $G$ of $\calF_u^w$ if and only if $F$ comes from setting a single fence component of $G$ equal to $0$ or 1 (corresponding to a minimal or maximal element of the quotient poset of $G$, respectively). But then one can recover the fence diagram of $G$ from the fence diagram of $F$, so $F$ is only contained in one facet. 
 \end{proof}

We may now prove \Cref{interior of the fence complex}. 
\begin{proof}[Proof of \Cref{interior of the fence complex}]
    Let $F\subset \calF_u^w$ be a face of shape $\lambda/\mu$. Necessarily, $\lambda/\mu$ is contained in the skew shape $\pi_k(w)/\pi_k(u)$. Note that $\int F\subset \int \calF_u^w$ if and only if $F\not\subset\d\calF_u^w$.

    Suppose first that $\int F\subset \int \calF_u^w$, i.e. that $F\not\subset\d\calF_u^w$. If $F$ is a facet of $\calF_u^w$ then $F$ has shape $\pi_k(w)/\pi_k(u)$ and $\delta(\w_F)=u_J$ by definition, and we are done. So, suppose instead that $F$ is a lower dimensional face, and pick a facet $G$ of $\calF_u^w$ containing it. If $\lambda/\mu$ is strictly smaller than $\pi_k(w)/\pi_k(u)$, then we construct a codimension $1$ face $G'\subset G$ containing $F$ such that the shape of $G'$ is also strictly smaller than $\pi_k(w)/\pi_k(u)$: if $\mu>\pi_k(u)$ then delete some minimal fence component of $G$, else if $\lambda<\pi_k(w)$ then delete some maximal fence component of $G$. By \Cref{lem:codim_1}, $F\subset G'\subset \d\calF_u^w$, contradicting the assumption on $F$. Thus, $\lambda = \pi_k(w)$ and $\mu = \pi_k(u)$.

    Assume now for the sake of contradiction that $\delta(\w_F)=v\ne u_Jw_J^{-1}$. \Cref{dem prod characterization} implies then that $v>u_Jw_J^{-1}$, so pick $v'$ such that $u_Jw_J^{-1}\lessdot v'\le v$. By \Cref{AAL}, $F$ is contained in a facet $H\subset \calF_{\pi_k(u)v}^{\pi_k(w)}$. Since $\w_H = v\ge v'$, $\w_H$ contains a reduced subword that evaluates to $v'$. Per \Cref{rem:deleting fences from reduced fence diagrams}, drawing in only the fences corresponding to this subword gives a fence diagram $H'$, which is a codimension $1$ face of $\calF_u^w$ containing $H$. Since $\delta(\w_{H'}) = v'>u_J$, it follows from \Cref{lem:codim_1} that $F\subset H'\subset \d\calF_u^w$, contradicting the assumption on $F$. Hence, $\delta(\w_F) = u_Jw_J^{-1}$.

    Conversely, suppose $\lambda=\pi_k(w), \mu=\pi_k(u)$, and $\delta(\w_F) = u_Jw_J^{-1}$. It suffices to show that $F\not\subset \d\calF_u^w$. Assume for the sake of contradiction that $F\subset \d\calF_u^w$, and pick a codimension $1$ face $G$ of $\calF_u^w$ such that $F\subset G\subset \d\calF_u^w$. By \Cref{lem:mario}, since $G$ also has shape $\pi_k(w)/\pi_k(u)$, we must have $\delta(\w_G)>u_Jw_J^{-1}$. However, every fence in $G$ is a fence in $F$, and hence $\w_G$ is a subword of $\w_F$. This implies that $\delta(\w_F)\ge \delta(\w_G)>u_Jw_J^{-1}$, a contradiction.
\end{proof}

\subsection{Fence Complexes as a Cell Decomposition of $Gr(k,n)_{\ge0}$}  
We now use the results from the previous section to give a presentation of $Gr(k,n)_{\geq 0}$ as a regular CW complex whose cells and attachments are parametrized by the poset of positroid strata, and where each cell is a polyhedral complex. We first prove some more results about the boundary of $\calF_u^w$.

\begin{prop}\label[proposition]{sidea: boundary containment}
    Fix $u'\leq_k w'$ such that $\Pi^{w'}_{u'} \subset \Pi^w_u$, i.e. $\langle u', w' \rangle \leq \langle u, w \rangle$ in $\mathcal{Q}(k,n)$. Then $\mathcal{F}^{w'}_{u'} \subset \mathcal{F}^w_u.$
\end{prop}

\begin{proof}
    First, observe that since fence complexes are independent of choice of representative of equivalence class in $\mathcal{Q}(k,n)$ (see \Cref{non-k-grass fence complex}), we may take $w$ to be $k$-Grassmannian.  It suffices to show the result when $\Pi^{w'}_{u'}$ is codimension one in $\Pi^w_u$. By the proof of \cite[Theorem 3.16]{KLS1} and \Cref{non-k-grass fence complex}, this means that we have three cases: either ($u = u'$ and $w' \lessdot_k w$), ($u \lessdot_k u'$ and $w' = w$), or ($u \lessdot u'$ and $w' = w$ with $\pi_k(u)= \pi_k(u')$). Observe that since $w$ is $k$-Grassmannian, $w'\leq w$ implies $w'\leq_k w$.  We now treat these cases separately.
    
    \begin{itemize} 
    \item \textbf{Case 1}: %%We recall that the faces of the Hibi polytope of $P_{k,n}$ which are contained in a given fence complex $\mathcal{F}^w_u$ are exactly those which arise from a collection of equality conditions in a fence diagram for $\mathcal{F}^w_u$ (including the implicit filling of all boxes outside of the underlying Young diagram $\pi_k(w) / \pi_k(u)$ of the fence diagram with $n$'s in the bottom right/$0$'s in the top left). Further, a face $F'$ of the Hibi polytope is contained in another face $F$ precisely when $F$'s equality conditions are a subset of $F'$'s.\\[.5pt]
    Assume we are given a $k$-Bruhat chain $C' := u = u_1 <_k ... <_k u_j = w'$ in $[u, w']$. Since $w' \leq_k w$, this extends to a $k$-Bruhat chain $C := u = u_1 <_k ... <_k u_j <_k w$ in $[u, w]$. Consider the fence diagrams $F_{C'}$ and $F_C$ associated to $C'$ and $C$, which live in $\mathcal{F}^{w'}_u$ and $\mathcal{F}^w_u$, respectively. Then $F_{C'}$ is just obtained from $F_C$ by deleting some bottom right fence component (i.e., imposing the condition  that the corresponding face is filled with $1$'s), so we have that the face corresponding to $F_{C'}$ is contained in the face corresponding to $F_C$. But this is true for all chains $C' \subset [u, w']$, and so $\mathcal{F}^{w'}_u \subset \mathcal{F}^w_u$.

    \item \textbf{Case 2}: Since $u \lessdot_k u'$, any chain $C' := u' = u_1 <_k ... <_k u_j = w$ extends to one of the form $C := u <_k u_1 <_k ... <_k u_j = w$. This implies again that the equalities in $F_{C'}$ are a subset of the equalities in $F_C$: more precisely, $F_{C'}$ is obtained from $F_C$ by deleting some top left fence component (i.e., imposing the condition that it is filled with $0$'s). Therefore $\mathcal{F}^w_{u'} \subset \mathcal{F}^w_u$ by the same argument as above.

    \item \textbf{Case 3}: Finally, if $u \leq u'$ and $u \nless_k u'$ (i.e. $\pi_k(u) = \pi_k(u')$), then the facets of $\mathcal{F}^w_{u'}$ have underlying skew-shape $\pi_k(w)/\pi_k(u)$ and Demazure product at least $u'_J$, which is greater than $u_J$ by assumption. Therefore \Cref{dem prod characterization} implies that every such facet is contained in $\mathcal{F}^w_u$, and so $\mathcal{F}^w_{u'}\subset \calF_u^w$.  
    \end{itemize} 
    \end{proof}

    \begin{prop}\label[proposition]{sideb: boundary containment}
        Let $F \subset \partial \mathcal{F}^w_u$ be a face in the boundary of a fence complex. Then, there exists $\langle u', w'\rangle < \langle u, w\rangle$ in $\mathcal{Q}(k,n)$ such that $F \subset \mathcal{F}^{w'}_{u'}$.
    \end{prop}

    \begin{proof}
     Assume first that $F$ is of shape $\pi_k(w)/\pi_k(u)$. Then, 
     \Cref{dem prod characterization} implies that $\delta(\text{word}_F) > u_Jw_J^{-1}$. In particular, setting $u' = \pi_k(u)\delta(\w_F)$, this means that $F$ is contained in the fence complex $\mathcal{F}^w_{u'}$. Then $\langle u', w\rangle < \langle u, w\rangle$ by definition.

     When $F$ is not of shape $\pi_k(w)/\pi_k(u)$, observe that it suffices  to check the case where $F$ is a codimension $1$ face. Therefore we may assume that $F$ is either of shape $\pi_k(w)/\mu$ with $\mu > \pi_k(u)$ or shape $\lambda/\pi_k(u)$ with $\lambda < \pi_k(w)$. Let $G$ be a facet of $\calF_u^w$ containing $F$.
     
     Say first that $F$ is of shape $\pi_k(w)/\mu$. Then, $F$ is cut out by the same inequalities and equalities as $G$, except some top left fence component $A$ of $G$ is set to $0$. Take a fence tableau of $G$ with cells in $A$ filled with $1$'s. This corresponds to a $k$-Bruhat chain $u=u_0\lessdot_k u_1\lessdot_k \cdots \lessdot_k u_{m-1} \lessdot_k u_m= w$. Note that the $k$-Bruhat chain $u_1\lessdot_k \cdots \lessdot_k u_{m-1} \lessdot_k u_m= w$ projects to a fence tableau with associated fence diagram $F$. Thus, $F\subset \mathcal{F}_{u_1}^w$. Since $[u_1,w]\subset [u,w]$, we have $\Pi_{u_1}^w \subset \Pi_u^w$. 

    If instead $F$ is of shape $\lambda/\pi_k(u)$ for some $\lambda<\pi_k(w)$, then $F$ is cut out by the same inequalities and equalities as $G$, except some bottom right fence component $B$ of $G$ is set to $1$. Take a fence tableau of $G$ with cells in $B$ filled with the value $|G|$. This corresponds to a $k$-Bruhat chain $u=u_0\lessdot_k u_1\lessdot_k \cdots \lessdot_k u_{m-1} \lessdot_k w$. Deleting the component $B$ gives a fence tableau for $F$, which corresponds to a $k$-Bruhat chain $u=u_0\lessdot_k u_1'\lessdot_k\cdots \lessdot_k u_{m-1}' = w'$. It follows that $F\subset \mathcal{F}_{u}^{w'}$. Since $[u,w']\subset [u,w]$ we have $\Pi_u^{w'}\subset \Pi_u^w$. 
    \end{proof}

    \begin{prop}\label[proposition]{boundary of the fence complex} 
Let $\mathcal{F}_u^w$ be a fence complex, considered as a polyhedral subcomplex of $\mathcal{O}(P_{k,n})$. Then $\partial \mathcal{F}_u^w$ is a union of lower dimensional fence complexes. In particular, $\partial \mathcal{F}_u^w$ is the union of all $\calF_{u'}^{w'}$ such that $\Pi_{u'}^{w'}\subset \Pi_u^w$.  
\end{prop}

\begin{proof}
    This follows from  \Cref{sidea: boundary containment} and \Cref{sideb: boundary containment}. 
\end{proof}

\begin{prop}\label[proposition]{one face left}
    Let $F$ be a face of the order polytope of $\mathcal{O}(P_{k,n})$. Then $\int F\subset \int\calF_u^w$ for precisely one fence complex $\calF_u^w$. 
\end{prop}

\begin{proof}
    By \Cref{prop:fence_faces}, $F$ corresponds to a fence diagram with shape $\lambda/\mu$ consisting of all of the cells $(i,j)$ such that $F$ contains some point with $x_{i,j}\notin\{0,1\}$. By \Cref{interior of the fence complex}, we have $\int F \subset \int\mathcal{F}^w_u$ if and only if the corresponding fence diagram satisfies $\pi_k(w) = \lambda$, $\pi_k(u) = \mu$, and $\delta(\w_F) = u_Jw_J^{-1}$. But $\pi_k(u),\pi_k(w)$ and $u_Jw_J^{-1}$ uniquely determine the class of $[u,w]_k$ in $\mathcal{Q}(k,n)$. 
\end{proof}

Since fence complexes are homeomorphic to closed balls by \Cref{fence complexes are homeomorphic to balls}, it follows from the above discussion that the fence complexes give $\calO(P_{k,n})$ the structure of a regular CW complex. Furthermore, the cell closure poset of this regular CW complex is the same as the stratification poset of positroid varieties. Galashin, Karp, and Lam showed in \cite{Galashin-karp-Lam} that the decomposition of $Gr(k,n)_{\geq 0}$ into positroid cells forms a regular CW complex. Since regular CW complexes are determined by their cell closure posets, we have shown the following theorem.

\begin{thm}\label[theorem]{regular CW complex structure} 
Each $\calF_u^w$ is homeomorphic to a closed ball, and the fence complexes $\calF_u^w$ give $P_{k,n}$ the structure of a regular CW complex with cell poset the same as $\Gr(k,n)_{\geq 0}$. In particular, $\Gr(k,n)_{\geq 0}$ and $P_{k,n}$ with the fence complex structure are isomorphic as regular CW-complexes.   
\end{thm} 

Our connection between the Ehrhart functions of $\calF_u^w$ and the Hilbert functions of $\Pi_u^w$ implies the following curious identity for positroid Hilbert functions. 
\begin{cor}
\[H_{\Pi_u^w}(n) =  \sum_{\Pi\subseteq\Pi_u^w} (-1)^{\dim \Pi} H_{\Pi}(-n),\] 
where the sum is over all positroid varieties $\Pi\subseteq\Pi_u^w$, including $\Pi_u^w$ itself. 
\end{cor}
\begin{proof}
 Since each point of $\calF_u^w$ is contained in the interior of precisely one fence complex, we have the following statement for Ehrhart polynomials, namely 
 \[\ehr_{\calF_u^w}(n) = \sum_{\calF} \intehr_\calF(n),\] where the sum is over all fence complexes $\calF\subset \calF_u^w$. Then, applying \Cref{mainresultC} gives the result. 
\end{proof}

\section{Background on Gr{\"o}bner Degenerations}\label[section]{Background on Grobner Degenerations}
In this section, we review some background on previously constructed degenerations of $\Gr(k,n)$ and $\Pi_u^w$. Let $S=\bbC[p_I]_{I\in {[n]\choose k}}$ be the polynomial ring with variables indexed by subsets $I\in {[n]\choose k}$. 

\subsection{Generalities on Gr{\"o}bner Degenerations and Standard Monomials}

Recall that via the Plücker embedding, we can view $\Gr(k,n)\subset \bbP^{{[n]\choose k}-1}$  as a projective variety with homogeneous coordinate ring $R=S/I_{k,n}$. The ideal $I_{k,n}$ is given by straightening relations of the form: 
\[p_Ip_J = p_{I\wedge J} p_{I\vee J} +\sum_{A<I\wedge J, B>I\vee J}c_{A,B} p_Ap_B,\] 
where $c_{A,B}\in \bbC$. Here, $I\wedge J$ and $I\vee J$ represent meet and join in Bruhat order on ${[n]\choose k}$.  For a full account of these relations, originally due to Hodge, see \cite{deConcini-Eisenbud-Procesi}, and see \cite[Lemma 7.15]{Gonciulea-Lakshmibai} for the fact that you can restrict the terms to $A<I\wedge J, B>I\vee J$.

It is known (from \cite{Sturmfels-degen} in type A, and in all Lie types from \cite{Gonciulea-Lakshmibai}) that there exists a weight degeneration of $R$  to the homogeneous coordinate ring of a toric variety. In particular, to each Plücker coordinate $p_I$, for $I=\{i_1<\cdots< i_k\}$, we assign the integral weight $\lambda(I) = i_1N^{k-1}+i_2N^{k-2}+\cdots +i_{k-1}N+i_k,$ where $N$ is some very large number. This induces a flat degeneration of $S/I_{k,n}$. 

In particular, recall (following \cite[Chapter 15.8]{EisenbudCAbook}) that for any integral weight function $\lambda$ on the polynomial ring $S$, and $g\in S$, we may define $\tilde{g}$ and $\initial_\lambda(g)$ as follows. Set $m=\max_{I} \lambda(I)$ where the maximum ranges over all $I\in {[n]\choose k}$, and write 
\[\tilde{g} := t^m g\left(t^{-\lambda(I)}p_I\biggm\lvert I\in {[n]\choose k}\right).\]
We can then write $\initial_\lambda(g) := t^m g\left(t^{-\lambda(I)}p_I\biggm\lvert I\in {[n]\choose k}\right)$ as the image of $\tilde{g}\otimes 1$ in $\tilde{I}\otimes \bbC[t]/(t)$ under the natural map to $S$. In other words, $\tilde g$ is obtained from $g$ by replacing each variable $p_I$ with $t^{-\lambda(I)}p_I$, and then multiplying by $t^m$, and $\initial_\lambda(g)$ can be thought of as setting $t=0$ in $\tilde{g}$.

We then define 
\[\tilde{I}_{k,n}%%\initial_\lambda(I_{k,n})
= \left\langle \tilde{g}\mid  g\in I_{k,n}\right\rangle \subset S[t]. \] 

Then, we have a flat family $S[t]/\tilde{I}$, where $S[t]/\tilde{I}_{k,n} \otimes_{\bbC[t]} \bbC[t,t^{-1}]\cong S/I_{k,n} [t,t^{-1}]$, and $S[t]/\tilde{I}_{k,n}\otimes_{\bbC[t]} \bbC[t]/(t)\cong S/\initial_\lambda(I_{k,n})$. Here,  
$\initial_\lambda(I_{k,n})$ is the image of $\tilde I_{k,n}\otimes_{\bbC[t]}\bbC[t]/(t)$ in $S$. Taking Proj, this corresponds to a flat degeneration of the Grassmannian. 
\begin{defn}
    Let $\psi: \mathscr{X}\to \bbC$ be the flat family defined by $\textrm{Proj} (S[t]/\tilde{I}_{k,n})$. 
\end{defn}
Observe that $\psi^{-1}(\bbC-\{0\})$ is simply the trivial family $\Gr(k,n)\times \bbC-\{0\}$, and $\psi^{-1}(0)$ is a toric variety: 
\begin{thm}[Sturmfels]
\label[theorem]{Gr degeneration} 
    For any integral weight function $\lambda$ on $S$ and $N >> 0$, we have $\initial_\lambda(I_{k,n})$ is the homogeneous ideal of $X(P_{k,n})\subset \bbP_\bbC^{{n\choose k}-1}$. 
\end{thm}

Here, $X(P_{k,n})$ is the toric variety of the order polytope $P_{k,n}$ (embedded into projective space via the line bundle corresponding to $P_{k,n})$. We have the following explicit description of its homogeneous ideal,  $I(P_{k,n})$. 

\begin{prop}
    $I(P_{k,n})$ is generated by relations of the form $p_Ip_J - p_{I\wedge J} p_{I\vee J}$ where $I$ and $J$ are incomparable in Bruhat order. 
\begin{proof}
Recall the relations of $I(P_{k,n})$ are generated by $p_Ip_J = p_Kp_L$ where $v_I+v_J = v_K+v_L$. This follows, for instance, from the fact that $P_{k,n}$ is a normal polytope and \cite[Proposition 4.8]{Eisenbud-Sturmfels}.  It is easy to check that $v_I+v_J = v_K+v_L$ if and only if $I\wedge J = K\wedge L$ and $I\vee J= K\vee L$. The result follows. 
\end{proof}
\end{prop}

In the upcoming sections, we will use the following alternative characterization of $\initial_\lambda$ for large enough $N$. 
\begin{defn}
    Let $I_1,\ldots, I_t, J_1,\ldots, J_t \in {[n]\choose k}$, with $I_r = \{i^r_1<\cdots i^r_k\}$ and $J_r =\{j_1^r<\cdots <j^r_k\}$. We say $p_{I_1}\cdots p_{I_t}$ is less than $p_{J_1}\cdots p_{J_t}$ in \emph{row sum lex} if we have  
    \[\left(\sum_{r=1}^t i_1^r, \sum_{r=1}^t i_2^r,\ldots, \sum_{r=1}^t i_k^r\right) < \left(\sum_{r=1}^t j_1^r, \sum_{r=1}^t j_2^r,\ldots, \sum_{r=1}^t j_k^r\right)  \] 
in lexicographical order. In other words, viewing each monomial as a $k\times t$ standard Young tableau, when we sum up the rows of the tableaux and compare using lexicographic order prioritizing the top row, $p_{I_1}\cdots p_{I_t}$ has a smaller row sum vector than $p_{J_1}\cdots p_{J_t}$. 
\end{defn}

\begin{rem}\label[remark]{suffices to consider finitely many} 
    Let $g= \sum a_{I_1,\ldots, I_t}p_{I_1}\cdots p_{I_t}$ be a homogeneous polynomial in $S$ of degree $d$. Similar to a monomial ordering, we may define the initial term: 
    \[\initial_{\rsl} (g) = \sum_{p_{I_1}\cdots p_{I_t} \textrm{ the smallest in } \textrm{row sum lex}} a_{I_1,\ldots, I_t}p_{I_1}\cdots p_{I_t}.\] 

    If $N>nd$ then observe that $\initial_{\rsl}(g) = \initial_\lambda (g)$ for any $\lambda$. In particular, for any finite set of polynomials $f_1,\ldots, f_m\in S$, we may take $N>>0$ such that $\initial_\lambda(f_i)=\initial_\rsl(f_i)$. This will suffice for our purposes. 
\end{rem}

$X(P_{k,n})$ degenerates further under the Gr{\"o}bner degeneration given by reverse lexicographic monomial order. To describe the latter Gr{\"o}bner degeneration explicitly, recall that a monomial in $S$ is called a \textit{standard monomial} if it is of the form $p_{I_1}\cdots p_{I _t}$ where $I_1\leq \cdots \leq I_t$ is a chain in Bruhat order. It was shown in \cite{Hodge_1943} that the standard monomials form a basis for the homogeneous coordinate ring $\bbC[\Gr(k,n)]$ of $\Gr(k,n)$ (for a more modern treatment, see \cite[Lemmas 7.2.2 and 7.2.3]{Bruns_Herzog_1998}). Moreover, fixing some total extension of Bruhat order on ${[n]\choose k}$ and letting $<_{\revlex}$ be the induced reverse lexicographic monomial order, the standard monomials form a basis for $S/\initial_\revlex (I_{k,n})$. This is due to Sturmfels and White in \cite{Sturmfels-White}. Hence, we have that \[\initial_{\revlex}(I_{k,n}) = \langle p_Ip_J: \textrm{$I$ and $J$ are not comparable in Bruhat order} \rangle. \] 

\begin{prop}
For $N>>0$, we have
\[\mathrm{in}_{\mathrm{revlex}}(I(P_{k,n})) = \mathrm{in}_{\mathrm{revlex}}(I_{k,n}).\] 
\begin{proof}
Since the two ideals have the same Hilbert function (as $I(P_{k,n})$ and $\initial_{\textrm{revlex}}(I_{k,n})$ are both weight degenerations of $I_{k,n}$), it suffices to show that $\initial_{\textrm{revlex}}(I_{k,n})$ is contained in $\initial_{\textrm{revlex}}(I(P_{k,n}))$. Recall that $I(P_{k,n})$ is generated by equations $p_Ip_J - p_{I\wedge J} p_{I\vee J}$ for $I,J$ incomparable in Bruhat order, and $\initial_{\textrm{revlex}}(I_{k,n})$ is generated by $p_Ip_J$ for $I,J$ incomparable in Bruhat order. Now, observe that $\initial_{\textrm{revlex}}(p_Ip_J-p_{I\wedge J} p_{I\vee J}) = p_Ip_J$, so it follows that $\initial_{\textrm{revlex}}(I_{k,n})\subset \initial_{\textrm{revlex}}(I(P_{k,n}))$, as desired. 
\end{proof}
\end{prop}

We now discuss the ideals of positroid varieties and their known Gr{\"o}bner degeneration under reverse lexicographic order. We begin with the following lemma, whose proof we include for lack of a better reference. 
\begin{lemma}
    Let $M$ be a positroid corresponding to an interval $[u,w]$ with $w$ $k$-Grassmannian. The bases of $M$ are precisely the sets in $\pi_k([u,w])\subset {[n]\choose k}$. 
\begin{proof}
   Recall from \cite[Corollary 5.12]{KLS1} that $\Pi_u^w$ is the closure of the GGMS matroid strata of $M$, and hence that the $T$-fixed points of $\Pi_u^w$ correspond exactly to bases of $M$. By \Cref{projected richardsons}, we have that $\Pi^w_u = \pi_k(X_u^w)$. Further, by Borel's fixed point theorem, if $p\in \pi_k(X_u^w)$ is a $T$-fixed point then $\pi^{-1}(p)$ must contain a $T$-fixed point of $X_u^w$. But the $T$-fixed points of $X_u^w$ are precisely $[u,w]$, so the result follows.
\end{proof}
\end{lemma}

We will write $I_u^w$ for the ideal of $\Pi_u^w$ in $Gr(k,n)$. Then, \cite[Theorem 5.15]{KLS1} shows that
\[\ I_u^w = I_{k,n}+ \langle p_I\mid  I\textrm{ is not a base of } M \rangle.\] 
The preceding lemma then implies that $I_u^w$ is given by \[I_{k,n}+ \langle p_I\mid I\notin \pi_k([u,w])\rangle \] 
Gr{\"o}bner degenerations of positroid varieties under reverse lexicographic monomial order were worked out by Knutson, Lam and Speyer in \cite{KLS1}. 

\begin{thm}[KLS]\label[theorem]{initial ideal of I_u^w}
  The ideal $\mathrm{in}_{\mathrm{revlex}}(I_u^w)$ is generated by all monomials $p_{I_1}\cdots p_{I_t}$ such that $\{I_1,\ldots, I_t\}$ is not a face of $\pi_k(\Delta[u,w])$. 
\end{thm}

\subsection{Toric Degenerations of Positroid Varieties}\label[section]{Toric Degenerations of Positroid Varieties}
Recall the notation $S:= \bbC[p_I]_{I\in {[n]\choose k}}$, and fix an integral weight function $\lambda$ on $S$ throughout. In this section we show the following theorem: 
\begin{thm}\label[theorem]{mainresultD}
The scheme-theoretic closure of $\Pi_u^w\times (\bbC-\{0\})$ in $\mathscr{X}$ is a flat family $\mathscr{X}_u^w \to \bbC$  with generic fiber $\Pi_u^w$ and special fiber $X(\calF_u^w)$. Furthermore, this degeneration is $T$-equivariant with respect to the action of the standard torus on $\Pi_u^w$. 
\end{thm} 

The following lemma follows from \cite[Proposition 4.8]{Eisenbud-Sturmfels}. Recall that we always denote the homogeneous ideal of a polarized toric variety $X(\mathcal{K})$ with respect to the projective embedding induced by $\mathcal{K}$ by $I(\mathcal{K})$. 

%\begin{lemma}\label{fence ideal} 
    %The homogeneous ideal of $X(\calF_u^w)$ in $S$ is generated by Hibi relations of the form $p_Ip_J-p_{I\vee J} p_{I\wedge J}$ and $p_{I_1}\cdots p_{I_t}$ where $I_1,\ldots, I_t$ are not contained in a facet of $\calF_u^w$.  
%\end{lemma}

\begin{lemma}\label[lemma]{fence ideal} 
    $I(\calF_u^w)\subset S$ is generated by binomials of the form $p_Ip_J - p_{I\wedge J}p_{I\vee J}$ and monomials of the form $p_{I_1}\cdots p_{I_t}$ where $I_1,\ldots, I_t \in {[n]\choose k}$ are minimal collections of $k$-element subsets such that $I_1,\ldots, I_t$ are not contained in any face of $\calF_u^w$.
\end{lemma}

Recall that given an order ideal $I\subset P$, its indicator vector gives a vertex $v_I$ of $\calO(P)$. We have the following lemma on containment of these vertices in faces. 

\begin{lemma}\label[lemma]{Hibi preserves faces}
Let $P$ be a poset, $\calO(P)$ the associated order polytope and $F$ a face of $\calO(P)$. Suppose $I,J$ are order ideals of $\calO(P)$, and $I\wedge J$ and $I\vee J$ their meets and joins. Then, we have that $v_I,v_J\in F$ if and only if $v_{I\wedge J}$ and $v_{I\vee J}$ are in $F$. 
\begin{proof}
    Let $H$ be a supporting linear functional for $F$, i.e. $\calO(P) \subset  \{x\in \bbR^{P}\mid H(x)\geq a\}$ with $F=\calO(P)\cap \{x\in \bbR^{|P|}\mid H(x)= a\}$. Note that $v_I+v_J = v_{I\wedge J} +v_{I\vee J}$. Then, $v_I,v_J \in F$ if and only if \[H(v_I)+H(v_J) = H(v_{I\wedge J}) + H(v_{I\vee J}) = 2a\]  which holds if and only if $v_{I\wedge J}$ and $v_{I\vee J}$ are in $F$. 
\end{proof}
\end{lemma}

Now, we show that under the assumption that a "nice" Gröbner basis with respect to revlex exists for $I_u^w$, we have $I(\calF_u^w)\subset \initial_\lambda(I_u^w)$. 

\begin{prop}\label[proposition]{nice Gr{\"o}bner basis implies containment}
    Suppose that $I_u^w\subset S$ has a Gr{\"o}bner basis with respect to reverse lexicographical order consisting of the straightening relations for $\Gr(k,n)$ along with polynomials $f_1,\ldots, f_m$ such that $\textrm{in}_\lambda(f_1),\ldots, \initial_\lambda(f_m)$ are monomials, and $\initial_\revlex(\initial_\lambda(f_i)) = \initial_\revlex(f_i)$. Then, $ I(\calF_u^w)\subset \initial_\lambda(I_u^w)$. 
\begin{proof}

     By applying $\initial_\lambda$ to the straightening relations, Gonciulea–Lakshmibai show in \cite[Theorem 5.2]{Gonciulea-Lakshmibai} that $p_Ip_J-p_{I\wedge J}p_{I\vee J}$ is in $\initial_\lambda(I_u^w)$ for all $I,J$. By \Cref{fence ideal}, it suffices to show that for every set $v_{I_1},\ldots, v_{I_t}$ of vertices not contained in a facet of $\F_u^w$, $p_{I_1}\cdots p_{I_t}\in \initial_\lambda(I_u^w)$.
    
    First of all, by the relations $p_Ip_J= p_{I\wedge J}p_{I\vee J}$ applied inductively, we have that \[p_{I_1}\cdots p_{I_t}\equiv  p_{J_1}\cdots p_{J_t} \pmod{\initial_\lambda(I_u^w)}\] for some chain $J_1\leq \cdots \leq J_t$ in Bruhat order. It suffices to show that $p_{J_1}\cdots p_{J_t}\in \initial_\lambda(I_u^w)$. By \Cref{Hibi preserves faces}, it follows that $J_1,\ldots, J_t$ is also not contained in a facet of $\F_u^w$. Furthermore, it suffices to show $p_{J_1}\cdots p_{J_t}\in \initial_\lambda(I_u^w)$ when $J_1,\ldots, J_t$ are minimally not contained in a facet, i.e. we may assume that $J_i\neq J_j$ for $i\neq j$ and any proper subset of the $J_i$'s is contained in a facet of $\F_u^w$.

    Now, observe that $p_{J_1}\cdots p_{J_t}$ is a minimal generator of $\initial_{\textrm{revlex}}(I_u^w)$. To see this, recall from \Cref{initial ideal of I_u^w} that $\initial_{\textrm{revlex}}(I_u^w)$ consists of monomials $p_{S_1}\cdots p_{S_r}$ such that $S_1,\ldots, S_r \notin \pi_k(\Delta[u,w])$. Since $J_1,\ldots, J_t$ is not contained in any facet of $\F_u^w$ and $\pi_k(\Delta[u,w])$ is the canonical triangulation of the fence complex, so respects the polyhedral complex structure, it follows that $\{J_1,\ldots, J_t\}\notin \pi_k(\Delta[u,w])$, so $p_{J_1}\cdots p_{J_t}\in \initial_{\textrm{revlex}}(I_u^w)$.
    
    To see why no other monomial in $\initial_{\textrm{revlex}}(I_u^w)$ non-trivially divides $p_{J_1}\cdots p_{J_t}$, observe that if $J_{i_1}<\cdots <J_{i_l}$ is a proper subset of $\{J_1,\ldots, J_t\}$, then there is some facet $F$ of $\F_u^w$ containing $J_{i_1},\ldots, J_{i_t}$. In particular, $J_{i_1},\ldots, J_{i_l}$ correspond to a chain of order ideals of the quotient poset $Q_F:= P_{\pi_k(w)/\pi_k(u)}/\sim_F$  corresponding to $F$. Refining this to a maximal chain $\mathcal{C}$ of order ideals of $Q_F$, we get a standard fence tableau by \Cref{cor:standard_fence_tableaux}. It follows that $\mathcal{C}$ is a facet of $\pi_k(\Delta[u,w])$, and so $J_{i_1}<\cdots <J_{i_l}$ is in $\pi_k(\Delta[u,w])$. Thus, $p_{J_{i_1}}\cdots p_{J_{i_l}}\notin \initial_{\textrm{revlex}}(I_u^w)$.

    In particular, there must exist some element of the Gr{\"o}bner basis of $I_u^w$ such that a scalar multiple of $p_{J_1}\cdots p_{J_t}$ is its initial term with respect to reverse lexicographic order. Since $J_1\leq \cdots \leq J_t$, this cannot be a straightening relation, as these all have initial term $p_Ip_J$ for $I$ and $J$ incomparable. Hence, for some $f_i$, we must have $\initial_{\textrm{revlex}}(f_i) = p_{J_1}\cdots p_{J_t}$. Since $\initial_\textrm{revlex}(\initial_\lambda(f_i)) =  \initial_\textrm{revlex}(f_i)$ by assumption, and $\initial_\lambda(f_i)$ is a monomial, it follows that, up to scalars, $p_{J_1}\cdots p_{J_t}$ equals $\initial_\lambda(f_i)$. Hence, $p_{J_1}\cdots p_{J_t}\in \initial_\lambda(I_u^w)$, as desired.  
\end{proof}
\end{prop}
In \cite{AGH}, Almousa, Gao and Huang construct a Gröbner basis for $I_u^w$ with respect to reverse lexicographic order. In \Cref{Weight Order Verifications for the AGH Gr{\"o}bner Basis}, we prove the following: 

\begin{thm}\label[theorem]{so we are done}
    The Gr{\"o}bner basis constructed by \cite{AGH} satisfies the conditions of \Cref{nice Gr{\"o}bner basis implies containment}. 
\end{thm}

We can now prove \Cref{mainresultD}.

\begin{proof}[Proof of \Cref{mainresultD}]
    As in the case of the Grassmannian,  \cite[Chapter 15.8]{EisenbudCAbook} implies the existence of a flat family $S[t]/\tilde I_u^w$, where $S[t]/\tilde I_u^w \otimes_{\bbC[t]} \bbC[t,t^{-1}]\cong S/I_u^w [t,t^{-1}]$, and $S[t]/\tilde I_u^w\otimes_{\bbC[t]} \bbC[t]/(t)\cong S/\initial_\lambda(I_u^w)$. In particular, taking Proj gives a flat family $\psi_u^w: \mathscr{X}_u^w\to \bbC$ such that $(\psi_u^w)^{-1}(\bbC-\{0\})$ is $\Pi_u^w\times (\bbC-\{0\})$, and the fiber of the map over $0$ is $\textrm{Proj } S/\initial_\lambda(I_u^w)$. By definition, $\mathscr{X}_u^w$ is a subscheme of $\mathscr{X}$, and $\mathscr{X}_u^w$ is the scheme-theoretic closure of $\Pi_u^w\times (\bbC-\{0\})$ by the uniqueness of flat families (cf. \cite[Proposition 9.8]{Hartshorne}). 

    It remains to show that $\textrm{Proj } S/\initial_\lambda(I_u^w)$ is indeed $X(\calF_u^w)$. It suffices to show that $\initial_\lambda(I_u^w)=I(\calF_u^w)$. By \Cref{ehr = cohom}, we have that $\ehr_{\calF_u^w} (n)$ equals the Hilbert polynomial of $X(\calF_u^w)$. On the other hand, we have that the Hilbert polynomial of $\Pi_u^w$ equals  $\ehr_{\calF_u^w} (n)$ from \Cref{ehrhart poly equals hilbert poly}. By \Cref{nice Gr{\"o}bner basis implies containment}, we have that $I(\calF_u^w)\subset \initial_\lambda(I_u^w)$. By assumption, $I(\calF_u^w)$ is radical. 

    The result now follows from the following easy lemma applied to $I(\calF_u^w), \initial_\lambda(I_u^w)$ and $S$. 

\begin{lemma}
    Let $I,J\subset R= \bbC[x_1,\ldots, x_n]$ be homogeneous ideals, neither of which equals $R$, such that $I\subset J$, $I$ is radical, and $R/I$, $R/J$ have the same Hilbert polynomial. Then, $I=J$. 
\begin{proof}
    Since $R/I$ and $R/J$ have the same Hilbert polynomial, there exists some $N>0$ such that $(R/I)_d = (R/J)_d$ for $d>N$. In particular, $I_d = J_d$ for $d>N$. Given some homogeneous element $y\in J$, consider $y^{N+1}$. Since $J\neq R$, we may assume that $y$ has positive degree. Then, $y^{N+1} \in I$. Since $I$ is radical, $y\in I$. Thus, $I=J$. 
\end{proof}
\end{lemma}
\end{proof}

\section{Cyclic Demazure Character Recurrences via Fence Complexes}\label[section]{Cyclic Demazure Character Recurrences via Fence Complexes}

In \cite{Lam}, Lam introduces the notion of a cyclic Demazure module. Geometrically, this is the $T$-representation given by $V_{\;u}^w(d\omega_k) := (\bbC[\Pi_u^w]_d)^*$, where $\bbC[\Pi_u^w]_d$ is the degree $d$ part of the homogeneous coordinate ring of $\Pi_u^w$. For alternative characterizations of this $T$-representation see \cite[Section 7.1]{AGH}. In \cite{AGH}, the authors prove a recursive formula for the character of cyclic Demazure modules $V_{\: u}^w(d\om_k)$ in terms of Demazure modules $V_{\;u'}^{w'}(d\omega_k)$ of positroids $\Pi_{u'}^{w'}$ contained in $\Pi_u^w$, as well as $V_{\;u}^w((d-1)\om_k)$. In this section, we prove a similar recursive formula by considering the polyhedral geometry of $\calF_u^w$. Importantly, throughout this section, we assume that $w$ is $k$-Grassmannian.

\subsection{Relationship to Fence Complexes}

First, we clarify the relationship of $d\calF_u^w$ to the character of $\ch (V_{\;u}^w(d\om_k))$. Recall that $T=(\bbC^*)^n$ acts on $S$, where $(t_1,\ldots, t_n)\cdot p_I = t_{i_1}^{-1}\cdots t_{i_k}^{-1} p_I$ for $(t_1,\ldots, t_n) \in T$. In particular, since the degeneration of $I_u^w$ to $I(\calF_u^w)$ respects the multigrading induced by this $T$-action, it follows that the characters of the $T$-representations on $(\bbC[\Pi_u^w]_d)^*$ and $(\bbC[X(\calF_u^w)]_d)^*$ are equal: 
\[\ch(S/I_u^w)_d = \ch(S/I(\calF_u^w))_d.\]

Let $M$ denote the lattice $\bbZ^{{[n]\choose k}}$. By \Cref{Order polytopes are normal}, any lattice point $p\in M\cap d\calO(P_{k,n})$ can be written as the sum of $d$ lattice points in $\calO(P_{k,n})$ whose corresponding partitions form a chain in Young's lattice. Moreover, it is easy to check from the definitions that this representation is unique. Converting these partitions into $k$-element sets gives a chain in ${[n]\choose k}$, and this gives a SSYT whose columns are these $k$-element subsets. We denote this SSYT by $\mathrm{SSYT}_d(p)$. 
\begin{rem}
    The above discussion is a special case of the algebra with straightening law structure on the coordinate ring of $\calO(P)$ for any order polytope, as discussed in \cite{Hibi}. However, his language is sufficiently different from ours that we will not pursue this perspective. 
\end{rem}

Recall from \Cref{normal polytope and coordinate ring} that the lattice points in $M\cap d\calF_u^w$ form a basis for $\bbC[X(\calF_u^w)]|_d$. Moreover, observe $\bbC[X(\calO(P_{k,n}))]_1 \cong S_1$, and that $v_I$ maps to the variable $p_I$ under this isomorphism. Since the variable $p_I$ has weight $t_{i_1}^{-1}\cdots t_{i_k}^{-1}$, where $I=\{i_1,\ldots, i_k\}$, it follows that $v_I$ does as well. In particular, for $p\in M\cap d|\calF_u^w|$, it follows that the weight of $p$ is given by the product of the weights of $v_{I_1},\ldots, v_{I_d}$, where $I_1,\ldots, I_d$ are the columns of $\mathrm{SSYT}_d(p)$.

For $(i,j)\in [k]\times [n-k]$, recall that $x_{i,j}$ denotes the entry in the $(i,j)$ cell of $p$. For convenience, let $x_{i,j} = d$ for $j>n-k$, and $x_{i,j} = 0$ for $j<0$. One then sees that
\[x_{i,j} - x_{i,{j-1}} = \#\text{ of times }(k+1-i+j)\text{ appears in row }(k+1-i)\text{ of }\mathrm{SSYT}_d(p).\]
Therefore if we define 
    \[y_m := \sum_{i=1}^k x_{i, i+m-k-1},\]
then $y_m - y_{m-1}$ is the number of times $m$ appears in $\mathrm{SSYT}_d(p)$. Since if $I=\{i_1,\ldots, i_k\}$ then $p_I$ has weight $(t_{i_1}\cdots t_{i_k})^{-1}$, it follows that the weight of $p$ is $(t_1^{y_1} t_2^{y_2-y_1}\cdots t_n^{y_n-y_{n-1}})^{-1}$. 
\begin{defn}

    We define $\psi_d:M\to \bbZ^n$ for each $p\in M$ to be the affine transformation given by 
    \[\psi_d(p) := (y_1,y_2-y_1,\ldots, y_n-y_{n-1}).\] 
\end{defn}
\begin{rem}\label[remark]{rem: psi_d and psi_d-1}
In words, $\psi_d(p)$ is defined so that its $i$th coordinate equals the number of times that $i$ appears in $\textrm{SSYT}_d(p)$. Observe that $\psi_d$ depends on $d$, since we set $x_{i,j}$ to $d$ for $j>n-k$. 
\end{rem}
It follows from the above discussion that
 \[\ch(V_u^w(d\om_k)) = \ch\left((S/I_u^w)_d ^* \right)   = \sum_{p\in d\calF_u^w} t^{\psi_d(p)}. \] 

We now give an identity utilizing only the polyhedral geometry of $\calF_u^w$ that will, under $\psi_d$, specialize to a recursive character formula for $\ch(V_{\;u}^w(d\om_k))$. In fact this will be the same as the AGH formula, though this will not be immediately obvious.

\begin{defn}
    For $u<w$, define
\[\calD_u^w := \{(x_{i,j})\in \calF_u^w\ |\ x_{i,j} = 0 \text{ for some }(i,j)\in \pi_k(w)/\pi_k(u)\}.\]
\end{defn}

\begin{rem}
    Observe that $\calD_u^w$ is the union of faces of $\calF_u^w$. More precisely,
\[\calD_u^w = \{F\subset \calF_u^w\mid F \text{ has shape }\lambda/\mu\text{ with }\mu\gneq \pi_k(u)\}.\]
\end{rem}

We work in the group algebra $\bbC [M] $. In this setting, observe that $\psi_d$ induces a vector space map between the group algebras $\bbC[M] \to \bbC[\bbZ^n]$, but not an algebra homomorphism. For a lattice point $p\in M$, we denote its corresponding element in $\bbC [M]$ by $z^p$. Define a lattice point $q \in M$ by $$q_{i,j} = \begin{cases}
    0 & \text{if }(i,j)\in \pi_k(u)\\
    1 & \text{else}.
\end{cases}$$ Then, observe that the lattice points of $d(\calF_u^w -q)$ that are not contained in $(d-1)(\calF_u^w-q)$ are precisely those in $d(\calD_u^w-q)$. By $X-q$, we mean the set $X$ translated down by $q$, i.e. $\{x-q\mid x\in X\}.$

Hence, we have the following identity in $\bbC[M]$: 

\[ \sum_{p\in d\calF_u^w} z^{p-dq} = \sum_{p\in (d-1)\calF_u^w} z^{p-(d-1)q} + \sum_{p\in d\calD_u^w} z^{p-dq}. \] 
Multiplying both sides by $z^{dq}$ gives 
\[\sum_{p\in d\calF_u^w} z^p = \sum_{p\in (d-1)\calF_u^w} z^{p+q} + \sum_{p\in d\calD_u^w} z^p.\] 
Applying $\psi_d$ to both sides gives 
\[ \ch((S/I_u^w)_d) = \sum_{p\in (d-1)\calF_u^w} t^{\psi_d(p+q)} + \sum_{p\in d\calD_u^w} t^{\psi_d(p)}. \] 
Observe that by \Cref{rem: psi_d and psi_d-1}, we have $\psi_d(p+q) = \psi_{d-1}(p)t^{u[k]}$. This follows as $\textrm{SSYT}_d(p+q)$ is given by appending $u[k]$ as a column to the left of $\textrm{SSYT}_{d-1}(p)$. Thus, this expression simplifies, which we record in the following proposition. 
\begin{prop}\label[proposition]{prop:starter_character_formula} 
    \[ \ch((S/I_u^w)_d) = t^{u[k]}\ch((S/I_u^w)_{d-1}) + \sum_{p\in d\calD_u^w} t^{\psi_d(p)}. \] 
\end{prop}

Here, we use $t^{u[k]}$ to denote $t_{u(1)}\cdots t_{u(k)}$. We now investigate $\calD_u^w$ in order to further simplify the rightmost term. The following is obvious from the definitions: 

\begin{lemma}
    $\calD_u^w$ is the union of all fence complexes $\calF_{u'}^{w'}\subset \calF_u^w$ where $\pi_k(u')\gneq \pi_k(u)$. 
\begin{proof}
    Clearly all such $\calF_{u'}^{w'}$ are contained in $\calD_u^w$. Conversely, if $F$ is some face of $\calD_u^w$, then $F$ has shape $\lambda/\mu$ for some $\mu\gneq \pi_k(u)$. Now $\int F\subset \int \calF_{u'}^{w'}$ for a unique fence complex $\calF_{u'}^{w'}\subset\calF_u^w$. By \Cref{interior of the fence complex}, $\pi_k(u') = \mu\gneq \pi_k(u)$. 
\end{proof}
\end{lemma}

\begin{lemma}\label[lemma]{lem:calD_u^w as a simplicial complex}
Identifying $\calF_u^w$ with the geometric realization of its canonical triangulation, $\calD_u^w$ is the geometric realization of the subcomplex of $\pi_k(\Delta[u,w])$ given by the union of all simplices that do not contain $\pi_k(u)$ as a vertex. 
\begin{proof}
Certainly if $S$ is some simplex contained in $\calD_u^w$, then $S$ does not contain $\pi_k(u)$ as a vertex. Thus, it suffices to show that any simplex $S$ not containing $\pi_k(u)$ as a vertex is contained in $\calD_u^w$. Recall that $S$ corresponds to some chain $u_1<_k u_2<_k\cdots <_k u_m$, where $u_m<_k w$ in $k$-Bruhat order, and $\pi_k(u_1)>\pi_k(u)$. In particular, saturating this chain to a saturated $k$-Bruhat chain from $u_1$ to $w$ gives a facet of $\calF_{u_1}^w\subset \calD_u^w$ containing $S$, so the result follows.
\end{proof}
\end{lemma} 

In view of the previous lemma, we fix the simplicial complex structure on $\calD_u^w$ to be the one induced by $\pi_k(\Delta[u,w])$. 

\begin{lemma}
    $\calD_u^w$ is homeomorphic to a closed ball. 
\begin{proof}
    By \Cref{lem:calD_u^w as a simplicial complex}, it follows that $\pi_k(\Delta[u,w]) = \Cone_{\pi_k(u)} (\calD_u^w)$, where $\Cone_{\pi_k(u)} (\calD_u^w)$ denotes the cone of the simplicial complex corresponding to $\calD_u^w$ over the vertex $\pi_k(u)$. Recall from \cite{KLS0} that $\pi_k(\Delta[u,w])$ is shellable. Then, since all facets of $\pi_k(\Delta[u,w])$ must contain $\pi_k(u)$, this descends to a shelling order on $\calD_u^w$. Furthermore, if $F$ is a codimension $1$ face of $\calD_u^w$, and $F$ is contained inside facets $F_1,\ldots, F_m$ of $\calD_u^w$, observe that $\Cone_{\pi_k(u)}(F)$ is a codimension $1$ face of $\pi_k(\Delta[u,w])$ and $\Cone_{\pi_k(u)}(F_1),\ldots, \Cone_{\pi_k(u)}(F_m)$ are facets containing $\Cone_{\pi_k(u)}(F)$. Since \cite[Lemma 10.1]{KLS0} proves that every codimension $1$ face in $\pi_k(\Delta[u,w])$ is contained in at most two facets, it follows that $m\leq 2$ and hence $\calD_u^w$ is either subthin or thin. 
    
    Moreover, take $S$ to be a maximal simplex of $\pi_k(\Delta[u,w])$, corresponding to the saturated $k$-Bruhat chain $u\lessdot_ku_1\lessdot_k\cdots\lessdot_k u_{m-1}\lessdot_k w$. Then $T=S\setminus\{\pi_k(u),\pi_k(w)\}$ is a codimension $1$ face of $\calD_u^w$, and $\Cone_{\pi_k(u)}(T)$ corresponds to the saturated $k$-Bruhat chain $u\lessdot_ku_1\lessdot_k\cdots \lessdot_k u_{m-1} = w'$ with $w'\lessdot_k w$. Hence, by \Cref{boundary of the fence complex}, $\Cone_{\pi_k(u)}T$ is contained in $|\pi_k(\Delta([u, w']))| = \calF_u^{w'}\subset \d\calF_u^w = |\bdy \pi_k(\Delta[u,w])|$. But if $T$ is contained in two maximal faces of $\calD_u^w$, it follows that $\Cone_{\pi_k(u)}(T)$ is contained in two maximal faces of $\pi_k(\Delta[u,w])$, a contradiction. Therefore $\calD_u^w$ must be subthin, so it is homeomorphic to a closed ball by \cite[Fact A.2.4.3]{Bjorner-Brenti}. 
\end{proof}
\end{lemma}

\begin{lemma}\label[lemma]{lem:bdy of S related to Cone(S)} 
    Let $\mathcal{S}$ be a simplicial complex homeomorphic to a closed ball. Then, $\d\calS$ consists of the union of all simplices $F$ of $\calS$ such that $|\Cone(F)|\subset |\d\Cone(\calS)|$. 
\begin{proof}
    Observe the boundary subcomplex is pure dimensional. It thus suffices to show that $\d\calS$ is the union of all \emph{codimension one} simplices $F$ of $\calS$ such that $|\Cone(F)|\subset|\d\Cone(\calS)|$. Recall, however, that if $F$ is a codimension $1$ simplex of $\calS$, then $F\in \d\calS$ if and only if it is contained in exactly one facet of $F$. However, this is equivalent to $\Cone(F)$ being contained in exactly one facet of $\Cone(\calS)$, i.e., $|\Cone(F)|\subset |\d\Cone(\calS)|$. The lemma follows.  
\end{proof} 
\end{lemma}

\begin{lemma}\label[lemma]{lemma: boundary of calD_u^w} 
The boundary of $\calD_u^w$ is given by the union of the fence complexes $\calF_{u'}^{w}$, where $\pi_k(u')> \pi_k(u)$ and $w>u'>\bar{u}>u$ for some $\bar{u}$ with $\pi_k(\bar{u})=\pi_k(u)$, together with the union of the fence complexes $\calF_{u'}^{w'}\subset \calF_u^w$ where $w'$ is a $k$-Grassmannian permutation strictly less than $w$ and $\pi_k(u')>\pi_k(u)$.
\end{lemma}

\begin{proof}
For a simplex $S$ not containing the vertex $\pi_k(u)$, we will let $\Cone(S)$ denote $\Cone_{\pi_k(u)}(S)$ throughout. By \Cref{lem:bdy of S related to Cone(S)}, we have that 
\[\bdy\calD_u^w = \{S\in \pi_k(\Delta[u,w])\mid \pi_k(u)\notin S,\ |\Cone(S)|\subset \d\calF_u^w\}.\] 

If $S$ is a simplex contained in $\calF_{u'}^w$ where $\pi_k(u')> \pi_k(u)$, and if $w>u'>\bar{u}>u$ for some $\bar{u}$ with $\pi_k(\bar{u})=\pi_k(u)$, then we claim that $|\Cone(S)|\subset \calF_{\bar{u}}^w$. One way to see this is to lift $S$ to a Bruhat chain contained in the interval $[u',w]$ and then observe we may extend this chain by adding $\bar{u}$ to get a lift of $\Cone(S)$. Hence, by \Cref{boundary of the fence complex}, we have  $|\Cone(S)|\subset \calF_{\bar{u}}^w\subset  \bdy \calF_u^w$. 

If instead, $S$ is contained in $\calF_{u'}^{w'}\subset \calF_u^w$, where $w'$ is a $k$-Grassmannian permutation less than $w$, then since $S\in \pi_k(\Delta[u,w])$, $S$ lifts to a $k$-Bruhat chain $u_1<\cdots <u_m$ contained in $[u,w]$ with $u_m<_k w$. In particular, $\Cone(S) \in \pi_k(\Delta([u,u_m]))$. This corresponds to some positroid variety $\Pi_{v}^x\subset \Pi_u^w$ where $x<w$ is $k$-Grassmannian. In particular, it follows that $|\Cone(S)|\subset \calF_v^x\subset\bdy \calF_u^w$. This shows that the boundary of $\calD_u^w$ contains all of the fence complexes in our asserted boundary set.

We now show the converse. Let $S$ be a simplex in $\bdy \calD_u^w$. We have two cases.  
\begin{itemize}
    \item \textbf{Case 1}: Suppose $\pi_k(w)\in S$. Since $\Cone(S)$ is in the boundary of $\calF_u^w$ but contains $\pi_k(u)$ and $\pi_k(w)$, it follows from \Cref{boundary of the fence complex} that $|\Cone(S)|\subset\calF_{\bar{u}}^w$ for some $\bar{u}>u$ with $\pi_k(\bar{u}) = \pi_k(u)$. This then implies that $|S|\subset \calF_{u'}^w$ with $w>u'>\bar u$. 

    \item \textbf{Case 2}: Suppose $\pi_k(w)\notin S$. Then, $\Cone(S)$ lifts to some $k$-Bruhat chain $u<_ku_1<_k\cdots <_ku_m$ in $[u,w]$ and $u_m<_k w$. We have $\Delta(\pi_k([u,u_m])_ = \Delta(\pi_k([u'',w'']))$ for a unique $u''<w''$ with $w''$ a $k$-Grassmannian permutation by \Cref{maximal chains bij}. In particular, since $[u,u_m]\subset [u,w]$ we have $\calF_{u''}^{w''}\subset \calF_u^w$. By choice of $u'',w''$ we then have $\pi_k(u)=\pi_k(u'')<\pi_k(w'')<\pi_k(w)$ in Bruhat order. It follows by \Cref{boundary of the fence complex} that  $|\Cone(S)|\subset \calF_{u''}^{w''}\subset \bdy\calF_u^w$. In this case, $S$ is contained in the fence complex $\calF_{u'}^{w'}$ corresponding to the positroid $\pi_k([u_1,u_m])$. Since $\pi_k([u_1,u_m])\subset \pi_k([u,w])$ it follows that $\calF_{u'}^{w'}\subset \calF_u^w$. 
\end{itemize}

In both cases, $S$ is contained in our asserted boundary set, so the claim follows. 
\end{proof}

In particular, now that we have a handle on the topology of $\calD_u^w$, we can apply Möbius inversion to get a recursion on cyclic Demazure characters. 

\begin{defn}
Let $\calL_k'(u,w)$ consist of all $x\in (u,w]$ with the following property: for all $\bar{u}$ with $u\lessdot \bar{u}\leq x$, we have $u<_k \bar{u}$. 
\end{defn}

In particular, $\int \calD_u^w$ is the union of $\int \calF_x^w$ for $x\in \calL_k'(u,w)$. Now, recall we had the expression: 
\[\sum_{p\in d\calD_u^w} z^p,\] 
and observe that 
\[\sum_{p\in d\calD_u^w} z^p = \sum_{\calF_{x}^{w'} \subset \calD_u^w} \sum_{p\in d\left(\int \calF_{x}^{w'}\right)\cap M} z^p. \]  

Since $\calD_u^w$ is a ball, we may apply \cite[Proposition 3.8.9]{EC1} to calculate the Möbius function of the poset of fence complexes contained in $\calD_u^w$. Then, using Möbius inversion along with our calculation of the boundary of $\calD_u^w$ implies that
\[\sum_{p\in d\calD_u^w} z^p = \sum_{x\in \calL_k'(u,w)} \sum_{p\in d\calF_x^w} (-1)^{\ell(x)-\ell(u)+1} z^p. \] 
In particular, applying  $\psi_d$  to the above expression and plugging this into \Cref{prop:starter_character_formula}, we get 
\begin{prop}\label[proposition]{character formula} 
\[ \ch(V_u^w(d\om_k)) = t^{u[k]}\ch(V_u^w((d-1)\om_k)) + \sum_{x\in \calL_k'(u,w)} (-1)^{\ell(x)-\ell(u)+1} \ch(V_x^w(d\om_k)).\] 
\end{prop}

\subsection{Recovering AGH's Recurrence.} 
Almousa, Gao and Huang prove in \cite{AGH} a character formula that looks very similar to the one we prove. They show this by relating the multigraded Hilbert series of positroids to the Hilbert series of matrix Schubert varieties, and then applying the $K$-theoretic Monk's rule, due to \cite{Lenart}. Despite our different proof strategy, we show that our formula is actually equivalent to theirs.

We first describe the AGH character formula.

\begin{defn}[\cite{AGH}]\label[definition]{AGH set} 
Let the set $\calL_k(u,w)$ consist of all $x\in (u,w]$ such that there exist transpositions $t_i = (a_i\ b_i)$ with $a_i\leq k<b_i$, and a chain 
\[u\lessdot_k ut_1\lessdot_k\cdots \lessdot_k ut_1t_2\cdots t_m = x,\] 
where either $b_i>b_{i+1}$ or ($b_i=b_{i+1}$ and $a_i<a_{i+1}$). 
\end{defn}

\begin{thm}[\cite{AGH}] 
The characters of cyclic Demazure modules satisfy the following recurrence:
\[ \ch (V_{\;u}^w(d\om_k)) = t^{u[k]}\ch(V_{\;u}^w((d-1)\om_k)) +\sum_{x\in \calL_k(u,w)} (-1)^{\ell(x)-\ell(u)+1}\ch (V^w_{\;x}(d\om_k)).\] 
\end{thm}

To show that the above recurrence is equivalent to ours in \Cref{character formula}, it suffices to show the following:

\begin{prop}\label[proposition]{the two sets are equal} 
    We have $\calL_k'(u,w) = \calL_k(u,w)$. 
\end{prop}

The rest of the section is dedicated to proving this proposition.

\begin{defn}
    For $x\in S_n$ and $i,j \in [n]$,  define \[ x[i,j]: = |\{ (a,x(a))\:|\: a\leq i, x(a)\geq j\}|.  \] 
\end{defn}
Recall the following characterization of Bruhat order (see \cite[Theorem 2.1.5]{Bjorner-Brenti}): 
\begin{thm}
    Let $x,y \in S_n$. Then, $x\leq y$ if and only if $x[i,j]\leq y[i,j]$ for all $i,j\in [n]$. 
\end{thm}

\begin{lemma}
    We have $\calL_k(u,w)\subset \calL_k'(u,w)$. 
\end{lemma}
\begin{proof}
    Take $x\in \calL_k(u,w)$. Then, by definition, there exists some chain \[u\lessdot_k ut_1\lessdot_k \cdots \lessdot_k ut_1\cdots t_m =x\] with $t_i = (a_i \  b_i)$ satisfying the conditions of \Cref{AGH set}.  It suffices to show that for each $\bar{u}\gtrdot u$ with $\bar{u}\not >_k  u$, we have $\bar{u}\not\leq ut_1\cdots t_m$ in Bruhat order. Note since $\bar{u}\gtrdot u$ but $\bar{u}\not >_k u$, we have $\bar{u} = u\cdot (i j) $ for $i<j$, where either $i, j \leq k$ or $i,j >k$. 
\begin{itemize} 
   \item \textbf{Case 1}: $i,j \leq k$. Then, $\bar{u}[i, u(j)]>u[i,u(j)]$. We show that actually $$ut_1\cdots t_m [i,u(j)]\leq u[i,u(j)],$$ \noindent and hence that  $ut_1\cdots t_m \not>\bar{u}$ in Bruhat order. %(in fact, since $ut_1\cdots t_m\geq u$, this actually shows $ut_1\cdots t_m [i,u(j)]= u[i,u(j)]$, but we will not use this).
    Suppose this were not the case. Then, there exists some $r'\leq i$ such that $ut_1\cdots t_m (r') \geq u(j)$ but $u(r')<u(j)$. Let $f$ be the smallest element of $[m]$ such that there exists some $r\leq i$ such that $ut_1\cdots t_f (r) \geq u(j)$, $u(r)<u(j)$. In particular, it follows that $a_f = r$ and \[\tag{$*$} u(b_f)\geq ut_1\cdots t_{f-1}(b_f)\geq u(j),\] where the first inequality follows from \Cref{thm: bergeron-sottile}. By the minimality of $f$, we also have that $$ut_1\cdots t_{f-1}(r)<u(j)\leq ut_1\cdots t_{f-1}(j),$$ \noindent where the second inequality follows again from \Cref{thm: bergeron-sottile}. If $ut_1\cdots t_{f-1}(j)<ut_1\cdots t_{f-1}(b_f)$, then since $r<j<b_f$, and  it follows that $ut_1\cdots t_f = ut_1\cdots t_{f-1} \cdot (r \ b_f)$ is not a $k$-Bruhat cover of $ut_1\cdots t_{f-1}$, and hence we have a contradiction. Thus, we may assume that $ut_1\cdots t_{f-1}(j)>ut_1\cdots t_{f-1}(b_f)$.

    Now, let $g$ be the first index in $[0,f-1]$ such that $ut_1\cdots t_g(j)>ut_1\cdots t_g(b_f)$. Observe that $g$ exists by our assumption that $ut_1\cdots t_{f-1}(j)>ut_1\cdots t_{f-1}(b_f)$. If $g=0$, then $u(j)>u(b_f)$, contradicting ($*$). Thus, $g>0$. It then follows by definition of $g$ that $a_g= j$. 
    
    We now have two cases, either $b_g>b_f$ or $b_g=b_f$. We cannot have $b_g<b_f$ since $g<f$, by the definition of $\calL_k(u,w)$. If we have $b_g>b_f$, then by our definition of $g$, observe that 
    \[ut_1\cdots t_{g-1}(j)<ut_1\cdots t_{g-1}(b_f).\] 

    Thus, we must have $ut_1\cdots t_{g-1}(b_f)>ut_1\cdots t_{g-1}(b_g)$, since otherwise \[ut_1\cdots t_{g-1}(j)< ut_1\cdots t_{g-1}(b_f)< ut_1\cdots t_{g-1}(b_g) ,\] 
    and $j<b_f<b_g$, which contradicts the fact that $ut_1\cdots t_g = ut_1\cdots t_{g-1}\cdot (j\ b_g)$ covers $ut_1\cdots t_{g-1}$ in Bruhat order. 

    The case where $b_g = b_f$ is even simpler, as in this case, we have $a_g = j$ and $a_f = r$. Since $r<j$, but $f>g$, this is a contradiction. Hence, no such $r'$ can exist and it follows that $ut_1\cdots t_m [i,u(j)]\leq u[i,u(j)], $ as desired. 
    
   \item \textbf{Case 2}: $i,j>k$. We show that \[ut_1\cdots t_m[j-1, u(i)+1]\leq u[j-1,u(i)+1],\] 
    which will imply that $ut_1\cdots t_m \not\geq \bar{u}$, since $\bar{u}[j-1, u(i)+1]>u[j-1, u(i)+1].$ Since $i<j$ and $u\cdot (i,j) > u$, we must have $u(i)<n$, so these quanitites are defined. 
    
    If $ut_1\cdots t_m[j-1, u(i)+1]> u[j-1,u(i)+1]$, let $h$ be the smallest index such that $ut_1\cdots t_h[j-1, u(i)+1]>u[j-1, u(i)+1]$. Then, we must have $$ut_1\cdots t_{h-1}(a_h)\leq u(i)<u(j) \:\text{and} \:ut_1\cdots t_{h-1}(b_h)>u(i).$$
    
    If $b_h>i$, this gives a contradiction to the fact that $ut_1\cdots t_h\gtrdot_k ut_1\cdots t_{h-1}$ is a $k$-Bruhat cover, since $a_h<i<b_h$, but \[ut_1\cdots t_{h-1}(a_h)\leq u(i) = ut_1\cdots t_{h-1}(i) < ut_1\cdots t_{h-1}(b_h),\] 
    where $u(i)=ut_1\cdots t_{h-1}(i)$ follows from the definition of $\calL_k(u,w)$, along with the fact that $b_h>i$ and $i\in [k+1,n]$. 

    If instead, $b_h\leq i$, then observe we cannot have $ut_1\cdots t_h [j-1, u(i)+1] > ut_1\cdots t_{h-1}[j-1, u(i)+1]$, since then we have \[b_h\in \{a\in [n]: a\leq j-1, ut_1\cdots t_{h-1}(a)\geq u(i)+1\}\] and 
    \[\{a\in [n]: a\leq j-1, ut_1\cdots t_{h-1}(a)\geq u(i)+1\} -\{b_h\} \cup \{a_h\} = \{a\in [n]: a\leq j-1, ut_1\cdots t_{h}(a)\geq u(i)+1\}.\] 
    Thus, $ut_1\cdots t_h [j-1, u(i)+1]\leq ut_1\cdots t_{h-1}[j-1,u(i)+1]$, a contradiction to our choice of $h$. 
\end{itemize} 
    Thus, in either case, we have a contradiction, so it follows that \[ut_1\cdots t_m[j-1, u(i)+1]\leq u[j-1,u(i)+1],\] as desired. 
\end{proof}

\begin{lemma}\label[lemma]{lem:goodchain_induction_1}
    Let $u\le x$ such that all Bruhat covers of $u$ in $[u,x]$ are $k$-Bruhat covers. Pick $a\le k<b$ with $b$ maximal, and then $a$ minimal with respect to the chosen $b$, such that $u\lessdot u\cdot t_1$ where $t_1 = (a\ b)$. Then, all Bruhat covers of $ut_1$ in $[ut_1,x]$ are $k$-Bruhat covers.
\end{lemma}

\begin{proof}
    Note that $u\lessdot ut_1\implies u(a)<u(b)$. Suppose for the sake of contradiction that $ut_1\lessdot z$ is a non $k$-Bruhat cover in $[ut_1, x]$, and write $t_2 = (ut_1)^{-1} z$. Necessarily, $t_1\ne t_2$ and $t_2\in S_k\times S_{n-k}$. If $t_1$ and $t_2$ commute, then $u\lessdot ut_2\lessdot ut_2t_1$. But $u\lessdot ut_2$ is a non $k$-Bruhat cover, contradicting the assumption on $u$. Thus, $t_1$ and $t_2$ do not commute. It follows that if $t_2\in S_k$, then $t_2 = (a\ c)$ for some $c\le k$, and if $t_2\in S_{n-k}$, then $t_2 = (b\ c)$ for some $c>k$.

    Let us consider the former case $t_2 = (a\ c)$. 

    \begin{itemize}
        \item \textbf{Case 1}: $u(a) < u(b)<u(c)$ or $u(c)<u(a)<u(b)$. Then $u\lessdot ut_2\lessdot ut_2(t_2t_1t_2)= ut_1t_2$. Now $u\lessdot ut_2$ is a non $k$-Bruhat cover, contradicting the assumption on $u$.

        \item \textbf{Case 2}: $u(a) < u(c) < u(b)$. Note that $ut_1\lessdot ut_1t_2$ along with $u(c)<u(b)$ implies that $c<a$. Then $u\lessdot ut_1t_2t_1\lessdot u(t_1t_2t_1)t_1 = ut_1t_2$. Now $t_1t_2t_1 = (c\ b)$ with $c<a$, contradicting the minimality of $a$.
    \end{itemize}

    When $t_2 = (b\ c)$, the casework is similar.
    \begin{itemize}
        \item \textbf{Case 1}: $u(a) < u(b)<u(c)$ or $u(c)<u(a)<u(b)$. Then $u\lessdot ut_2\lessdot ut_2(t_2t_1t_2)= ut_1t_2$. $u\lessdot ut_2$ is a non $k$-Bruhat cover, contradicting the assumption on $u$.

        \item \textbf{Case 2}: $u(a) < u(c) < u(b)$. $ut_1\lessdot ut_1t_2$ along with $u(a)<u(c)$ implies that $b<c$. Then $u\lessdot ut_1t_2t_1\lessdot u(t_1t_2t_1)t_1 = ut_1t_2$. Now $t_1t_2t_1 = (a\ c)$ with $b<c$, contradicting the maximality of $b$.
    \end{itemize}    
\end{proof}

\begin{lemma}\label[lemma]{lem:goodchain_induction_2}
    Let $u\lessdot_k ut_1\lessdot_k ut_1t_2\le x$, where $t_i = (a_i\ b_i)$ with $a_i\le k<b_i$ for $i = 1, 2$. Assume further that either $b_2>b_1$ or ($b_2 = b_1$ and $a_2<a_1$). Then, either $t_1$ and $t_2$ commute, or there exists $\bar u$ with $u\lessdot \bar u\le x$ but $u\nless_k \bar u$.
\end{lemma}

\begin{proof}
    We split into two cases.

    \begin{itemize}
        \item \textbf{Case 1}: $b_2>b_1$. If $a_1\ne a_2$, then $t_1$ and $t_2$ commute and we are done.

        Hence, assume $a_1 = a_2$. Then $u\lessdot ut_1\lessdot ut_1t_2\implies u(a)<u(b_1)<u(b_2)$, and thus $u\lessdot ut_1t_2t_1\lessdot u(t_1t_2t_1)t_1 = ut_1t_2$, with $u\lessdot ut_1t_2t_1$ a non $k$-Bruhat cover. Taking $\bar u = ut_1t_2t_1$ works. 

        \item \textbf{Case 2}: $b_2=b_1$ and $a_2<a_1$. Here, $u\lessdot ut_1\lessdot ut_1t_2\implies u(a_2)<u(a_1)<u(b)$. Then $u\lessdot ut_1t_2t_1\lessdot u(t_1t_2t_1)t_1$, with $u\lessdot ut_1t_2t_1$ a non $k$-Bruhat cover. Taking $\bar u = ut_1t_2t_1$ works.
    \end{itemize}
\end{proof}

The following lemma completes the proof of \Cref{the two sets are equal}.
\begin{lemma}
    We have $\calL_k'(u, w)\subset \calL_k(u, w)$.
\end{lemma}

\begin{proof}
    We proceed by induction on $\ell(w)-\ell(u)$. When $u=w$, both sets are empty. Henceforth assume that $u\lneq w$.

    Suppose $x\in \calL_k'(u, w)$. Then every cover of $u$ in $[u,x]$ is a $k$-Bruhat cover. Pick $a_1\le k<b_1$ with $b_1$ maximal, and then $a_1$ minimal such that $u\lessdot ut_1<x$, where $t_1 = (a_1\ b_1)$. By \Cref{lem:goodchain_induction_1}, all Bruhat covers of $ut_1$ in $[ut_1, x]$ are $k$-Bruhat covers of $ut_1$. Thus $x\in \calL_k'(ut_1, w)$. By the induction hypothesis $x\in \calL_k(ut_1, w)$. Therefore, there exist transpositions $t_i = (a_i\ b_i)$ with $a_i\le k< b_i$ for $2\le i\le m$, and a chain 
    \[ut_1\lessdot _kut_1t_2\lessdot_k\cdots \lessdot_k ut_1t_2\cdots t_m = x\]
    where either $b_i>b_{i+1}$ or ($b_i = b_{i+1}$ and $ a_i<a_{i+1}$) for $2\le i\le m-1$.

    We claim finally that $b_1>b_2$ or ($b_1=b_2$ and $a_1<a_2$). If not, \Cref{lem:goodchain_induction_2} implies either that $t_1$ and $t_2$ commute, or that there exists $\bar u$ with $u\lessdot \bar u\le x$ but $u\nless_k \bar u$. The latter is impossible, for it would imply $x\notin \calL_k'(u, w)$. The former is impossible, for then $u\lessdot ut_2$, contradicting the choice of $b_1$ and $a_1$. Hence $b_1>b_2$ or ($b_1 = b_2$ and $a_1<a_2$), and the chain
    \[u\lessdot_k ut_1\lessdot_k ut_1t_2\lessdot_k \cdots \lessdot_k ut_1t_2\cdots t_m = x\]
    implies that $x\in \calL_k(u, w)$.
\end{proof}

\section{Further Directions}

We present here several natural questions that arose from our work. The first concerns the question of extending our results to other partial flag varieties in type A. It seems that the natural generalizations of positroids here are projected Richardson varieties (cf. \cite{KLS2}). 

In Kim's thesis \cite{Kim}, he shows that under the Gonciulea–Lakshmibai degeneration of the complete flag variety in \cite{Gonciulea-Lakshmibai}, a Richardson variety $X_u^w$ degenerates to the union of toric varieties corresponding to certain faces of the Gelfand–Tsetlin polytope associated to a regular dominant weight $\lambda$. These faces are indexed by certain pairs of pipe dreams. Let $\mathcal{G}_u^w$ denote the polyhedral complex consisting of these faces. We ask the following two questions in analogy with \Cref{regular CW complex structure}. 

\begin{question}\label[question]{question: Kim analogy to closed ball}
    Is $\mathcal{G}_u^w$ homeomorphic to a closed ball? 
\end{question}

\begin{question}\label[question]{question: Kim analogy of boundary}
    What is the boundary of $\mathcal{G}_u^w$? Does it consist of precisely the union of all $\mathcal{G}_{u'}^{w'}$ where $X_{u'}^{w'}\subset X_u^w$? 
\end{question}

We also ask about analogues in other partial flag varieties: 
\begin{question}\label[question]{question: partial flag varieties} 
    Under the Gonciulea–Lakshmibai degeneration, do projected Richardsons in partial flag varieties also degenerate to reduced unions of toric varieties?
\end{question}
Of course, if this is true, then one could also ask \Cref{question: Kim analogy to closed ball} and \Cref{question: Kim analogy of boundary} for partial flag varieties. Two of the main difficulties in extending our results to this setting are: 
\begin{enumerate}[(i)] 
    \item It is not clear what the analogy of fence diagrams is in the partial flag variety setting. 
    \item The Gelfand–Tsetlin polytope is no longer an order polytope, so it is unclear if it still has a natural unimodular triangulation to consider. 
\end{enumerate}

In analogy with pipe dreams, we ask for an analogue of \textit{mitosis}. Since mitosis is a combinatorial lift of the divided difference operator, it seems possible that an answer to the following question could give a more compact character formula for cyclic Demazure modules than the one \cite{AGH} proves. 

\begin{problem}\label[problem]{analogue of mitosis} 
Given all of the reduced fence diagrams of $\Pi_u^w$, give a combinatorial rule constructing all of the fence complexes of $\Pi_{u'}^{w'}$ such that  $\Pi_u^w\subset \Pi_{u'}^{w'}$ is a codimension one subvariety. 
\end{problem} 

Finally, while we utilize Galashin, Karp and Lam's result heavily to show that fence complexes give a presentation of $\Gr(k,n)_{\geq 0}$ as a regular CW complex, it is natural to wonder whether one could in fact use fence complexes to provide a new proof of their theorem. We therefore pose the following problem. 

\begin{problem}
    Construct an explicit homeomorphism between $\Gr(k,n)_{\geq 0}$ and $\calO(P_{k,n})$ that maps positroid cells to their corresponding fence complexes. 
\end{problem}

\appendix

\section{Weight Order Verifications for the AGH Gr{\"o}bner Basis}\label[appendix]{Weight Order Verifications for the AGH Gr{\"o}bner Basis}
\subsection{Almousa–Gao–Huang's Gröbner basis}

Knutson, Lam and Speyer prove that $S/\initial_{\revlex} (I_u^w) = SR(\pi_k(\Delta[u,w])) $ inductively, and in particular, do not provide an explicit Gröbner basis of $I_u^w$. In \cite{AGH}, Almousa, Gao and Huang provides such an explicit formulation for a Gröbner basis with respect to revlex. To describe it, we first observe that one can think about a monomial $p_{I_1}\cdots p_{I_d}\in S$ as a $k\times d$ tableau $T$, with $I_1$ in increasing order (from top to bottom) in the first column, $I_2$ in increasing order in the second column, etc. One can also think about $T$ with the columns permuted by $\sigma$, which we correspond to $\text{sign}(\sigma)p_{I_1} \cdots p_{I_d}$. Almousa, Gao and Huang give a description of the Gr{\"o}bner bases of positroid varieties in terms of the following combinatorics.

\begin{defn}
    Let $T$ be a $k \times d$ semistandard Young tableau with entry $t_{i,j}$ in cell $(i, j)$. A \textit{generalized antidiagonal $D = \{t_{i_1, j_1}, ..., t_{i_r, j_r}\}$ of $T$} is a set of elements of $T$ satisfying $t_{i_1, j_1} < \dots < t_{i_r, j_r}$, $i_1 <\dots< i_r$, and $j_1 \geq \dots \geq j_r$.
\end{defn}

\begin{figure}[H]
\centering
\def\borderInnerSep{2pt}
\begin{tikzpicture}[scale=.9]
\draw (0,0)-- (5,0)--(5,6)--(0,6)--(0,0);
\draw (0,1)--(5,1);
\draw (0,2)--(5,2);
\draw (0,3)--(5,3);
\draw (0,4)--(5,4);
\draw (0,5)--(5,5);
\draw (2,1)--(2,3);
\draw (0,1)--(2,1);
\draw (1,2)--(3,2);
\draw (1,0)--(1,6);
\draw (2,0)--(2,6);
\draw (3,0)--(3,6);
\draw (4,0)--(4,6);
\draw[maroon] (4.2, 5.2)--(4.8,5.2)--(4.8,5.8)--(4.2,5.8)--(4.2,5.2);
\draw[maroon] (3.2, 4.2)--(3.8,4.2)--(3.8,4.8)--(3.2,4.8)--(3.2,4.2);
\draw[maroon] (2.2, 2.2)--(2.8,2.2)--(2.8,2.8)--(2.2,2.8)--(2.2,2.2);
\draw[maroon] (1.2, 1.2)--(1.8,1.2)--(1.8,1.8)--(1.2,1.8)--(1.2,1.2);
\draw[maroon] (1.2, 0.2)--(1.8,0.2)--(1.8,0.8)--(1.2,0.8)--(1.2,0.2);
\def\nodescl{0.7}
\node (A) at (0.5,5.5) {1};
\node (A) at (1.5,5.5) {1};
\node (A) at (2.5,5.5) {2};
\node (A) at (3.5,5.5) {2};
\node (A) at (4.5,5.5) {3};
\node (A) at (0.5,4.5) {2};
\node (A) at (1.5,4.5) {3};
\node (A) at (2.5,4.5) {3};
\node (A) at (3.5,4.5) {4};
\node (A) at (4.5,4.5) {5};
\node (A) at (0.5,3.5) {3};
\node (A) at (1.5,3.5) {4};
\node (A) at (2.5,3.5) {4};
\node (A) at (3.5,3.5) {5};
\node (A) at (4.5,3.5) {6};
\node (A) at (0.5,2.5) {4};
\node (A) at (1.5,2.5) {5};
\node (A) at (2.5,2.5) {6};
\node (A) at (3.5,2.5) {7};
\node (A) at (4.5,2.5) {8};
\node (A) at (0.5,1.5) {5};
\node (A) at (1.5,1.5) {7};
\node (A) at (2.5,1.5) {8};
\node (A) at (3.5,1.5) {9};
\node (A) at (4.5,1.5) {9};
\node (A) at (0.5,0.5) {7};
\node (A) at (1.5,0.5) {9};
\node (A) at (2.5,0.5) {10};
\node (A) at (3.5,0.5) {11};
\node (A) at (4.5,0.5) {12};

\end{tikzpicture}\label{fig:generalized anti-diag}
\caption{A generalized antidiagonal, with terms boxed in purple, of a semistandard Young tableau.}
\end{figure}

In other words, if you read the columns of $T$ from right to left, and each column from top to bottom, a generalized antidiagonal is a sequence of strictly increasing entries in this reading order. We will abuse notation and refer to a generalized antidiagonal $D$ simply by the values of its entries, dropping the reference to their cell in $T$. That is, we will usually write $D=\{a_1<\cdots <a_r\}$, where each $a_t$ is implicitly assumed to live in a cell $(i_t,j_t)$ as in the definition above.

 Given a generalized antidiagonal $D=\{a_1<\cdots <a_r\}$ of $T$, we write $S_D$ for the symmetric group on letters $a_1,\ldots, a_r$. Then, $S_D$ naturally acts on the set of $k \times d$ tableaux with precisely the entries of $D$ in cells $(i_1, j_1), \dots, (i_r, j_r)$ and which otherwise agree with $T$. We will often identify $S_D$ with $S_r$, with $\sigma\in S_r$ acting on $a_1,\ldots, a_r$ via $\sigma\cdot a_i = a_{\sigma(i)}$. 

\begin{defn}
    Given a standard monomial $\mathbf{m} = p_{I_1} \cdots p_{I_d}$ for $Gr(k,n)$, and thinking of $I_i$ as a subset of $[n]$, we write $\mathbf{m}^\vee$ to denote the standard monomial for $Gr(n-k, n)$ given by $\mathbf{m}^\vee := p_{[n] \setminus I_1} \cdots p_{[n] \setminus I_d}$.

    Let $T$ be the tableau corresponding to $\mathbf{m}$. Then, we define $T^\vee$ to be the tableau corresponding to $\mathbf{m}^\vee$. Subsequent to a choice of integer $n$, $\vee$ induces an involutive bijection between $k \times d$ column increasing tableaux with entries at most $n$ and $(n - k) \times d$ with entries at most $n$.
\end{defn}

Recall the following fact, due to Hodge in \cite{Hodge_1943} (cf. \cite[Lemma 7.2.3]{Bruns_Herzog_1998}). 
\begin{lemma}
    The map from the set of standard monomials for $Gr(k,n)$ with $d$ factors and the set of $k \times d$ semistandard Young tableaux with fillings at most $n$ which sends $\bf{m}$ to $T$ is a bijection.
\end{lemma}

Therefore, to ease notation, we will write $T$ in place of its corresponding standard monomial in most formulas which follow. With this notation in place, we have the following explicit construction of a Gröbner basis due to \cite{AGH}. 
\begin{thm}[AGH]\label[theorem]{AGH Grobner basis}
    $\Pi_u^w$ has a Gröbner basis for $<_\mathrm{revlex}$ consisting of the straightening relations and $f_1,\ldots, f_m \in S$, where each $f_i$ is either of the form \[\sum_{w\in S_D} (-1)^{\ell(w)} w\cdot T\] or \[\left(\sum_{w\in S_r} (-1)^{\ell(w)} w\cdot T\right)^\vee.\] 
    Here, $T$ is a SSYT, $D$ is some generalized antidiagonal of length $r$ in $T$, and $S_r$ acts on the antidiagonal $D$ by replacing $a_i$ with $a_{w(i)}$. Furthermore, up to a scalar factor, $\initial_\revlex\left(\sum_{w\in S_D} (-1)^{\ell(w)} w\cdot T\right)$ equals $T$, and $\initial_\revlex\left(\sum_{w\in S_r} (-1)^{\ell(w)} w\cdot T\right)^\vee$ equals $T^\vee$. 
\end{thm}
\subsection{Proof of \Cref{so we are done} } 
Recall from \Cref{Background on Grobner Degenerations} that for  $I=\{i_1<\cdots< i_k\}$, we assign the integral weight $\lambda(I) = i_1N^{k-1}+i_2N^{k-2}+\cdots +i_{k-1}N+i_k$. In this section, we show that the Gröbner basis of \cite{AGH} described in \Cref{Background on Grobner Degenerations} satisfies the conditions of \Cref{nice Gr{\"o}bner basis implies containment}. To do this we show that the $f_i$ defined in \Cref{AGH Grobner basis} satisfy the conditions of \Cref{nice Gr{\"o}bner basis implies containment}.

We start by fixing $N>>0$, so that all of the $\initial_\lambda (f)$ we consider are actually $\initial_\rsl(f)$ (we only consider finitely many, so we can do this without loss of generality by \Cref{suffices to consider finitely many}). Let $T$ be a SSYT with generalized antidiagonal $D=a_1<\cdots <a_r$. The following lemma allows us to reduce to considering $w\cdot T$ for special $w$. 

\begin{lemma} \label[lemma]{simplify permutation} 
    Let $T$ be a $k\times d$ SSYT tableau with generalized antidiagonal $D=\{a_1<\cdots < a_r\}$, and let $w\in S_r$. Then, there exists some $w'\in S_r$ such that:
    \begin{enumerate}[(1)] 
        \item $w\cdot T$ and $w' \cdot T$ are the same tableau up to reordering columns. 
        \item $a_i$ and $a_{w'(i)}$ are not in the same column for all $i$ such that $w'(i)\neq i$. 
        \item For each column $C$, if $a_{i_1}<\cdots <a_{i_t}$ are the entries of $D$ in $C$ not fixed by $w'$, we have $w'(a_{i_1})<\cdots <w'(a_{i_t})$. 
    \end{enumerate}
\begin{proof} 
If $a_{x_1},a_{x_2}$ are in the same column and $w(x_1)=x_2$, then we can replace $w$ with the permutation $v$ such that $v(x) = w(x)$ for all $x\neq x_1,x_2$, $v(x) = w(x)$, and $v(x_1) = w(x_2)$, $v(x_2) = x_2$. In other words, we set $v= (w(x_2),x_2) \cdot w$. By construction, $v\cdot T$ is the same as $w\cdot T$ up to rearranging the order in columns.  Repeating this process for $v$, and continuing until we no longer have any elements $x$ such that $a_{v(x)}$ and $a_x$ are in the same column where $v(x)\neq x$, we may (and will) assume that $w$ satisfies (1) and (2).

Now, take a column $C$ and let $i_1,\ldots, i_t$ be the entries of $D$ in $C$ not fixed by $w$. Observe that if $\sigma$ is a permutation of $i_1,\ldots, i_t$, fixing all other elements of $[r]$, then $(w\circ \sigma)\cdot T$ equals $w\cdot T$ up to rearranging columns. In particular, we may take $\sigma_C$ to be such that if $a_{i_1}<\cdots < a_{i_t}$, then $w(\sigma_C(a_{i_1}))<\cdots <w(\sigma_C(a_{i_t}))$. In addition, by (2) we have $w(i_j)\notin \{i_1,\ldots, i_t\}$ for each $j$, and thus $w(\sigma_C(i_j))\notin \{i_1\ldots, i_t\}$.

Precomposing $w$ with the product of the $\sigma_C$, ranging over all columns $C$ with multiple entries, gives $w'$, which satisfies properties (1), (2) and (3) as desired. 
\end{proof}
\end{lemma}

We first address the case where we don't have to dualize. 

\begin{prop}\label[proposition]{non dual case} 
    The initial term of $\sum_{w\in S_D} (-1)^{\ell(w)}w\cdot T$ with respect to row sum lex is $T$. In symbols, 
    \[\initial_{\rsl}\left(\sum_{w\in S_D} (-1)^{\ell(w)}w\cdot T\right) = c T\]  for some $c\in \bbC$. 
\begin{proof}
    First of all, observe that if $w$ is some stabilizer of $T$, then by our conventions on the signs of the monomials when one permutes columns we have that $(-1)^{\ell(w)}w\cdot T$ is a positive multiple of $T$. Hence, $T$ does appear in the sum with nontrivial coefficient.

    Now, suppose that $w\in S_D$ and $w\cdot T$ is some tableau with no repeated entries in a column, distinct from $T$ up to rearranging columns. It suffices to show that $w\cdot T>T$ in row sum lex. Using the notation of \Cref{simplify permutation}, we may replace $w$ with $w'$, and assume that $w$ satisfies condition (2).
    
    Let $i$ be the smallest element of $[r]$ such that $w(i)\neq i$. It follows that $w(i)>i$. Let $i$ be in cell $(r,c)$ of $T$, and by (2), $w(i)$ is in some column $c'$ with $c'\neq c$. Note that in $wT$, when one rearranges the entries to be column standard, we have that nothing above row $r$ in column $c$ changes. This is because every entry above row $r$ in column $c$ is not affected by $w$, and all the entries below row $r$ either remain the same, or are replaced by some number larger than all entries above row $r$ (as $a_i$ is the lowest number affected by $w$, and all other entries being swapped around are some $a_j>a_i$).

    In particular, since the $(r,c)$-entry of the standardization of $wT$ must either equal $a_{w(i)}$ or something larger than $a_i$, it follows that the $r$th row sum of $wT$ is larger than the $r$th row sum of $T$. For $r'<r$, the $r'$th row sum is not affected by $wT$. This is because the only columns $d$ that will change under applying $w$ are those with an entry $a_j$ in the antidiagonal, for $i<j$. In particular, $d\leq c$, and hence the $(r',d)$ entry of $T$ is less than $a_i$, as $T$ is a SSYT. It follows that the $(r',d)$ entry does not change when applying $w$ and column standardizing. 
\end{proof}
\end{prop} 

The case where we dualize is significantly more involved. We now introduce some notation. Fix $w\in S_r$, satisfying the conditions of \Cref{simplify permutation}. 

\begin{defn}
    Let $A_T=\{i\in [r]\mid w(i)>i\}$ and $D_T = \{i\in [r]\mid w(i)<i\}$. 
\end{defn}

Given a SSYT $T$ with antidiagonal $D=\{a_1<\cdots <a_r\}$, let $c_i$ denote the column that $a_i$ is in (note we may have $c_i = c_j$ for $i\neq j$). For a nonfixed point $i$ of $w$, we define \[f(i) = \begin{cases} a_i-|\col (c_i)\cap [a_i]|+1 & \textrm{if } i\in A_T \\ a_{w(i)} - |\col(c_i)\cap [a_{w(i)}]| & \textrm{if } i\in D_T  \end{cases}.\]

Here, $\col(c)$ denotes the set of numbers in column $c$ of $T$. 
\begin{lemma}\label[lemma]{f behaves well} 
Fix a column $c$ and let $a_{y_1}<\cdots < a_{y_t}$ be entries of column $c$ not fixed by $w$. Then, once columns in $(wT)^\vee$ are reordered in increasing order from top to bottom, $f(y_1)$ is precisely the smallest $r$ such that the box in row $r$, column $c$ of $(wT)^\vee$ is different from the corresponding box in $(T)^\vee$. Moreover, if $y_1<w(y_1)$, then the entry in $(r,c)$ decreases from $T^\vee$ to $(wT)^\vee$. If $y_1>w(y_1)$, the entry in $(r,c)$ increases from $T^\vee$ to $(wT)^\vee$. 
\end{lemma}
\begin{proof}
   Observe that for a given column $c$, if $a_{y_1}<\cdots < a_{y_t}$ are entries of column $c$ not fixed by $w$, then $f(y_1)$ is precisely the smallest row $r$ of the box in column $c$ in $(wT)^\vee$ that is different from the corresponding box in $(T)^\vee$. To see this, observe by assumption (3) of \Cref{simplify permutation}, column $c$ of $(wT)^\vee$ is precisely column $c$ of $(T)^\vee$ with $a_{w(y_1)},\ldots, a_{w(y_t)}$ replaced by $a_{y_1},\ldots, a_{y_t}$ respectively, and then reordered so that the column is increasing from top to bottom. By assumption (2) of \Cref{simplify permutation}, we have that $a_{w(y_1)}$ appears in a row above $a_{w(y_2)},\ldots, a_{w(y_t)}$ in column $c$ of $(T)^\vee$, and since $a_{y_1}>a_{y_i}$ for $i>1$, only the rows affected by replacing $a_{w(y_1)}$ with $a_{y_1}$ in $T^\vee$ and then reordering are candidates for $r$.

    If $y_1<w(y_1)$, then when we replace $a_{w(y_1)}$ with $a_{y_1}$ and reorder, $a_{y_1}$ will end up in row $a_{y_1}-|\col(c_{y_1})\cap [a_{y_1}]|+1$, and all rows above this will be unchanged. Observe that in this case, the entry in row $(r,c)$ decreases from $T^\vee$ to $(wT)^\vee$.  If $w(y_1)<y_1$, then replacing $a_{w(y_1)}$ with $a_{y_1}$ will affect only the row that $a_{w(y_1)}$ is in, and rows below it. This topmost affected row is precisely $a_{w(y_1)}-|\col(c_{y_1})\cap[a_{w(y_1)}]|$. Letting $x$ be the entry directly below $a_{w(y_1)}$ in $T^\vee$, the entry in $(r,c)$ in $(wT)^\vee$ will either be $a_{y_1}$ if $a_{y_1}<x$ or $x$ if $x<a_{y_1}$. In the second case, $x>a_{w(y_1)}$ since $T^\vee$ is increasing as you go down columns. Hence, either way, the entry in $(r,c)$ increases from $T^\vee$ to $(wT)^\vee$, as desired. 
    
    \end{proof}
\begin{lemma}\label[lemma]{top entry of column has smallest f} 
    With the same notation and assumptions as \Cref{f behaves well}, $f(y_1)\leq f(y_j)$ for all $j>1$. 
\begin{proof}
    Observe that $a-|\col(c)\cap [a]|$ is precisely the number of entries in column $c$ of $T^\vee$ that are less than or equal to $a$. In particular, for $a<b$, we have $a-|\col(a)\cap [a]|\leq b-|\col(b)\cap [b]|$. We now use this observation to handle all the cases. 
    
    First, suppose $y_1\in A_T$. First assume, $y_j\in A_T$. Then, we have $a_{y_1}<a_{y_j}$, so if $y_j$ is in $A_T$, we  have $f(y_1)\leq f(y_j)$. If $y_j\in D_T$, then by assumption (2) of \Cref{simplify permutation}, we have $w(y_j)>w(y_1)>y_1$, so $a_{w(y_j)}>a_{y_1}$. Thus, \[a_{y_1}-|\col(c)\cap [a_{y_1}]| \leq a_{w(y_j)}-|\col(c)\cap [a_{w(y_j)}]|.\] 
    Moreover, this inequality is strict, since $w(y_j)$ is in column $c$ of $T^\vee$. Hence $f(y_1)\leq f(y_j)$. 

    Now, suppose instead that $y_1\in D_T$. If $y_j\in A_T$, then since $w(y_1)<y_1<y_j$, we have $f(y_1)\leq f(y_j)$. If $y_j\in D_T$, then $w(y_1)<y_1<y_j<w(y_j)$ and so $f(y_1)\leq f(y_j)$.  
\end{proof}
\end{lemma}

\begin{lemma}\label[lemma]{rel min implies f not minimum} 
    Let $w$ be of the form of \Cref{simplify permutation}. Suppose $t\in D_T$, and $w(t)\in A_T$. Then, $f(w(t))> f(t)$. 
\begin{proof}
    We must show that 
    \[a_{w(t)}-|\col(c_{w(t)})\cap [a_{w(t)}]| = f(w(t))-1\geq f(t) = a_{w(t)}-|\col(c_t)\cap[a_{w(t)}]|.\] 
    This is equivalent to showing that \[\tag{$*$} |\col(c_t)\cap [a_{w(t)}]|\geq |\col(c_{w(t)})\cap [a_{w(t)}]|.\]  Since $t\in D_T$, $t>w(t)$, so $c_{w(t)}\geq c_t$. Then, ($*$) follows from the fact that $T$ is a SSYT. 
\end{proof}
\end{lemma}

\begin{lemma}\label[lemma]{y>x>w(y) implies not minimum} 
Suppose $x\in A_T$ and $y>x>w(y)$ for some $y\in [r]$. Then, $f(x)>f(y)$. 
\begin{proof}
    Since $y<w(y)$, we have $y\in D_T$. Thus, we must show the inequality 
    \[a_x-|\col(c_x)\cap [a_x]| \geq a_{w(y)} -|\col(c_y)\cap [a_{w(y)}]|. \] 
    In other words, if $a_x$ is in row $r_x$ of $T$, and $r_y$ is the largest row in column $c_y$ of $T$ with an entry less than $a_{w(y)}$, then since $r_x = |\col(c_x)\cap [a_x]|$ and $r_y = |\col(c_y)\cap [a_{w(y)}]|$ (here we use assumption (2) of \Cref{simplify permutation} to conclude $a_{w(y)}$ is not in column $c_y$), we wish to show that 
    \[r_x-r_y\leq a_x -a_{w(y)}.\]

    Since $x>w(y)$, it follows that $a_x-a_{w(y)}$ is positive, so if $r_x\leq r_y$, we are done. Henceforth, assume that $r_x>r_y$. Then, note that if $b_1$ is the $(r_y+1, c_y)$ entry of $T$, we have by definition of $r_y$ that $b_1>a_{w(y)}$. Let $b_2$ be the $(r_y+1, c_x)$ entry of $T$. Since $D$ is a generalized antidiagonal and $x<y$, we have $c_x\geq c_y$. Then, since $T$ is a SSYT, it follows that $b_2\geq b_1>a_{w(y)}$. 

    Note that the number of entries in $\col (c_x)$ in between $a_x$ and $b_2$, inclusive, is precisely $r_x-r_y$. Of course, this is less than or equal to $a_x-b_2+1$. Since $b_2>a_{w(y)}$, it follows that 
    \[r_x-r_y\leq a_x-b_2+1\leq a_x-a_{w(y)},\] 
    as desired. 
\end{proof}
\end{lemma}

The setting of the previous lemma happens rather often, as the next result shows. 
\begin{lemma} \label[lemma]{cycle lemma} 
    For $w\in S_r$, let $x\in [r]$ satisfy $w(x)>x$. Then, either $x$ is a minimum of its cycle (i.e. for all $k\in \bbN$, $w^k(x)\geq x$), or there exists some $y\in \{w^k(x)\mid k\in \bbN\}$ such that $y>x>w(y)$. 
\begin{proof}
    Suppose $x$ is not a minimum of its cycle. Denote the cycle $(z_1,\ldots, z_m)$, where $z_1$ is the maximum of the cycle, and let $z_l$ be the minimum, i.e. $z_l\leq z_j\leq z_1$ for all $j\in [m]$. By definition of a cycle, we have $z_t = w^{t-1}(z_1)$ for $1\leq t\leq m$, and $w(z_m) = z_1$. Define the periodic function $f:\bbZ\to \bbR$ given by $f(i) = z_i$ where indices are given mod $m$. Then, $f$ has a unique continuous extension $F:[1,m]\to \bbR$ such that $F$ is linear on $[i,i+1]$ for all $i\in [m-1]$.

    Note that $x\neq z_1$ since $w(x)>x$ and $x\neq z_l$ by assumption. Then, by the Intermediate Value Theorem, since $F(1)<x<F(l)$, there exists some $\alpha \in (1,l)$ such that $F(\alpha)=x$. Since $F$ is piecewise linear and with finitely many different linear parts, none of which are constant, on $[1,l]$, there exist finitely many such $\alpha$. Suppose $\alpha$ is the largest possible, i.e. $F(\beta)\neq x$ for $\beta \in (\alpha, l)$. 
    
    If $\alpha$ is an integer, then $F(\alpha+1)<F(\alpha)$, otherwise either $\alpha+1=l$, or  $F(\alpha+1)>x>F(l)$, and the Intermediate Value Theorem contradicts our choice of $\alpha$. In either case, $w(x)<x$, a contradiction. Hence, $\alpha$ is not an integer. Thus, $\alpha \in (i,i+1)$ where $1\leq i\leq i+1\leq l$. If $F$ is increasing on $(i,i+1)$, then $i+1<l$, and $F(i+1)>x> F(l)$, and once again the Intermediate Value Theorem contradicts our choice of $\alpha$. Hence, $F$ is decreasing on $(i,i+1)$, so $F(i)>x>F(i+1)$. In particular, $z_i>x>z_{i+1}= w(z_i)$, as desired. 
    
\end{proof}
\end{lemma}

\begin{lemma}\label[lemma]{dual case} 
    Let $T$ be a $k\times d$ SSYT. Then, for all $w$ such that $wT$ and $T$ are not equal as tableaux (up to permutation of columns), we have that $(wT)^\vee > T^\vee$ in row sum lex. In particular, 
    \[\mathrm{in}_{\rsl}\left(\sum_{w\in S_r} (wT)^\vee \right) \] is a scalar multiple of $T^\vee$. 
\begin{proof}
    We may choose $w$ to be in the form from \Cref{simplify permutation}. Let $r$ be the smallest row such that $(wT)^\vee$ and $T^\vee$ differ, let $c$ be a column, and $a_{x_1}<\ldots< a_{x_m}$ the entries of $D$ in $c$ that are affected by $w$ (i.e. $w(x_i)\neq x_i$ for each $i$). Suppose $x_1\in A_T$. In this case, we show that $f(x_1)>r$, and hence the only columns for which row $r$ changes from $T^\vee$ to $(wT)^\vee$ are those where the row $r$ entry increases. Take $c$ to have minimal $f(x_1)$ out of all such columns with $x_1\in A_T$. Now, we have two cases:
\begin{itemize}
    \item \textbf{Case 1}: $x_1$ is the minimum of its cycle in $w$. If the cycle is of length $2$, then $w^{-1}(x_1)>x_1$, so $w^{-1}(x_1)\in D_T$, and hence by \Cref{rel min implies f not minimum}, we have $f(w^{-1}(x_1))<f(x_1)$. In particular, if $a_s$ is the smallest entry of $c_{w^{-1}(x_1)}$ in $D$ such that $w(s)\neq s$, we have $f(s)\leq f(w^{-1}(x_1))<f(x_1)$ by \Cref{top entry of column has smallest f}, and since $r\leq f(s)$ by \Cref{f behaves well}, it follows that $f(x_1)>r$.

   \item  \textbf{Case 2}: $x_1$ is not the minimum of its cycle in $w$. Then, by \Cref{cycle lemma}, there exists some $y\in \{w^k(x_1):k\in \bbN\}$ such that $y>x_1>w(y)$. In this case, \Cref{y>x>w(y) implies not minimum} implies that $f(x_1)>f(y)$ for some other $a_y$. Note by \Cref{top entry of column has smallest f}, $a_y$ and $a_{x_1}$ are not in the same column. Let $a_z$ be the smallest entry of $D$ in column $c_y$ such that $w(z)\neq z$. Then, by \Cref{top entry of column has smallest f} again, $f(z)\leq f(y) <f(x_1)$. In particular, the smallest row changed in column $c_y$ from $T^\vee$ to $(wT)^\vee$ is less than the smallest row changed in column $c$ from $T^\vee$ to $(wT)^\vee$, so $f(x_1)>r$, as desired.
\end{itemize} 

    For the final statement of the lemma, the same argument as in \Cref{non dual case} implies that $T^\vee$ appears in $\sum_{w\in S_r} (-1)^{\ell(w)}(wT)^\vee $ with nonzero coefficient. The result now follows from the first statement of the lemma. 
\end{proof}
\end{lemma} 

\begin{proof}[Proof of \Cref{so we are done}.]
By \Cref{AGH Grobner basis} and our computations above, it follows that $\initial_\revlex(\initial_\rsl(f_i)) = \initial_\revlex(f_i)$ for all $f_i$ in the AGH Gröbner basis. Then, \Cref{non dual case} and \Cref{dual case} show that the AGH Gröbner basis satisfies our desired criteria. 
\end{proof}

\nocite{*} 

\bibliographystyle{plainurl}
\bibliography{refs}

@book {beck2007computing,
    AUTHOR = {Beck, Matthias and Robins, Sinai},
     TITLE = {Computing the continuous discretely},
    SERIES = {Undergraduate Texts in Mathematics},
   EDITION = {Second},
      NOTE = {Integer-point enumeration in polyhedra,
              With illustrations by David Austin},
 PUBLISHER = {Springer, New York},
      YEAR = {2015},
     PAGES = {xx+285},
      ISBN = {978-1-4939-2968-9; 978-1-4939-2969-6},
   MRCLASS = {11P21 (05A15 05B15 11-02 11H06 52B05 52B20)},
  MRNUMBER = {3410115},
       DOI = {10.1007/978-1-4939-2969-6},
       URL = {https://doi.org/10.1007/978-1-4939-2969-6},
}

@article{KLS1,
  title={Positroid Varieties: Juggling and Geometry},
  author={Knutson, Allen and Lam, Thomas and Speyer, David E},
  journal={Compositio Mathematica},
  volume={149},
  number={10},
  pages={1710--1752},
  year={2013},
  publisher={London Mathematical Society}
}

@book {CLS,
    AUTHOR = {Cox, David A. and Little, John B. and Schenck, Henry K.},
     TITLE = {Toric varieties},
    SERIES = {Graduate Studies in Mathematics},
    VOLUME = {124},
 PUBLISHER = {American Mathematical Society, Providence, RI},
      YEAR = {2011},
     PAGES = {xxiv+841},
      ISBN = {978-0-8218-4819-7},
   MRCLASS = {14M25 (05A15 05E45 52B12)},
  MRNUMBER = {2810322},
MRREVIEWER = {Ivan\ Arzhantsev},
       DOI = {10.1090/gsm/124},
       URL = {https://doi.org/10.1090/gsm/124},
}

@article{Lam, title={Cyclic {D}emazure Modules and Positroid Varieties}, volume={26}, url={https://www.combinatorics.org/ojs/index.php/eljc/article/view/v26i2p28}, DOI={10.37236/8383}, abstractNote={A positroid variety is an intersection of cyclically rotated Grassmannian Schubert varieties.  Each graded piece of the homogeneous coordinate ring of a positroid variety is the intersection of cyclically rotated (rectangular) Demazure modules, which we call the cyclic Demazure module.  In this note, we show that the cyclic Demazure module has a canonical basis, and define the cyclic Demazure crystal.}, number={2}, journal={The Electronic Journal of Combinatorics}, author={Lam, Thomas}, year={2019}, month={May}, pages={pp.2-28}}

@phdthesis{Kim,
    author = {Kim, Giwan} ,
    title = {Richardson Varieties in a Toric Degeneration of the Flag Variety},
    school = {The University of Michigan},
    year = {2015}
}

@article{Kogan-Miller,
title = {Toric degeneration of {S}chubert varieties and {G}elfand–{T}setlin polytopes},
journal = {Advances in Mathematics},
volume = {193},
number = {1},
pages = {1-17},
year = {2005},
issn = {0001-8708},
doi = {https://doi.org/10.1016/j.aim.2004.03.017},
url = {https://www.sciencedirect.com/science/article/pii/S000187080400129X},
author = {Mikhail Kogan and Ezra Miller}
}

@article{Darayon,
title = {Arithmetically {G}orenstein and
{G}orenstein {R}ichardson Varieties in
the {G}rassmannian},
journal = {(preprint)},
year = {2012},
author = {Darayon, Chayapa},
}

@book {Brion-Kumar,
    AUTHOR = {Brion, Michel and Kumar, Shrawan},
     TITLE = {Frobenius splitting methods in geometry and representation
              theory},
    SERIES = {Progress in Mathematics},
    VOLUME = {231},
 PUBLISHER = {Birkh\"auser Boston, Inc., Boston, MA},
      YEAR = {2005},
     PAGES = {x+250},
      ISBN = {0-8176-4191-2},
   MRCLASS = {14M15 (13A35 14C05 17B10 20G05)},
  MRNUMBER = {2107324},
MRREVIEWER = {Vikram\ B.\ Mehta},
}

@article {Lenart,
    AUTHOR = {Lenart, Cristian},
     TITLE = {A {$K$}-theory version of {M}onk's formula and some related
              multiplication formulas},
   JOURNAL = {J. Pure Appl. Algebra},
  FJOURNAL = {Journal of Pure and Applied Algebra},
    VOLUME = {179},
      YEAR = {2003},
    NUMBER = {1-2},
     PAGES = {137--158},
      ISSN = {0022-4049,1873-1376},
   MRCLASS = {14M15 (05E15 19L64)},
  MRNUMBER = {1958380},
MRREVIEWER = {Frank\ Sottile},
       DOI = {10.1016/S0022-4049(02)00208-6},
       URL = {https://doi.org/10.1016/S0022-4049(02)00208-6},
}

@article {Payne,
    AUTHOR = {Payne, Sam},
     TITLE = {Frobenius splittings of toric varieties},
   JOURNAL = {Algebra Number Theory},
  FJOURNAL = {Algebra \& Number Theory},
    VOLUME = {3},
      YEAR = {2009},
    NUMBER = {1},
     PAGES = {107--119},
      ISSN = {1937-0652,1944-7833},
   MRCLASS = {14M25 (13A35)},
  MRNUMBER = {2491910},
MRREVIEWER = {Magda\ Sebestean},
       DOI = {10.2140/ant.2009.3.107},
       URL = {https://doi.org/10.2140/ant.2009.3.107},
}

@book {Hartshorne,
    AUTHOR = {Hartshorne, Robin},
     TITLE = {Algebraic geometry},
    SERIES = {Graduate Texts in Mathematics},
    VOLUME = {No. 52},
 PUBLISHER = {Springer-Verlag, New York-Heidelberg},
      YEAR = {1977},
     PAGES = {xvi+496},
      ISBN = {0-387-90244-9},
   MRCLASS = {14-01},
  MRNUMBER = {463157},
MRREVIEWER = {Robert\ Speiser},
}

@book {Beck-Sanyal,
    AUTHOR = {Beck, Matthias and Sanyal, Raman},
     TITLE = {Combinatorial reciprocity theorems},
    SERIES = {Graduate Studies in Mathematics},
    VOLUME = {195},
      NOTE = {An invitation to enumerative geometric combinatorics},
 PUBLISHER = {American Mathematical Society, Providence, RI},
      YEAR = {2018},
     PAGES = {xiv+308},
      ISBN = {978-1-4704-2200-4},
   MRCLASS = {05-01 (05A15 05C30 06A07 11P21 52B05 52C35)},
  MRNUMBER = {3839322},
MRREVIEWER = {Philippe\ Nadeau},
       DOI = {10.1090/gsm/195},
       URL = {https://doi.org/10.1090/gsm/195},
}

@misc{KLS0,
      title={Positroid varieties {I}: juggling and geometry}, 
      author={Allen Knutson and Thomas Lam and David E Speyer},
      year={2009},
      eprint={0903.3694},
      archivePrefix={arXiv},
      primaryClass={math.AG},
      url={https://arxiv.org/abs/0903.3694}, 
}

@article{KLS2,
  title={Projections of {R}ichardson  Varieties},
  author={Knutson, Allen and Lam, Thomas and Speyer, David E},
  journal={Journal f{\"u}r die reine und angewandte Mathematik (Crelles Journal)},
  volume={2014},
  number={687},
  pages={133--157},
  year={2012},
  publisher={Walter de Gruyter GmbH}
}

@misc{speyer2024richardsonvarietiesprojectedrichardson,
      title={Richardson varieties, projected {R}ichardson varieties and positroid varieties}, 
      author={David E Speyer},
      year={2024},
      eprint={2303.04831},
      archivePrefix={arXiv},
      primaryClass={math.AG},
      url={https://arxiv.org/abs/2303.04831}, 
}

@book {EC1,
    AUTHOR = {Stanley, Richard P.},
     TITLE = {Enumerative Combinatorics. {V}olume 1},
    SERIES = {Cambridge Studies in Advanced Mathematics},
    VOLUME = {49},
   EDITION = {Second},
 PUBLISHER = {Cambridge University Press, Cambridge},
      YEAR = {2012},
     PAGES = {xiv+626},
      ISBN = {978-1-107-60262-5},
   MRCLASS = {05-02 (05A15 06-02)},
  MRNUMBER = {2868112},
}

@book {FultonAnderson,
    AUTHOR = {Anderson, David and Fulton, William},
     TITLE = {Equivariant cohomology in algebraic geometry},
    SERIES = {Cambridge Studies in Advanced Mathematics},
    VOLUME = {210},
 PUBLISHER = {Cambridge University Press, Cambridge},
      YEAR = {2024},
     PAGES = {xv+446},
      ISBN = {978-1-00-934998-7},
   MRCLASS = {14L30 (05E14 14-02 14F43 14Mxx 20G05 55N91)},
  MRNUMBER = {4655919},
MRREVIEWER = {Michael\ Orin\ Joyce},
}

@misc{AGH,
      title={Standard Monomials for Positroid Varieties}, 
      author={Ayah Almousa and Shiliang Gao and Daoji Huang},
      year={2024},
      eprint={2309.15384},
      archivePrefix={arXiv},
      primaryClass={math.AG},
      url={https://arxiv.org/abs/2309.15384}, 
}

@article {Galashin-Karp-Lam,
    AUTHOR = {Galashin, Pavel and Karp, Steven N. and Lam, Thomas},
     TITLE = {Regularity theorem for totally nonnegative flag varieties},
   JOURNAL = {J. Amer. Math. Soc.},
  FJOURNAL = {Journal of the American Mathematical Society},
    VOLUME = {35},
      YEAR = {2022},
    NUMBER = {2},
     PAGES = {513--579},
      ISSN = {0894-0347,1088-6834},
   MRCLASS = {14M15 (05E14 15B48 20G20)},
  MRNUMBER = {4374956},
MRREVIEWER = {Jorge\ A.\ Vargas},
       DOI = {10.1090/jams/983},
       URL = {https://doi.org/10.1090/jams/983},
}

@article {Rietsch-Williams,
    AUTHOR = {Rietsch, Konstanze and Williams, Lauren},
     TITLE = {Discrete {M}orse theory for totally non-negative flag
              varieties},
   JOURNAL = {Adv. Math.},
  FJOURNAL = {Advances in Mathematics},
    VOLUME = {223},
      YEAR = {2010},
    NUMBER = {6},
     PAGES = {1855--1884},
      ISSN = {0001-8708,1090-2082},
   MRCLASS = {57T99 (14M15 20G20 52B22)},
  MRNUMBER = {2601003},
MRREVIEWER = {Julianna\ Tymoczko},
       DOI = {10.1016/j.aim.2009.10.011},
       URL = {https://doi.org/10.1016/j.aim.2009.10.011},
}

@book{Bruns_Herzog_1998, place={Cambridge}, edition={2}, series={Cambridge Studies in Advanced Mathematics}, title={{C}ohen-{M}acaulay Rings}, publisher={Cambridge University Press}, author={Bruns, Winfried and Herzog, H. Jürgen}, year={1998}, collection={Cambridge Studies in Advanced Mathematics}}

@article{Hodge_1943, title={Some Enumerative Results in the Theory of Forms}, volume={39}, DOI={10.1017/S0305004100017631}, number={1}, journal={Mathematical Proceedings of the Cambridge Philosophical Society}, author={Hodge, W. V. D.}, year={1943}, pages={22–30}}

@article {Eisenbud-Sturmfels,
    AUTHOR = {Eisenbud, David and Sturmfels, Bernd},
     TITLE = {Binomial ideals},
   JOURNAL = {Duke Math. J.},
  FJOURNAL = {Duke Mathematical Journal},
    VOLUME = {84},
      YEAR = {1996},
    NUMBER = {1},
     PAGES = {1--45},
      ISSN = {0012-7094,1547-7398},
   MRCLASS = {13P10 (13A30 14M25)},
  MRNUMBER = {1394747},
MRREVIEWER = {P.\ Schenzel},
       DOI = {10.1215/S0012-7094-96-08401-X},
       URL = {https://doi.org/10.1215/S0012-7094-96-08401-X},
}

@book{EisenbudCAbook,
series = {Graduate texts in mathematics ; 150},
publisher = {Springer-Verlag},
isbn = {0387942696},
year = {1995},
title = {Commutative algebra with a view toward algebraic geometry},
language = {eng},
address = {New York},
author = {Eisenbud, David},
}

@book {FultonToricVarieties,
    AUTHOR = {Fulton, William},
     TITLE = {Introduction to toric varieties},
    SERIES = {Annals of Mathematics Studies},
    VOLUME = {131},
      NOTE = {The William H. Roever Lectures in Geometry},
 PUBLISHER = {Princeton University Press, Princeton, NJ},
      YEAR = {1993},
     PAGES = {xii+157},
      ISBN = {0-691-00049-2},
   MRCLASS = {14M25 (14-02 14J30)},
  MRNUMBER = {1234037},
MRREVIEWER = {T.\ Oda},
       DOI = {10.1515/9781400882526},
       URL = {https://doi.org/10.1515/9781400882526},
}

@article {Bergeron-Sottile,
    AUTHOR = {Bergeron, Nantel and Sottile, Frank},
     TITLE = {Schubert polynomials, the {B}ruhat order, and the geometry of
              flag manifolds},
   JOURNAL = {Duke Math. J.},
  FJOURNAL = {Duke Mathematical Journal},
    VOLUME = {95},
      YEAR = {1998},
    NUMBER = {2},
     PAGES = {373--423},
      ISSN = {0012-7094,1547-7398},
   MRCLASS = {05E15 (05E05 14M15 14N15)},
  MRNUMBER = {1652021},
MRREVIEWER = {Witold\ Kra\'skiewicz},
       DOI = {10.1215/S0012-7094-98-09511-4},
       URL = {https://doi.org/10.1215/S0012-7094-98-09511-4},
}

@article {Sturmfels-White,
    AUTHOR = {Sturmfels, Bernd and White, Neil},
     TITLE = {Gr\"obner bases and invariant theory},
   JOURNAL = {Adv. Math.},
  FJOURNAL = {Advances in Mathematics},
    VOLUME = {76},
      YEAR = {1989},
    NUMBER = {2},
     PAGES = {245--259},
      ISSN = {0001-8708,1090-2082},
   MRCLASS = {13-04 (15A75)},
  MRNUMBER = {1013672},
MRREVIEWER = {Ralf\ Fr\"oberg},
       DOI = {10.1016/0001-8708(89)90053-4},
       URL = {https://doi.org/10.1016/0001-8708(89)90053-4},
}

@incollection {Hibi,
    AUTHOR = {Hibi, Takayuki},
     TITLE = {Distributive lattices, affine semigroup rings and algebras
              with straightening laws},
 BOOKTITLE = {Commutative algebra and combinatorics ({K}yoto, 1985)},
    SERIES = {Adv. Stud. Pure Math.},
    VOLUME = {11},
     PAGES = {93--109},
 PUBLISHER = {North-Holland, Amsterdam},
      YEAR = {1987},
      ISBN = {0-444-70314-4},
   MRCLASS = {13H10 (06D99)},
  MRNUMBER = {951198},
MRREVIEWER = {Keiichi\ Watanabe},
       DOI = {10.2969/aspm/01110093},
       URL = {https://doi.org/10.2969/aspm/01110093},
}

@book {Bjorner-Brenti,
    AUTHOR = {Bj\"orner, Anders and Brenti, Francesco},
     TITLE = {Combinatorics of {C}oxeter groups},
    SERIES = {Graduate Texts in Mathematics},
    VOLUME = {231},
 PUBLISHER = {Springer, New York},
      YEAR = {2005},
     PAGES = {xiv+363},
      ISBN = {978-3540-442387; 3-540-44238-3},
   MRCLASS = {05-01 (05E15 20F55)},
  MRNUMBER = {2133266},
MRREVIEWER = {Jian-yi\ Shi},
}

@article {deConcini-Eisenbud-Procesi,
    AUTHOR = {de Concini, C. and Eisenbud, David and Procesi, C.},
     TITLE = {Young diagrams and determinantal varieties},
   JOURNAL = {Invent. Math.},
  FJOURNAL = {Inventiones Mathematicae},
    VOLUME = {56},
      YEAR = {1980},
    NUMBER = {2},
     PAGES = {129--165},
      ISSN = {0020-9910,1432-1297},
   MRCLASS = {14M12 (14L30 15A72 20C30)},
  MRNUMBER = {558865},
MRREVIEWER = {Vladimir\ L.\ Popov},
       DOI = {10.1007/BF01392548},
       URL = {https://doi.org/10.1007/BF01392548},
}

@article {Gonciulea-Lakshmibai, 
    AUTHOR = {Gonciulea, N. and Lakshmibai, V.},
     TITLE = {Degenerations of flag and {S}chubert varieties to toric
              varieties},
   JOURNAL = {Transform. Groups},
  FJOURNAL = {Transformation Groups},
    VOLUME = {1},
      YEAR = {1996},
    NUMBER = {3},
     PAGES = {215--248},
      ISSN = {1083-4362,1531-586X},
   MRCLASS = {14M15 (14M25)},
  MRNUMBER = {1417711},
MRREVIEWER = {E.\ Aky\i ld\i z},
       DOI = {10.1007/BF02549207},
       URL = {https://doi.org/10.1007/BF02549207},
}

@book {Sturmfels-degen,
    AUTHOR = {Sturmfels, Bernd},
     TITLE = {Gr\"obner bases and convex polytopes},
    SERIES = {University Lecture Series},
    VOLUME = {8},
 PUBLISHER = {American Mathematical Society, Providence, RI},
      YEAR = {1996},
     PAGES = {xii+162},
      ISBN = {0-8218-0487-1},
   MRCLASS = {13P10 (14M25 52B20)},
  MRNUMBER = {1363949},
MRREVIEWER = {P.\ Schenzel},
       DOI = {10.1090/ulect/008},
       URL = {https://doi.org/10.1090/ulect/008},
}

@incollection {Francisco2014ASO,
    AUTHOR = {Francisco, Christopher A. and Mermin, Jeffrey and Schweig,
              Jay},
     TITLE = {A survey of {S}tanley-{R}eisner theory},
 BOOKTITLE = {Connections between algebra, combinatorics, and geometry},
    SERIES = {Springer Proc. Math. Stat.},
    VOLUME = {76},
     PAGES = {209--234},
 PUBLISHER = {Springer, New York},
      YEAR = {2014},
      ISBN = {978-1-4939-0625-3; 978-1-4939-0626-0},
   MRCLASS = {13F55 (05E45)},
  MRNUMBER = {3213521},
MRREVIEWER = {Satoshi\ Murai},
       DOI = {10.1007/978-1-4939-0626-0\_5},
       URL = {https://doi.org/10.1007/978-1-4939-0626-0_5},
}

@incollection {Alexandersson-Alhajjar,
    AUTHOR = {Alexandersson, Per and Alhajjar, Elie},
     TITLE = {Ehrhart positivity and {D}emazure characters},
 BOOKTITLE = {Algebraic and geometric combinatorics on lattice polytopes},
     PAGES = {56--71},
 PUBLISHER = {World Sci. Publ., Hackensack, NJ},
      YEAR = {2019},
      ISBN = {978-981-120-047-2},
   MRCLASS = {52B20},
  MRNUMBER = {3971685},
}

@article{STANLEY197857,
title = {Hilbert functions of graded algebras},
journal = {Advances in Mathematics},
volume = {28},
number = {1},
pages = {57-83},
year = {1978},
issn = {0001-8708},
doi = {https://doi.org/10.1016/0001-8708(78)90045-2},
url = {https://www.sciencedirect.com/science/article/pii/0001870878900452},
author = {Richard P Stanley}
}

@misc{elizalde2025symmetryascentdescentdistributions,
      title={Symmetry of ascent and descent distributions on rectangular and staircase tableaux}, 
      author={Sergi Elizalde},
      year={2025},
      eprint={2501.07573},
      archivePrefix={arXiv},
      primaryClass={math.CO},
      url={https://arxiv.org/abs/2501.07573}, 
}

@misc{postnikovpositroids, 
title ={Total positivity, {G}rassmannians, and networks}, 
author = {Alexander Postnikov},
url = { https://math.mit.edu/∼apost/papers/tpgrass.pdf} 
}

@article {Speyer-Williams-Postnikov,
    AUTHOR = {Postnikov, Alexander and Speyer, David and Williams, Lauren},
     TITLE = {Matching polytopes, toric geometry, and the totally
              non-negative {G}rassmannian},
   JOURNAL = {J. Algebraic Combin.},
  FJOURNAL = {Journal of Algebraic Combinatorics. An International Journal},
    VOLUME = {30},
      YEAR = {2009},
    NUMBER = {2},
     PAGES = {173--191},
      ISSN = {0925-9899,1572-9192},
   MRCLASS = {20G20 (05B35 13F60 14M25 52B70)},
  MRNUMBER = {2525057},
MRREVIEWER = {T.\ Oda},
       DOI = {10.1007/s10801-008-0160-1},
       URL = {https://doi.org/10.1007/s10801-008-0160-1},
}

@article{REINER2005247,
title = {On the {C}harney–{D}avis and {N}eggers–{S}tanley conjectures},
journal = {Journal of Combinatorial Theory, Series A},
volume = {109},
number = {2},
pages = {247-280},
year = {2005},
issn = {0097-3165},
doi = {https://doi.org/10.1016/j.jcta.2004.09.003},
url = {https://www.sciencedirect.com/science/article/pii/S0097316504001360},
author = {Victor Reiner and Volkmar Welker}
}

@article {Stanley-twoposetpolytopes,
    AUTHOR = {Stanley, Richard P.},
     TITLE = {Two poset polytopes},
   JOURNAL = {Discrete Comput. Geom.},
  FJOURNAL = {Discrete \& Computational Geometry. An International Journal
              of Mathematics and Computer Science},
    VOLUME = {1},
      YEAR = {1986},
    NUMBER = {1},
     PAGES = {9--23},
      ISSN = {0179-5376,1432-0444},
   MRCLASS = {52A25 (52A40)},
  MRNUMBER = {824105},
MRREVIEWER = {P.\ McMullen},
       DOI = {10.1007/BF02187680},
       URL = {https://doi.org/10.1007/BF02187680},
}

@book {bruns2009polytopes,
    AUTHOR = {Bruns, Winfried and Gubeladze, Joseph},
     TITLE = {Polytopes, rings, and {$K$}-theory},
    SERIES = {Springer Monographs in Mathematics},
 PUBLISHER = {Springer, Dordrecht},
      YEAR = {2009},
     PAGES = {xiv+461},
      ISBN = {978-0-387-76355-2},
   MRCLASS = {19-02 (11H06 13F45 14M25 52-01 52B20)},
  MRNUMBER = {2508056},
MRREVIEWER = {T.\ Oda},
       DOI = {10.1007/b105283},
       URL = {https://doi.org/10.1007/b105283},
}

\end{document}